\documentclass[11pt]{amsart}

\usepackage{amssymb}
\usepackage{pb-diagram,pb-xy} 
\usepackage[all,cmtip,arrow]{xy} 
\usepackage[margin=1in]{geometry}  
\usepackage{hyperref} 

\numberwithin{equation}{section}

\theoremstyle{plain}
\newtheorem{theorem}[equation]{Theorem}

\newtheorem{lemma}[equation]{Lemma}

\newtheorem{corollary}[equation]{Corollary}
\newtheorem{proposition}[equation]{Proposition}

\newtheorem{prop}[equation]{Proposition}

\newtheorem*{claim*}{Claim}

\theoremstyle{definition}
\newtheorem{definition}[equation]{Definition}

\newtheorem{notation}[equation]{Notation}

\newtheorem{remark}[equation]{Remark}
\newtheorem*{remark*}{Remark}

\newtheorem*{remarks*}{Remarks}
\newtheorem{example}[equation]{Example}

\newtheorem{hypothesis}[equation]{Hypothesis}

\newcommand{\isom}{\cong}                       
\newcommand{\homeq}{\simeq}                     
\newcommand{\smsh}{\wedge}                      
\newcommand{\binBox}{\mathbin{\Box}}

\newcommand{\Smsh}{\bigwedge}                   
\newcommand{\Wdge}{\bigvee}                     

\newcommand{\R}{\mathbb{R}}                     

\newcommand{\cat}[1]{\mathcal{#1}}              
\newcommand{\Cfin}{\mathsf{C^{fin}}}
\newcommand{\C}{\mathsf{C}}
\newcommand{\D}{\mathsf{D}}

\DeclareMathOperator*{\hocolim}{hocolim}
\DeclareMathOperator*{\colim}{colim}
\DeclareMathOperator*{\holim}{holim}
\newcommand{\Map}{\operatorname{Map} }
\newcommand{\Nat}{\operatorname{Nat} }
\newcommand{\hNat}{\widetilde{\operatorname{Nat}} }
\newcommand{\creff}{\operatorname{cr} }
\newcommand{\Hom}{\operatorname{Hom} }

\newcommand{\Tot}{\operatorname{Tot} }

\newcommand{\Tate}{\operatorname{Tate}}
\newcommand{\hHom}{\widetilde{\operatorname{Hom}} }
\newcommand{\Deltainj}{\Delta_{\mathsf{inj}}}
\newcommand{\hM}{\tilde{\cat{M}}}

\DeclareMathOperator*{\hofib}{hofib}

\DeclareMathOperator*{\thocofib}{thocofib}
\DeclareMathOperator*{\tcof}{tcof}
\DeclareMathOperator*{\Sing}{Sing}

\newcommand{\un}[1]{\underline{#1}}

\newcommand{\sset}{\mathsf{sSet}}            
\newcommand{\based}{\mathsf{Top}_*}
\newcommand{\spectra}{{\mathsf{Sp}}}              
\newcommand{\finspec}{{\mathsf{Sp^{fin}}}}              
\newcommand{\finbased}{{\mathsf{Top^{fin}_*}}}
\newcommand{\symseq}{\mathsf{[\Sigma,Sp]}}

\newcommand{\cTop}{\mathsf{cTop}}
\newcommand{\Top}{\mathsf{Top}}

\newcommand{\weq}{\; \tilde{\longrightarrow} \;}      
\newcommand{\lweq}{\; \tilde{\longleftarrow} \;}      

\newcommand{\epi}{\twoheadrightarrow}           
\newcommand{\into}{\hookrightarrow}

\newcommand{\dual}{\mathbb{D}}                  
\newcommand{\der}{\partial}                     

\newcommand{\ord}[1]{$#1$\textsuperscript{th}}

\begin{document}

\title[A classification of Taylor towers]{A classification of Taylor towers of functors of spaces and spectra}

\author{Gregory Arone}
\author{Michael Ching}

\begin{abstract}
We describe new structure on the Goodwillie derivatives of a functor, and we show how the full Taylor tower of the functor can be recovered from this structure. This new structure takes the form of a coalgebra over a certain comonad which we construct, and whose precise nature depends on the source and target categories of the functor in question. The Taylor tower can be recovered from standard cosimplicial cobar constructions on the coalgebra formed by the derivatives. We get from this an equivalence between the homotopy category of polynomial functors and that of bounded coalgebras over this comonad.

For functors with values in the category of spectra, we give a rather explicit description of the associated comonads and their coalgebras. In particular, for functors from based spaces to spectra we interpret this new structure as that of a divided power right module over the operad formed by the derivatives of the identity on based spaces.
\end{abstract}

\maketitle
\thispagestyle{empty}

Goodwillie's calculus of homotopy functors, developed in \cite{goodwillie:1990,goodwillie:1991,goodwillie:2003}, provides a systematic filtration of any functor (between sufficiently nice model categories) that preserves weak equivalences. If $F: \C \to \D$ is such a functor, then there is a `Taylor tower'
\[ F \to \dots \to P_nF \to P_{n-1}F \to \dots \to P_0F \]
where $F \to P_nF$ is the universal natural transformation from $F$ to an `$n$-excisive' functor. For nice $F$, and for sufficiently highly connected $X \in \C$,  the tower `converges' in the sense that there is a weak equivalence
\[ F(X) \weq \holim_n P_nF(X). \]
The terms $P_nF$ are in general difficult to understand, but Goodwillie gave in \cite{goodwillie:2003} simple descriptions of the `homogeneous layers' of the Taylor tower, that is the fibres
\[ D_nF := \hofib(P_nF \to P_{n-1}F). \]
For a functor $F: \based \to \based$ of based topological spaces, and a finite CW complex $X$, we have
\[ D_nF(X) \homeq \Omega^\infty (\der_nF \smsh (\Sigma^\infty X)^{\smsh n})_{h\Sigma_n} \]
where $\der_nF$ is a spectrum with action of the symmetric group $\Sigma_n$. We think of $\der_nF$ as the `\ord{n} Taylor coefficient' of the functor $F$ expanded around the one-point space $*$, or the `\ord{n} derivative' of $F$ at $*$. Since we only consider Taylor expansions at $*$ we refer to $\der_nF$ just as the `\ord{n} derivative of $F$'.

We use the notation
\[
\der_*F:=(\der_1F, \der_2F, \ldots, \der_nF, \ldots)
\]
for the sequence of derivatives of $F$ together with their symmetric group actions. Thus $\der_*F$ is a symmetric sequence of spectra. Goodwillie's theorem says that the symmetric sequence $\der_*F$ determines the layers in the Taylor tower of $F$.

Our goal in the present paper is to describe additional structure on $\der_*F$, which is sufficient to recover the entire Taylor tower of $F$, rather than just the homogeneous layers. Part of the additional structure was previously studied by the authors in \cite{arone/ching:2011}, where it was shown that $\der_*F$ has the structure of a module over a certain operad $\der_* I$ (namely that formed by the derivatives of the identity on $\based$). The precise type of module structure (left, right or bi-) depends on the source and target categories of $F$.

It is clear however that this module structure does not tell the full story. For example, if $F$ is a functor from and to the category of spectra, then the methods of~\cite{arone/ching:2011} do not endow $\der_*F$ with any additional structure. It is easy to see, though, that the symmetric sequence $\der_*F$ alone does not determine the Taylor tower of $F$, even in this case.

To obtain further structure we use a homotopic version of descent theory as in, for example, \cite{hess:2010}. By this we mean the following. Let $\C$ and $\D$ each stand for either the category $\based$ (of based topological spaces) or $\spectra$ (of spectra). Let $[\C, \D]$ be the category of pointed, finitary (i.e., preserving filtered homotopy colimits) simplicial functors from $\C$ to $\D$. Taking Goodwillie derivatives can be thought of as a functor
\[
\der_*\colon [\C, \D] \longrightarrow \cat{M}
\]
where $\cat{M}$ is either a category of modules over $\der_*I$, or the category of symmetric sequences, depending on $\C$ and $\D$. The following is one of the key observations of this paper: the functor $\der_*$ has a right adjoint, which we denote as
\[
\Phi\colon \cat{M}\to [\C, \D].
\]

It follows in a standard way that the composite $\der_* \Phi$ is a
comonad on $\cat{M}$ and that for any $F \in [\C, \D]$,
$\der_*F$ is a $\der_* \Phi$-coalgebra. We can thus view $\der_*$ as a
functor
\[
\der_*: [\C,\D] \to \der_* \Phi\mbox{-coalg.}
\]
The main result of this paper then says that the Taylor tower of $F$
can be recovered from the coalgebra $\der_*F$ via a cobar
construction. More precisely we have the following.

\begin{theorem} \label{thm:intro1}
For $F \in [\mathsf{C},\mathsf{D}]$ as above, we have
\[ P_nF \homeq \operatorname{cobar}(\Phi,\der_*\Phi,\der_{\leq n}F) \]
where $\der_{\leq n}$ denotes the truncated $\der_* \Phi$-coalgebra
consisting only of the first $n$ derivatives of $F$.

Moreover, if the Taylor tower of $F$ converges then
\[ F \homeq \operatorname{cobar}(\Phi,\der_*\Phi,\der_*F). \]
\end{theorem}
These results are proved in Theorem~\ref{thm:main} and Corollary~\ref{cor:truncated}.

We say that a symmetric sequence is `$N$-truncated' if all its terms above the \ord{N} are contractible. We say that it is `bounded' if it is $N$-truncated for some $N$. It turns
out that every bounded $\der_*\Phi$-coalgebra arises, up to
homotopy, as the derivatives of some functor. We therefore have the
following result which is proved in Theorem~\ref{thm:n-excisive}.

\begin{theorem} \label{thm:intro2}
For each integer $N \geq 1$, there is an equivalence between the homotopy category of
pointed finitary $N$-excisive functors $F: \mathsf{C} \to \mathsf{D}$ and that of
$N$-truncated $\der_*\Phi$-coalgebras. Letting $N\to\infty$, there is an equivalence between the homotopy category of all polynomial functors and that of all bounded $\der_*\Phi$-coalgebras.
\end{theorem}

\begin{remarks*}
We do not know of a model structure on the category of
$\der_*\Phi$-coalgebras and thus the above result does not arise from
a Quillen equivalence. Instead we develop directly a homotopy theory
for coalgebras over a comonad which may be of
independent interest. (See Section~\ref{sec:coalgebras}.)

In reality, we need to replace the comonad $\der_* \Phi$ with a homotopically correct version. To do this we replace $[\C,\D]$ with a Quillen equivalent model category in which every object is cofibrant. The comonad we actually use is then $\der_* uc \Phi$ where $u$ and $c$ form a Quillen equivalence between $[\C,\D]$ and this other category. This is explained in detail in Section~\ref{sec:descent}. With this in mind, we drop $u$ and $c$ from the notation and just write $\der_* \Phi$ for the comonad in question.

Though we only prove our results for functors between the categories $\based$ and $\spectra$, much of our approach relies only on formal properties of the calculus of functors. In principle it should be applicable to any functor for which there is an appropriate notion of Taylor tower and of derivatives.
\end{remarks*}

To make Theorems \ref{thm:intro1} and \ref{thm:intro2} more useful, it is desirable to have an explicit description of the comonad $\der_*\Phi$ and of the category of coalgebras over this comonad. In this paper we partially achieve this goal for functors with values in spectra. We now describe our results about this, so most definitions and statements have two versions: one for functors $\spectra \to \spectra$ and one for functors $\based \to \spectra$.

Let us use the notation $X^{[n]}$ to mean the following: If $X$ is a spectrum then $X^{[n]}=X^{\wedge n}$. If $X$ is a space then $X^{[n]}=X^{\wedge n}/\Delta^n X$ where $\Delta^n X\subset X^{\wedge n}$ is the fat diagonal. For $r \leq n$, there is a functor $K_r$ from the category of $\Sigma_n$-spectra to the category of $\Sigma_r$-spectra given by
\begin{equation}\label{eq: K_r}
K_rA_n \homeq \der_r\left[X \mapsto (A_n \smsh X^{[n]})^{h\Sigma_n}\right].
\end{equation}
Here $X$ may live either in $\spectra$ or $\based$ depending on the type of functors being considered.

It turns out that the constructions $K_r$ encode all the information about the comonad $\der_*\Phi$, at least if we restrict our attention to truncated symmetric sequences. More specifically, for $r \leq s \leq n$ there is a natural $\Sigma_r$-equivariant map
\[ \delta_{r,s}: K_r A_n \to K_r K_s A_n \]
and for each $r$, a $\Sigma_r$-equivariant map
\[ \epsilon_r: K_r A_r \to A_r \]
that together reflect the comonad structure on $\der_* \Phi$. These maps are associative and unital in an appropriate sense. We prove the following result (Lemma~\ref{lem:derAn}), which encodes the $\der_*\Phi$-coalgebra structure on $\der_*F$ in terms of the individual maps $\delta_{r,s}$ and $\epsilon_r$.

\begin{theorem}\label{thm:spectra-valued}
Let $\C$ be either $\based$ or $\spectra$. Let $F\colon \C\longrightarrow \spectra$ be a functor. For each $r\le n$ there is a $\Sigma_r$-equivariant map
\[ \theta_{r,n}: \der_rF \to K_r \der_n F.\]
Moreover, for each $r \le s \le n$, the following diagram commutes
\[ \begin{diagram}
  \node{\der_r F} \arrow{e,t}{\theta_{r,s}} \arrow{s,l}{\theta_{r,n}} \node{K_r \der_s F} \arrow{s,r}{K_r\theta_{s,n}} \\
  \node{K_r \der_n F} \arrow{e,t}{\delta_{r,s}} \node{K_r K_s \der_n F}
\end{diagram} \]
and for each $r$ the following composite is the identity
\[
\der_rF \stackrel{\theta_{r,r}}{\longrightarrow} K_r\der_rF\stackrel{\epsilon_r}{\longrightarrow}\der_rF.
\]
The Taylor tower of $F$ can then be recovered from the symmetric sequence $\der_*F$ and the maps $\theta_{r,n}$.
\end{theorem}

Note that~\eqref{eq: K_r} describes the construction $K_r$ only up to homotopy. The choice of model for $K_r$ and the maps $\delta_{r,s}$ matters, because for Theorem~\ref{thm:spectra-valued} one needs a model for which the maps $\delta_{r,s}$ are strictly associative and unital. At the same time one would like to have a model that is as simple and explicit as possible. In this paper we do give a strictly associative model for the $K_r$ and $\delta_{r,s}$, but it is not really explicit. On the other hand, we also give more explicit models for $K_r$ (in Proposition~\ref{prop:derAn-spsp} for the $[\spectra, \spectra]$ case and in Proposition~\ref{prop:norm} for the $[\based, \spectra]$ case), but these models are only associative up to homotopy. We have another approach to these constructions, based on modules over pro-operads, that we believe yields models that are both explicit and strictly associative. (See Remark \ref{rem:pro-operad}.) For reasons of space we leave these results to a subsequent paper.

For a functor $F: \based \to \spectra$, the $\der_*\Phi$-coalgebra structure on $\der_*F$ takes a fairly simple form which is worth noting here. The following is Proposition~\ref{prop:norm}.
\begin{theorem}\label{thm:spaces-spectra}
The derivatives $\der_*F$ of a functor $F: \based \to \spectra$ form a right module over the operad $\der_*I$. They also are equipped with maps $\theta_{r,n}$ that make the following diagrams commute up to homotopy:
\[ \begin{diagram}
  \node[2]{ \left[ \prod_{\un{n} \epi \un{r}} \Map(\der_{n_1}I \smsh \dots \smsh \der_{n_r}I, \der_nF) \right]_{h\Sigma_n}} \arrow{s,r}{N} \\
  \node{\der_rF} \arrow{ne,t}{\theta_{r,n}} \arrow{e,b}{\psi_{r,n}}
    \node{ \left[ \prod_{\un{n} \epi \un{r}} \Map(\der_{n_1}I \smsh \dots \smsh \der_{n_r}I, \der_nF) \right]^{h\Sigma_n}.}
\end{diagram} \]
Here the maps $\psi_{r,n}$ are determined by the right $\der_*I$-module structure on $\der_*F$, and $N$ is the usual norm map from homotopy orbits to homotopy fixed points.
\end{theorem}
The structure on $\der_*F$ described in Theorem~\ref{thm:spaces-spectra} can be called a \emph{divided power right $\der_*I$-module} (at least up to homotopy), with the lifts $\theta_{r,n}$ forming the analogue of a divided power structure on a commutative algebra.

\begin{example}
We demonstrate our theory here by classifying $2$-excisive functors with values in $\spectra$. These cases are well-known, and can be figured out by hand quite easily, but serve to illustrate the simplest cases of Theorems \ref{thm:spectra-valued} and \ref{thm:spaces-spectra}. Later in the paper we give a similar description of $3$-excisive functors.

Let us consider the $[\spectra, \spectra]$ case first.
It is not difficult to show that in this case $K_1(A_2)$ is equivalent to the Tate construction $\Tate_{\Sigma_2}(A_2)$. It follows that the first two derivatives of a functor $F\colon \spectra \to \spectra$ are connected by a map
\[
\der_1F \longrightarrow \Tate_{\Sigma_2} (\der_2F).
\]
Since at this level there are no compatibility conditions, this map completely determines the quadratic part of $F$. To put it another way: homotopy classes of $2$-excisive functors from $\spectra$ to $\spectra$ with prescribed derivatives $A_1$ and $A_2$ are in bijective correspondence with homotopy classes of maps $A_1 \rightarrow \Tate_{\Sigma_2}(A_2)$.

In the $[\based, \spectra]$ case, the right $\der_*I$-module structure on $\der_*F$ amounts to a map
\[
\der_1F \longrightarrow \Map(\partial_2I, \partial_2F)^{h\Sigma_2}.
\]
Since $\partial_2I\simeq S^{-1}$, this is the same thing as a map
\[
\der_1F \longrightarrow  (\partial_2F \wedge S^1)^{h\Sigma_2}.
\]
Theorem~\ref{thm:spaces-spectra} then says that this map lifts, over the norm, to a map of the form
\[
\der_1F \longrightarrow (\partial_2F\wedge S^1)_{h\Sigma_2}.
\]
Again there are no compatibility conditions so this map completely determines the quadratic part of $F$. Thus homotopy classes of $2$-excisive functors from $\based$ to $\spectra$ with prescribed derivatives $A_1$ and $A_2$ are in bijective correspondence with homotopy classes of maps
$A_1 \rightarrow (A_2\wedge S^1)_{h\Sigma_2}$.
\end{example}

\begin{remark}
Our results for functors from $\spectra$ to $\spectra$ overlap with those of McCarthy \cite{mccarthy:2001}. Here is one way to see the connection. It follows from McCarthy's work on dual calculus that there are homotopy pullback squares of the following form.
\[ \begin{diagram}
  \node{P_nF(X)} \arrow{e} \arrow{s} \node{(\der_nF \smsh X^{\smsh n})^{h\Sigma_n}} \arrow{s} \\
  \node{P_{n-1}F(X)} \arrow{e} \node{\Tate_{\Sigma_n}(\der_nF \smsh X^{\smsh n}).}
\end{diagram} \]
For an explicit construction of this square see Kuhn (\cite[1.9]{kuhn:2004}) or Chaoha (\cite[Theorem 3.1]{chaoha:2004}).
These squares give an inductive description of the data needed to reconstruct the Taylor tower. Our maps $\theta_{r,n}$ in the $\spectra$ to $\spectra$ case can be derived from McCarthy's pullback square: they are given by taking the \ord{r} derivative of the bottom horizontal map. What is new here is the explicit description of the compatibility conditions satisfied by these maps.

Conversely, it is possible to use our approach to derive McCarthy' pullback square, and also extend it to functors from $\based$ to $\spectra$. This is done in Corollary~\ref{cor:universal}.
\end{remark}

Our general theory also applies to functors that take values in based spaces. In each such case (either from spectra to spaces, or spaces to spaces) there is a comonad $\der_*\Phi$ (on the category of left $\der_*I$-modules, or $\der_*I$-bimodules, respectively) that acts on the derivatives of such a functor, and such that the Taylor tower of the functor can be recovered from this action. In these cases, however, it appears to be much harder to give a completely explicit description of what it means to be a coalgebra over $\der_*\Phi$. We do offer, in Section \ref{sec:top}, a classification of $2$-excisive functors in each of these settings.

Note that while we have far less to say about the space-valued case, we view the fact that our general theory does apply to it as one of the most significant results of the paper. The complexity we encounter in trying to calculate the comonad $\der_* \Phi$ reflects structure on the derivatives of space-valued functors that bears further study.

\subsection*{Outline of the paper}

Our main result provides an equivalence between a homotopy category of functors and that of coalgebras over a certain comonad. In section \ref{sec:coalgebras} we describe the homotopy theory for such coalgebras that we have in mind, and we construct the relevant homotopy category. Section \ref{sec:descent} is concerned with our version of homotopic descent theory. Here we prove a homotopic version of the Barr-Beck Comonadicity Theorem for a Quillen adjunction.

In section \ref{sec:taylor} we prove our main result classifying Taylor towers. Then in section \ref{sec:sp} we focus on spectrum-valued functors and derive some general results about them. In section~\ref{sec:specspec} we give an explicit (but only up to homotopy) description of the coalgebra structures for functors from $\spectra$ to $\spectra$, and in section~\ref{sec:topspec} we do the same for functors from $\based$ to $\spectra$. Finally, in section \ref{sec:top}, we show that our general theory also applies to functors with values in based spaces.

\subsection*{Acknowledgements}

We wish to acknowledge our debt to Kathryn Hess's paper on homotopic descent~\cite{hess:2010} which provided the initial idea for this approach to analyzing the calculus of functors. As usual, we benefited greatly along the way from conversations with Bill Dwyer. We also thank Haynes Miller for pointing us towards Radulescu-Banu's work on the construction of a cofibrant replacement comonad. The second author's joint work with Emily Riehl \cite{ching/riehl:2014} allowed for a substantial simplification of section 2, and the second author would like to thank John E. Harper for useful conversations concerning section 1. The first author was supported by NSF grant DMS-0968221 and the second author by NSF grant DMS-1144149.

\section{Homotopy theory for coalgebras over comonads} \label{sec:coalgebras}

In this section we consider a comonad $K$ defined on some category $\cat{B}$ that is equipped with a homotopy theory. Our goal is to describe a homotopy theory for the category of coalgebras over $K$. In our case these homotopy theories are described in terms of an enrichment in topological spaces.

Our approach is to define, for each pair of coalgebras, $A$ and $A'$, a suitable space of `derived' coalgebra maps from $A$ to $A'$. These mapping spaces have a composition that is associative up to coherent homotopies and so determine a topologically-enriched $A_\infty$-category whose objects are the $K$-coalgebras. There is an associated homotopy category (in which composition is strictly associative) given by taking path components of the mapping spaces. We show that a derived coalgebra map from $A$ to $A'$ is invertible in this homotopy category if and only if its underlying map $A \to A'$ is an isomorphism in the homotopy category associated to the topological category $\cat{B}$.

Previous work has established such homotopy theories in a different way. Model structures on particular categories of coalgebras have been studied by various authors starting with Quillen's work \cite{quillen:1969} on rational homotopy theory. Hess and Shipley \cite{hess/shipley:2014} have given general conditions under which coalgebras over a comonad on a model category inherit a model structure in which the cofibrations and weak equivalences are detected in the underlying category.

In the quasi-categorical setting, Riehl and Verity \cite{riehl/verity:2013} have described quasi-categories of algebras over a (homotopy coherent) monad. Their results have dual versions for coalgebras over a comonad. Lurie's approach to monads \cite[4.7]{lurie:2014} also seems likely to dualize to the comonad case.

\begin{notation} \label{not:enrich}
For this section, we let $\cat{B}$ be a category enriched in the category $\Top$ of compactly generated weak Hausdorff spaces. For objects $A,A'$ in $\cat{B}$, we write $\Hom(A,A')$ for the space of maps from $A$ to $A'$. Throughout this section, any functor $K: \cat{B} \to \cat{B}$ is assumed to be topologically-enriched, so that there are maps of spaces
\[ \Hom(A,A') \to \Hom(KA,KA') \]
for any $A,A' \in \cat{B}$.
\end{notation}

\begin{definition} \label{def:comonad}
A \emph{comonad} $K$ on the category $\cat{B}$ consists of a (topologically-enriched) functor
\[ K: \cat{B} \to \cat{B}, \]
together with (enriched) natural transformations
\[ \delta: K \to KK; \quad \epsilon: K \to 1_{\cat{B}} \]
such that the following diagrams commute
\[ \begin{diagram}
  \node{K} \arrow{e,t}{\delta} \arrow{s,l}{\delta} \node{KK} \arrow{s,r}{K\delta} \\ \node{KK} \arrow{e,t}{\delta K} \node{KKK}
\end{diagram}, \quad
\begin{diagram}
  \node{K} \arrow{e,t}{\delta} \arrow{se,=} \arrow{s,l}{\delta} \node{KK} \arrow{s,r}{K\epsilon} \\ \node{KK} \arrow{e,b}{\epsilon K} \node{K}
\end{diagram} \]
A \emph{$K$-coalgebra} consists of an object $A \in \cat{B}$ together with a morphism
\[ \theta: A \to KA \]
such that the following diagrams commute in $\cat{B}$
\[ \begin{diagram}
  \node{A} \arrow{e,t}{\theta} \arrow{s,l}{\theta} \node{KA} \arrow{s,r}{K\theta} \\ \node{KA} \arrow{e,t}{\delta_A} \node{KKA}
\end{diagram}, \quad
\begin{diagram}
  \node{A} \arrow{e,t}{\theta} \arrow{se,=} \node{KA} \arrow{s,r}{\epsilon_A} \\ \node[2]{A}
\end{diagram} \]
For $K$-coalgebras $A$ and $A'$, a \emph{(strict) morphism of $K$-coalgebras from $A$ to $A'$} is a morphism
\[ f: A \to A' \]
in $\cat{B}$ such that the following diagram commutes
\[ \begin{diagram}
  \node{A} \arrow{e,t}{f} \arrow{s,l}{\theta_A} \node{A'} \arrow{s,r}{\theta_{A'}} \\ \node{KA} \arrow{e,t}{Kf} \node{KA'}
\end{diagram} \]
\end{definition}

We specify homotopical information for the category of $K$-coalgebras first by means of an enrichment over the category $\cTop$ of \emph{cosimplicial spaces} with respect to the `box product'. This is a symmetric monoidal structure due to Batanin \cite{batanin:1998} and used to great effect by McClure and Smith in their proof of the Deligne Conjecture \cite{mcclure/smith:2002}.

\begin{definition} \label{def:css-hom}
Let $A$ and $A'$ be $K$-coalgebras in $\cat{B}$. We define a cosimplicial space $\Hom^{\bullet}_{K}(A,A')$ by
\[ \Hom^{\bullet}_{K}(A,A') := \Hom(A,K^{\bullet}A') \]
with coface maps $\delta^i: \Hom(A,K^m A') \to \Hom(A,K^{m+1}A')$ given
\begin{itemize}
  \item for $i = 0$, by the composite
  \[ \Hom(A,K^m A') \to \Hom(KA,K^{m+1}A') \to \Hom(A,K^{m+1}A') \]
  where the first map comes from the topological enrichment of $K$, and the second from the $K$-coalgebra structure on $A$;
  \item for $i = 1,\dots,m$, by applying the comultiplication map $\delta: K \to KK$ to the \ord{i} copy of $K$ in $K^m A'$;
  \item for $i = m+1$, by the $K$-coalgebra structure on $A'$.
\end{itemize}
and codegeneracy maps $\sigma^j: \Hom(A,K^m A') \to \Hom(A,K^{m-1}A')$ given
\begin{itemize}
  \item for $j = 0,\dots,m-1$, by applying the counit $\epsilon: K \to 1_{\cat{B}}$ to the \ord{j+1} copy of $K$.
\end{itemize}
\end{definition}

\begin{definition} \label{def:box}
For two cosimplicial spaces $X^\bullet,Y^\bullet$, we define $(X^\bullet \binBox Y^\bullet)$ to be the cosimplicial space given by
\[ (X^\bullet \binBox Y^\bullet)^m := \colim \left( \coprod_{p+q = m-1} X^p \times Y^q \rightrightarrows \coprod_{p+q = m} X^p \times Y^q \right). \]
The two maps in this coequalizer are, respectively, $(\delta^{p+1},1)$ and $(1,\delta^0)$. The coface maps on $X^\bullet \binBox Y^\bullet$ are given by
\[ \delta^i = \begin{cases} (\delta^i,1) & i = 0,\dots,p+1; \\ (1,\delta^{i-p-1}) & i = p+1,\dots,m+1, \end{cases} \]
and the codegeneracy maps by
\[ \sigma^j = \begin{cases} (\sigma^j,1) & j = 0,\dots,p-1; \\ (1,\sigma^{j-p}) & j = p,\dots,m-1. \end{cases} \]
\end{definition}

\begin{proposition}[Batanin, \cite{batanin:1998}] \label{prop:box}
The construction $\binBox$ is a monoidal product on the category $\cTop$ of cosimplicial spaces with unit given by the constant cosimplicial space $*$.
\end{proposition}

To see that the cosimplicial spaces in Definition \ref{def:css-hom} form part of a $\cTop$-enriched category, we have to describe the composition and identity morphisms.

\begin{definition} \label{def:css-comp}
Let $A,A',A''$ be $K$-coalgebras. We define a composition map
\[ \mu: \Hom^\bullet_{K}(A,A') \binBox \Hom^\bullet_{K}(A',A'') \to \Hom^\bullet_{K}(A,A'') \]
via the composites
\[ \Hom(A,K^pA') \times \Hom(A',K^qA'') \to \Hom(A,K^pA') \times \Hom(K^pA',K^{p+q}A'') \to \Hom(A,K^{p+q}A''). \]
where the first map uses the topological enrichment of $K$ and the second is composition in the topological category $\cat{B}$.

We also define an identity map
\[ \iota: * \to \Hom^\bullet_{K}(A,A) \]
via the composites
\[ \iota_m: * \to \Hom(A,A) \to \Hom(K^mA,K^mA) \to \Hom(A,K^mA) \]
where the first map picks out the identity morphism on $A$, the second comes from the topological enrichment of $K$, and the last is given by iterating the $K$-coalgebra structure on $A$.
\end{definition}

\begin{proposition} \label{prop:css-enrich}
The maps $\mu$ and $\iota$ of \ref{def:css-comp} determine a category enriched in $(\mathsf{cTop},\binBox,*)$ whose objects are the $K$-coalgebras. We denote this category by $\cat{B}_K$.
\end{proposition}

\begin{remark} \label{rem:ss-hom}
The $\cTop$-enriched category $\cat{B}_K$ has an underlying topological category with the same objects, and with the space of morphisms from $A$ to $A'$ given by
\[ \Hom_{K}(A,A') := \Hom_{\cTop}(*,\Hom^\bullet_{K}(A,A')). \]
The resulting topological category is the usual category of $K$-coalgebras. In particular, the points in $\Hom_{K}(A,A')$ are the strict $K$-coalgebra morphisms $A \to A'$ in the sense of Definition \ref{def:comonad}.
\end{remark}

There is another way to construct mapping spaces in a $\cTop$-enriched category, which is by taking totalizations of the cosimplicial spaces instead of their strict limits. As we see below this produces a topological category for which composition is not strictly associative, but is associative up to higher coherent homotopies. It is this approach that yields the homotopy theory we are interested in.

It is convenient to use the `restricted' (or `fat') totalization throughout this paper since this has the correct homotopy type without requiring a Reedy fibrant replacement.

\begin{definition} \label{def:tot}
The \emph{restricted totalization} of a cosimplicial space $X^\bullet$ is the space
\[ \Tot(X^\bullet) := \Hom_{\Deltainj}(\Delta^\bullet,X^\bullet) \]
where $\Deltainj$ is the subcategory of the simplicial indexing category consisting only of the injective morphisms. The restricted totalization is homotopy invariant, and is equivalent to the ordinary totalization of a Reedy fibrant replacement.
\end{definition}

\begin{definition} \label{def:derived-ss-hom}
Let $A,A'$ be $K$-coalgebras in $\cat{B}$. The \emph{space of derived $K$-coalgebra maps from $A$ to $A'$} is the restricted totalization
\[ \hHom_{K}(A,A') := \Tot \Hom^\bullet_{K}(A,A') = \Hom_{\Deltainj}(\Delta^\bullet,\Hom(A,K^\bullet A')). \]
The points in this space are no longer strict $K$-coalgebra maps, but commute with the coalgebra structures only up to higher coherent homotopies.
\end{definition}

\begin{definition} \label{def:derived-coalgebra-map}
A \emph{derived $K$-coalgebra map from $A$ to $A'$} is a point in the space $\hHom_{K}(A,A')$. Explicitly, such an $f$ consists of a collection of morphisms
\[ f_n: \Delta^n \to \Hom(A,K^n A') \]
satisfying some compatibility conditions. In particular, there is a morphism
\[ f_0: A \to A' \]
in the underlying category of the topological category $\cat{B}$, and the map $f_1: \Delta^1 \to \Hom(A,KA')$ provides a homotopy between the two composites in the square
\[ \begin{diagram}
  \node{A} \arrow{e} \arrow{s} \node{A'} \arrow{s} \\
  \node{KA} \arrow{e} \node{KA'}
\end{diagram} \]
The maps $f_n$ can be viewed as a set of higher coherent homotopies that generalize this description of $f_1$.
\end{definition}

There is no single well-defined and associative composition for derived $K$-coalgebra maps, but we can compose them up to the action of the following $A_\infty$-operad.

\begin{definition} \label{def:operad}
We define a (non-symmetric) operad in $\Top$ by taking, for $n \geq 0$,
\[ \mathsf{A}_n := \Hom_{\Deltainj}(\Delta^\bullet, (\Delta^\bullet)^{\binBox n}). \]
This is the `coendomorphism operad' of the cosimplicial space $\Delta^\bullet$ with respect to the box product. McClure and Smith prove in \cite[3.5]{mcclure/smith:2004} that $\mathsf{A}_n$ is contractible for each $n \geq 0$.

For $n \geq 0$ and for $K$-coalgebras $A_0,\dots,A_n$ we define natural composition maps
\begin{equation} \label{eq:comp} \mathsf{A}_n \times \hHom_{K}(A_0,A_1) \times \dots \times \hHom_{K}(A_{n-1},A_n) \to \hHom_{K}(A_0,A_n) \end{equation}
by the composites
\[ \begin{split}
    \Hom_{\Deltainj}(\Delta^\bullet, &(\Delta^\bullet)^{\binBox n}) \times \Hom_{\Deltainj}(\Delta^\bullet,\Hom^{\bullet}_{K}(A_0,A_1)) \times \dots \times \Hom_{\Deltainj}(\Delta^\bullet,\Hom^{\bullet}_{K}(A_{n-1},A_n)) \\
        &\to \Hom_{\Deltainj}(\Delta^\bullet, (\Delta^\bullet)^{\binBox n}) \times \Hom_{\Deltainj}((\Delta^\bullet)^{\binBox n}, \Hom^\bullet_{K}(A_0,A_1) \binBox \dots \binBox \Hom^\bullet_{K}(A_{n-1},A_n)) \\
        &\to \Hom_{\Deltainj}(\Delta^\bullet, \Hom^\bullet_{K}(A_0,A_1) \binBox \dots \binBox \Hom^\bullet_{K}(A_{n-1},A_n)) \\
        &\to \Hom_{\Deltainj}(\Delta^\bullet, \Hom^\bullet_{K}(A_0,A_n))
\end{split} \]
where the final map is given by iterating the map $\mu$ of Definition \ref{def:css-comp}.
\end{definition}

\begin{proposition} \label{prop:operad}
The maps (\ref{eq:comp}) determine a topological-$A_{\infty}$-category whose objects are the $K$-coalgebras in $\cat{B}$ and whose morphism spaces are the spaces $\hHom_{K}(A,A')$.
\end{proposition}
\begin{proof}
An $A_\infty$-category consists of generalized composition maps of the form (\ref{eq:comp}) that satisfy associativity and unit conditions with respect to the operad structure on the $A_\infty$ operad $\mathsf{A}$. In our case these conditions follow formally from the construction of the maps (\ref{eq:comp}) and the fact that $\mathsf{A}$ is the coendomorphism operad on the object $\Delta^\bullet$.
\end{proof}

As for an ordinary topological category, we can associate a homotopy category to any topological $A_\infty$-category by taking path components of the mapping objects.

\begin{definition} \label{def:homotopy}
The \emph{homotopy category of $K$-coalgebras} has as its objects the $K$-coalgebras in $\cat{B}$, and the morphisms from $A$ to $A'$ are the path components of the mapping space $\hHom_{K}(A,A')$. We denote the morphism sets in this homotopy category by
\[ [A,A']_{K} := \pi_0 \hHom_{K}(A,A') = \pi_0 \Tot \Hom(A,K^\bullet A'). \]
Composition and identities are given by applying $\pi_0$ to the maps (\ref{eq:comp}) with $n = 2$ and $n = 0$ respectively. The associativity and unit conditions follow from the associativity of the maps (\ref{eq:comp}) and the fact that $\mathsf{A}_n$ is contractible for all $n$.
\end{definition}

We now identify those derived coalgebra maps that are invertible in this homotopy category. It turns out to be precisely those whose underlying map is invertible in the homotopy category associated to the topological category $\cat{B}$.

\begin{proposition} \label{prop:weqs}
The derived $K$-coalgebra map $f:A \to A'$ induces an isomorphism in the homotopy category of $K$-coalgebras if and only if the underlying map $f_0: A \to A'$ induces an isomorphism in the homotopy category of the topological category $\cat{B}$.
\end{proposition}
\begin{proof}
The map $f$ induces an isomorphism in the homotopy category of $K$-coalgebras if and only if, for any $K$-coalgebra $A''$, composition with $f$ induces a bijection
\[ f^*: [A',A'']_{K} \arrow{e,t}{\isom} [A,A'']_{K}. \]
In order to analyze this condition, we fix a point in the space $\mathsf{A}_2$ by means of an isomorphism of cosimplicial spaces
\[ \alpha: \Delta^\bullet \arrow{e,t}{\isom} \Delta^\bullet \binBox \Delta^\bullet \]
as described in McClure-Smith \cite[3.5]{mcclure/smith:2004}. This choice then fixes a composition operation for derived $K$-coalgebra maps.

The map $f^*$ can then be described in the following way. It is given by taking the path components of a map of spaces of the form
\[ \dgTEXTARROWLENGTH=3em \Tot \Hom(A', K^\bullet A'') \arrow{e,t}{\alpha^*} \Tot (\Delta^\bullet \binBox \Hom(A', K^\bullet A'')) \arrow{e,t}{\beta^*} \Tot \Hom(A, K^\bullet A''). \]
The first map $\alpha^*$ is determined by $\alpha$ and is the composite
\[ \dgTEXTARROWLENGTH=4em \begin{split} \Hom_{\Deltainj}(\Delta^\bullet, X^\bullet)
    &\arrow{e,t}{\Delta^\bullet \binBox -} \Hom_{\Deltainj}(\Delta^\bullet \binBox \Delta^\bullet, \Delta^\bullet \binBox X^\bullet) \\
    &\arrow{e,t}{\alpha} \Hom_{\Deltainj}(\Delta^\bullet, \Delta^\bullet \binBox X^\bullet)
\end{split} \]
where $X^\bullet = \Hom^\bullet_{K}(A',A'')$. The map $\beta^*$ is induced by taking totalizations of the following map of cosimplicial spaces
\[ \beta: \Delta^\bullet \binBox \Hom(A',K^\bullet A'') \arrow{e,t}{f} \Hom(A,K^\bullet A') \binBox \Hom(A',K^\bullet A'') \arrow{e,t}{\mu} \Hom(A,K^\bullet A''). \]
We first show that $\alpha^*$ is a weak equivalence. For this we need the following lemma.

\begin{lemma} \label{lem:ij}
Let $X^{\bullet}$ be any cosimplicial space, and let $j: \Delta^{\bullet} \binBox X^{\bullet} \to X^{\bullet}$ be the map induced by $\Delta^\bullet \to *$. Then $j$ is a levelwise weak equivalence of cosimplicial spaces.
\end{lemma}
\begin{proof}
For each $m \geq 0$, we define
\[ i: X^m \to [\Delta^\bullet \binBox X^\bullet]^m \]
by the sequence
\[ X^m \isom \Delta^0 \times X^m \into \coprod_{p+q=m} \Delta^p \times X^q \to [\Delta^\bullet \binBox X^\bullet]^m. \]
Then $ji$ is the identity on $X^m$. For each $p$, there is a `straight-line' homotopy
\[ h: \Delta^1 \times \Delta^p \to \Delta^p \]
between the identity and the constant map to the `terminal' vertex of $\Delta^p$ (i.e. the image of the map $\Delta^0 \to \Delta^p$ induced by the map $[0] \to [p]$ in $\Deltainj$ given by $0 \mapsto p$). These together induce a homotopy
\[ \Delta^1 \times [\Delta^\bullet \binBox X^\bullet]^m \to [\Delta^\bullet \binBox X^\bullet]^m \]
between the identity map and $ij$. Thus $j$ is a levelwise weak equivalence.
\end{proof}

Now post-compose $\alpha^*$ with the weak equivalence
\[ j_*: \Hom_{\Deltainj}(\Delta^\bullet, \Delta^\bullet \binBox X^\bullet) \to \Hom_{\Deltainj}(\Delta^\bullet, X^\bullet). \]
The resulting map
\[ \Hom_{\Deltainj}(\Delta^\bullet,X^\bullet) \to \Hom_{\Deltainj}(\Delta^\bullet, X^\bullet) \]
is that induced by the composite
\[ \Delta^\bullet \arrow{e,tb}{\isom}{\alpha} \Delta^\bullet \binBox \Delta^\bullet \arrow{e,tb}{\sim}{j} \Delta^\bullet \]
and so is also a weak equivalence. Thus, $\alpha^*$ is a weak equivalence.

It now follows that $f$ induces an isomorphism in the homotopy category if and only if the map $\beta^*$ is a $\pi_0$-isomorphism for all $A''$.

Suppose first that $f_0$ induces an isomorphism in the homotopy category of $\cat{B}$. Then the map $f_0^*$ in the following diagram of spaces is a weak homotopy equivalence for each $m \geq 0$:
\[ \begin{diagram}
  \node{\Hom(A',K^m A'')} \arrow{e,tb}{i}{\sim} \arrow{se,tb}{\sim}{f_0^*} \node{[\Delta^\bullet \binBox \Hom(A',K^\bullet A'')]^m} \arrow{s,r}{\beta^m} \\
  \node[2]{\Hom(A,K^m A'')}
\end{diagram} \]
Here $i$ is the weak equivalence of Lemma~\ref{lem:ij}. The above diagram commutes and so we deduce that $\beta^m$ is a weak equivalence for all $m$. Therefore the induced map $\beta^*$ is a weak equivalence, so in particular a $\pi_0$-isomorphism.

Conversely, suppose that $\beta^*$ is a $\pi_0$-isomorphism for all $K$-coalgebras $A''$. Then, in particular, taking $A''$ to be the cofree coalgebra $KX$ for $X \in \cat{B}$, we have diagrams of spaces
\[ \begin{diagram}
    \node{\Tot(\Delta^\bullet \binBox \Hom(A',K^\bullet KX))} \arrow{s,lr}{j}{\sim} \arrow{e,t}{\beta^*} \node{\Tot \Hom(A,K^\bullet KX)} \arrow[2]{s,r}{\sim} \\
    \node{\Tot \Hom(A',K^\bullet KX)} \arrow{s,r}{\sim} \\
    \node{\Hom(A',X)} \arrow{e,t}{f_0^*} \node{\Hom(A,X)}
\end{diagram} \]
where the bottom-left and right-hand vertical maps are induced by the extra codegeneracies in the coaugmented cosimplicial objects $\Hom(-,K^\bullet KX)$. Since $\beta^*$ is a $\pi_0$-isomorphism, it follows that $f_0^*$ is also, i.e. that $f_0$ induces bijections
\[ [A',X] \arrow{e,t}{\isom} [A,X] \]
for all $X \in \cat{B}$, and hence induces an isomorphism in the homotopy category.
\end{proof}

\begin{remark}
There are, of course, alternative ways to construct a homotopy category of $K$-coalgebras. One could, for example, start with the ordinary category $\cat{B}_K$ whose morphisms are the strict $K$-coalgebra maps, and invert those morphisms that are equivalences in $\cat{B}$. It is not clear in general if this should give the same homotopy category. The particular category we describe above, however, is what appears in our version of homotopic descent theory, to which we now turn.
\end{remark}

In the remainder of the paper, we work with categories enriched in simplicial sets rather than topological spaces. The results of this section can easily be transported to that context using the geometric realization and singular simplicial set functors. The following proposition summarizes the situation.

\begin{proposition} \label{prop:simplicial}
Let $\cat{B}$ be a category enriched in simplicial sets and let $K: \cat{B} \to \cat{B}$ be a simplicially-enriched comonad. Then there is a simplicial-$A_\infty$-category $\cat{B}_K$ whose objects are the $K$-coalgebras, and with simplicial mapping spaces given by the singular simplicial sets
\begin{equation} \label{eq:simplicial} \hHom_{K}(A,A') := \Sing \Tot(|\Hom_{\cat{B}}(A,K^\bullet A')|). \end{equation}
Let $f:A \to A'$ be a derived $K$-coalgebra map for which $f_0:A \to A'$ arises from a morphism in the simplicial category $\cat{B}$. Then $f$ induces an isomorphism in the associated homotopy category of $K$-coalgebras if and only if $f_0$ induces an isomorphism in the homotopy category of $\cat{B}$.
\end{proposition}
\begin{proof}
We can make $\cat{B}$ into a topologically-enriched category by using the geometric realization functor to define mapping spaces:
\[ \Hom^{\Top}_{\cat{B}}(X,X') := |\Hom_{\cat{B}}(X,X')|. \]
The simplicially-enriched functor $K: \cat{B} \to \cat{B}$ is then also enriched with respect to this topological structure. The theory of this section now applies. In particular, Proposition~\ref{prop:operad} gives us a topological-$A_\infty$-category whose objects are the $K$-coalgebras and whose mapping spaces are the totalizations of the cosimplicial spaces
\[ \Hom^{\Top}_{\cat{B}}(A,K^\bullet A') = |\Hom_{\cat{B}}(A,K^\bullet A')| \]
and where composition is controlled by the $A_\infty$-operad $\mathsf{A}$ of Definition~\ref{def:operad}. Applying the singular simplicial set functor to $\mathsf{A}$ and to these mapping spaces, we obtain a simplicial-$A_\infty$-category with simplicial sets of maps given by (\ref{eq:simplicial}) and with composition controlled by an $A_\infty$-operad in simplicial sets. The homotopy category of $K$-coalgebras is then given by applying $\pi_0$ to the simplicial sets $\hHom_{K}(A,A')$ which is of course equivalent to applying $\pi_0$ to the mapping spaces before taking singular simplicial sets.

The process of forming a topological category from a simplicial category via geometric realization changes the notion of `morphism'. A morphism $X \to X'$ in the topological category $\cat{B}$ is a point in the geometric realization $|\Hom_{\cat{B}}(X,X')|$. Such a point may or may not arise from a vertex in the simplicial set $\Hom_{\cat{B}}(X,X')$, that is from a morphism in the simplicial category $\cat{B}$. However, the homotopy categories associated to the simplicial and topological enrichments of $\cat{B}$ are the same. (We can apply $\pi_0$ either before or after geometric realization.)

So let $f:A \to A'$ be a derived $K$-coalgebra map for which the underlying map $f_0:A \to A'$ \emph{does} arise from a morphism in the simplicial category $\cat{B}$. Then it follows from \ref{prop:weqs} that $f$ induces an isomorphism in the homotopy category of $K$-coalgebras if and only if $f_0$ induces an isomorphism in the homotopy category of $\cat{B}$.
\end{proof}

\section{Descent for Quillen adjunctions} \label{sec:descent}

In this section we develop a homotopical version of descent theory for studying Quillen adjunctions between model categories. Associated to any adjunction $F: \cat{A} \rightleftarrows \cat{B}: G$, there is a `descent theory' that compares the category $\cat{A}$ with the category $\cat{B}_K$ of coalgebras over the comonad $K = FG$. Classically, this was developed by Grothendieck to study the extension/restriction of scalars adjunction associated to a ring homomorphism.

Recently, interest has developed in `homotopic' descent theory. (See, for example, Hess~\cite{hess:2010} and Lurie~\cite[4.7]{lurie:2014}.) Now $\cat{A}$ and $\cat{B}$ are categories with some notion of homotopy, preserved by the adjunction, and the question is whether the categories $\cat{A}$ and $\cat{B}_K$ are equivalent in some homotopical sense.

Here we develop part of such a theory for Quillen adjunctions satisfying the following conditions.

\begin{hypothesis} \label{hyp:adjunction}
We assume that $\cat{A}$ and $\cat{B}$ are simplicial model categories and that we have a simplicially-enriched Quillen adjunction
\begin{equation} \label{eq:adj-FG} F: \cat{A} \rightleftarrows \cat{B}: G. \end{equation}
We assume further that
\begin{enumerate}
  \item either:
  \begin{enumerate}
    \item all objects in $\cat{A}$ are cofibrant; or
    \item $\cat{A}$ is Quillen equivalent to a combinatorial model category;
   \end{enumerate}
   \item all objects in $\cat{B}$ are fibrant.
\end{enumerate}
The point of condition (1b) is that it allows us to replace $\cat{A}$ with a model category in which every object is cofibrant: see Theorem~\ref{thm:models} below. Note that all standard examples of cofibrantly generated model categories are known to be Quillen equivalent to a combinatorial model category. (See \cite{rosicky:2009} for more on the connection.) In particular, this includes the categories of functors between based spaces and spectra that we will use in this paper.

Condition (2) can be relaxed to include any model category $\cat{B}$ that admits a simplicial Quillen equivalence to a model category in which all objects are fibrant: see Remark~\ref{rem:nikolaus}. The categories we have in mind for $\cat{B}$ already satisfy this condition so we do not need any greater generalization here.
\end{hypothesis}

The comonad directly associated to the adjunction (\ref{eq:adj-FG}), that is $FG$, is not typically well-behaved with respect to the homotopy theory. For example, if $G$ does not take values in cofibrant objects, there is no general reason why $FG$ should preserve weak equivalences. In order to get a homotopically meaningful descent theory, we replace the adjunction (\ref{eq:adj-FG}) with one that does not have these problems.

\begin{theorem} \label{thm:models}
Under the conditions of Hypothesis~\ref{hyp:adjunction} there exists a simplicial model category $\cat{A}_c$ in which every object is cofibrant, and a simplicial Quillen equivalence (with left adjoint $u$):
\[ u: \cat{A}_c \rightleftarrows \cat{A} : c. \]
\end{theorem}
\begin{proof}
When every object of $\cat{A}$ is cofibrant, we can of course take $(u,c)$ to be the identity Quillen equivalence. Under condition (1b) of \ref{hyp:adjunction}, the existence of such an equivalence is due to Ching and Riehl \cite{ching/riehl:2014}.
\end{proof}

\begin{remark}
When $\cat{A}$ is itself a combinatorial simplicial model category (as opposed to merely Quillen equivalent to one) we can be more concrete about the category $\cat{A}_c$. In this case, it can be taken to be the category of coalgebras over a simplicially-enriched cofibrant replacement comonad for $\cat{A}$, as constructed via the simplicial version of Garner's small object argument. That there is a suitable model structure on $\cat{A}_c$ is the content of \cite{ching/riehl:2014}.
\end{remark}

\begin{definition} \label{def:K-comonad}
We now fix a choice of the Quillen equivalence provided by Theorem~\ref{thm:models} and consider the composite of the two adjunctions (with left adjoints on top):
\begin{equation} \label{eq:adjs}
\begin{xy}
  (0,0)*+{\cat{A}_c}="a";(15,0)*+{\cat{A}}="b";(30,0)*+{\cat{B}}="c";
  {\ar^{u}@<0.5ex>"a";"b"};{\ar^{c}@<0.5ex>"b";"a"};
  {\ar^{F}@<0.5ex>"b";"c"};{\ar^{G}@<0.5ex>"c";"b"};
\end{xy}
\end{equation}
Let $(F',G')$ be the functors in this composed Quillen adjunction, i.e. $F' = Fu$ and $G' = cG$. We then view $(F',G')$ as a homotopically well-behaved version of the original adjunction $(F,G)$. In particular, notice that $F'$ and $G'$ both preserve all weak equivalences.

We define a comonad $K : \cat{B} \to \cat{B}$ by
\[ K = F'G' \]
with comultiplication/counit maps coming from the unit/counit for the adjunction. The comonad $K$ is also simplicial because it a composite of simplicial functors.
\end{definition}

\begin{remark} \label{rem:nikolaus}
Nikolaus shows in \cite{nikolaus:2011} that any cofibrantly generated model category $\cat{B}$ in which trivial cofibrations are monomorphisms admits a Quillen equivalence
\[ \cat{B} \rightleftarrows \cat{B}^f \]
where every object in the model category $\cat{B}^f$ is fibrant. Specifically, $\cat{B}^f$ is the category of algebras over a fibrant replacement monad constructed from Garner's small object argument. A simplicial version of Nikolaus's argument could be used to relax condition (2) of Hypothesis~\ref{hyp:adjunction}. In this case we would take the comonad $K$ to be defined on $\cat{B}^f$ and to be that associated to the composite of all three adjunctions
\[ \cat{A}_c \rightleftarrows \cat{A} \rightleftarrows \cat{B} \rightleftarrows \cat{B}^f. \]
The theory developed in the remainder of this section would then be unchanged, except that it would provide a comparison between the category $\cat{A}_c$ and the coalgebras for a comonad on $\cat{B}^f$ instead of on $\cat{B}$.
\end{remark}

Our plan is now to develop the theory of Section~\ref{sec:coalgebras} for the comonad $K$. In order to apply this, however, we need to have an underlying category that is enriched in \emph{fibrant} simplicial sets. We therefore restrict to the subcategory of cofibrant objects in $\cat{B}$.

\begin{definition}
Let $\tilde{\cat{B}}$ denote the full subcategory of cofibrant objects in $\cat{B}$. Since every object in the simplicial model category $\cat{B}$ is fibrant, it follows that the category $\tilde{\cat{B}}$ is enriched in fibrant simplicial sets. Moreover, the comonad $K$ restricts to a comonad
\[ K : \tilde{\cat{B}} \to \tilde{\cat{B}}. \]
Applying the simplicial version of the theory of Section~\ref{sec:coalgebras} (see Proposition~\ref{prop:simplicial}), we obtain a simplicial $A_\infty$-category $\tilde{\cat{B}}_{K}$ whose objects are the $K$-coalgebras that are cofibrant in $\cat{B}$, and with simplicial mapping objects given by (\ref{eq:simplicial}). Associated to $\tilde{\cat{B}}_{K}$ is a homotopy category which we refer to as the \emph{homotopy category of $K$-coalgebras}.
\end{definition}

\begin{definition} \label{def:K-coalgebras}
For $X \in \cat{A}_c$, the object $F'X$ has a canonical $K$-coalgebra structure induced by the unit of the adjunction $(F',G')$ in the usual way. Moreover, $F'X = FuX$ is cofibrant in $\cat{B}$ and so $F'$ lifts to a functor
\[ F': \cat{A}_c \to \tilde{\cat{B}}_{K}. \]
This is a functor on the level of strict categories where the right-hand side is the usual category of $K$-coalgebras. We are more interested, however, in thinking of the target of $F'$ as a simplicial $A_\infty$-category. In particular, $F'$ induces natural maps of simplicial sets
\begin{equation} \label{eq:hom-maps} \Hom_{\cat{A}_c}(X,X') \to \hHom_{K}(F'X,F'X') \end{equation}
which commute with the relevant composition maps.
\end{definition}

\begin{definition} \label{def:ho}
If $f: X \to X'$ is a weak equivalence in $\cat{A}_c$ then the induced map $F'(f)$ is a derived (in fact, strict) map of $K$-coalgebras whose underlying map is a weak equivalence in $\cat{B}$. By Proposition~\ref{prop:weqs}, we therefore get an induced functor on the level of homotopy categories which we denote by $\mathsf{Ho}(F)$:
\[ \mathsf{Ho}(F): \mathsf{Ho}(\cat{A}) \isom \mathsf{Ho}(\cat{A}_c) \to \mathsf{Ho}(\tilde{\cat{B}}_{K}). \]
\end{definition}

Our version of homotopic descent theory is concerned with the extent to which $\mathsf{Ho}(F)$ is an equivalence of categories. We start by describing how to recover an object in $\cat{A}$ from a $K$-coalgebra using a cobar construction.

\begin{definition} \label{def:cobar}
Let $A$ be a $K$-coalgebra. The \emph{cosimplicial cobar construction} on $A$ is the cosimplicial object in $\cat{A}_c$ given by
\[ \mathfrak{C}^\bullet(A) = G'(F'G')^\bullet A = G'K^\bullet A. \]
The coface maps $\delta^i: G'(F'G')^m A \to G'(F'G')^{m+1}A$ are given
\begin{itemize}
  \item for $i = 0,\dots,m$, by the unit $1 \to G'F'$ applied before the \ord{i+1} copy of $G'$;
  \item for $i = m+1$ by the $K$-coalgebra structure on $A$;
\end{itemize}
and the codegeneracy $\sigma^j: G'(F'G')^m A \to G'(F'G')^{m-1}A$ is given
\begin{itemize}
  \item for $j = 0,\dots,m-1$ by applying the counit $F'G' \to 1$ to the \ord{j+1} copy of $F'G'$.
\end{itemize}
We then define the \emph{cobar construction} on a $K$-coalgebra $A$ to be the object of $\cat{A}_c$ given by
\[ \mathfrak{C}(A) := \Tot(\mathfrak{C}^\bullet(A)). \]
This is the (restricted) totalization of a cosimplicial object in $\cat{A}_c$ formed using the cotensoring of $\cat{A}_c$ over simplicial sets. Since the cosimplicial object $\mathfrak{C}^\bullet(A)$ is levelwise fibrant, the cobar construction $\mathfrak{C}(A)$ is a fibrant object in $\cat{A}_c$.
\end{definition}

\begin{definition} \label{def:complete}
For $X \in \cat{A}_c$, there is a canonical map
\[ \eta_{X} : X \to \mathfrak{C}(F'X) = \Tot(G'(F'G')^\bullet F'X) \]
induced by the map
\[ X \to G'F'X \]
which is a coaugmentation for the cosimplicial cobar construction on $F'X$.

We refer to the object $\mathfrak{C}(F'X)$ as the \emph{$F'$-completion} of $X$. If $\eta_X$ is a weak equivalence, we say that $X$ is \emph{$F'$-complete}.
\end{definition}

\begin{definition} \label{def:complete-F}
We say that a fibrant object $X \in \cat{A}$ is \emph{$F$-complete} when the corresponding object $cX \in \cat{A}_c$ is $F'$-complete. In that case we have a zigzag of equivalences
\[ X \lweq ucX \weq u\mathfrak{C}(F'cX) \]
so that $X$ can be recovered, up to weak equivalence, from the $K$-coalgebra structure on $F'cX$.
\end{definition}

We now show that the functor $\mathsf{Ho}(F)$ of Definition~\ref{def:ho} embeds the homotopy category of $F$-complete objects in $\cat{A}$ into the homotopy category of $K$-coalgebras. This follows from the following result.

\begin{proposition} \label{prop:descent}
For objects $X,X' \in \cat{A}_c$ with $X'$ fibrant and $F'$-complete, the natural map
\[ \Hom_{\cat{A}_c}(X,X') \to \hHom_{K}(F'X,F'X'), \]
is a weak equivalence of simplicial sets.
\end{proposition}
\begin{proof}
The given map can be written as the composite
\[ \begin{split} \Hom_{\cat{A}_c}(X,X') &\arrow{e,t}{\eta_*} \Hom_{\cat{A}_c}(X,\Tot(G'(F'G')^\bullet F'X')) \\
    &\isom \Tot \Hom_{\cat{A}_c}(X,G'(F'G')^\bullet F'X') \\
    &\isom \Tot \Hom_{\cat{B}}(F'X,(F'G')^\bullet F'X') \isom \hHom_{K}(F'X,F'X')
\end{split} \]
By assumption $\eta$ is a weak equivalence between fibrant objects in $\cat{A}_c$ and so, as every object $X \in \cat{A}_c$ is cofibrant, the first map in this composite is a weak equivalence.
\end{proof}

\begin{corollary} \label{cor:descent}
Under the conditions of \ref{hyp:adjunction} the functor $\mathsf{Ho}(F)$ of Definition~\ref{def:ho} is a fully-faithful embedding of the homotopy category of $F$-complete objects in $\cat{A}$ into the homotopy category of $K$-coalgebras.
\end{corollary}
\begin{proof}
This follows by applying $\pi_0$ to the weak equivalences of Proposition \ref{prop:descent}.
\end{proof}

The cobar construction $\mathfrak{C}$ can be viewed as an `up-to-homotopy' right adjoint to the functor
\[ F' : \cat{A}_c \to \tilde{\cat{B}}_{K}. \]
The map $\eta$ of Definition~\ref{def:complete} plays the role of the unit for this adjunction and the completeness condition is therefore a condition that this unit induces an isomorphism on the level of homotopy. We now turn to the dual question of what plays the role of the counit and when this counit induces an isomorphism.

\begin{definition} \label{def:counit}
Let $A$ be a $K$-coalgebra. Then we define a \emph{derived} map of $K$-coalgebras
\[ \epsilon: F'\mathfrak{C}(A) \to A, \]
as follows. It follows from \ref{def:derived-coalgebra-map} that such an $\epsilon$ consists of maps of spaces
\[ \epsilon_n: \Delta^n \to |\Hom_{\cat{B}}(F'\Tot(G'(F'G')^\bullet A),(F'G')^n A)|. \]
We define these to be the composites
\[ \begin{split}
  \Delta^n &\to |\Hom_{\cat{A}_c}(\Tot(G'(F'G')^\bullet A),G'(F'G')^n A)| \\
    &\to |\Hom_{\cat{B}}(F'\Tot(G'(F'G')^\bullet A),F'G'(F'G')^n A)| \\
    &\to |\Hom_{\cat{B}}(F'\Tot(G'(F'G')^\bullet A),(F'G')^n A)|
\end{split} \]
where the first map is the geometric realization of the map adjoint to the projection
\[ \Tot(G'(F'G')^\bullet A) \to \Hom(\Delta^n,G'(F'G')^n A), \]
the second uses the simplicial enrichment of $F'$, and the last involves the counit map
\[ F'G' \to 1. \]
\end{definition}

\begin{lemma}
For a $K$-coalgebra $A$, the maps $\epsilon_n$ of Definition~\ref{def:counit} determine a derived morphism of $K$-coalgebras $\epsilon : F'\mathfrak{C}(A) \to A$.
\end{lemma}
\begin{proof}
We can check that the $\epsilon_n$ commute in a suitable way with the coface maps in the relevant cosimplicial objects.
\end{proof}

The following result now gives us a condition for $\epsilon$ to induce an isomorphism in the homotopy category.

\begin{prop} \label{prop:Tot}
Let $A$ be a $K$-coalgebra that is cofibrant in $\cat{B}$. Then the derived $K$-coalgebra map $\epsilon: F'\mathfrak{C}(A) \to A$ induces an isomorphism in the homotopy category of $K$-coalgebras if and only if the canonical map
\[ F'\Tot(G'(F'G')^\bullet A) \to \Tot(F'G'(F'G')^\bullet A), \]
associated to the simplicial enrichment of $F'$, is a weak equivalence in $\cat{B}$. (Note that the Tot on the left-hand side is calculated in $\cat{A}_c$ and that on the right-hand side is calculated in $\cat{B}$.)
\end{prop}
\begin{proof}
By Proposition~\ref{prop:weqs}, $\epsilon$ induces an isomorphism in the homotopy category if and only if the map $\epsilon_0$ of Definition~\ref{def:counit} induces an isomorphism in the homotopy category of $\tilde{\cat{B}}$, hence if and only if $\epsilon_0$ is a weak equivalence in the model category $\cat{B}$. The map $\epsilon_0$ can be factored as
\[ F'\Tot(G'(F'G')^\bullet A) \to \Tot(F'G'(F'G')^\bullet A) \weq A \]
where the first map is that associated to the simplicial enrichment of $F'$, and the second is the composite $\Tot(F'G'(F'G')^\bullet A) \to F'G'A \to A$ of the projection to $0$-simplices with the counit for the adjunction. That second map is a weak equivalence since the cosimplicial object $(F'G'(F'G')^\bullet A)$ has extra codegeneracies. Therefore the overall map is a weak equivalence if and only if the first factor is, which is the condition specified in the proposition.
\end{proof}

We are now in a position to give conditions that allow us to identity the image of the embedding of Corollary~\ref{cor:descent}. This is the content of our version of the Homotopic Barr-Beck Theorem.

\begin{prop} \label{prop:barr-beck}
Let $(F,G)$ be a Quillen adjunction as in \ref{hyp:adjunction}. Let $\cat{A}_0$ be a collection of objects in $\cat{A}$ and $\cat{C}_0$ a collection of objects in $\tilde{\cat{B}}_{K}$. Suppose that:
\begin{enumerate}
  \item for every object $X \in \cat{A}_0$, $F'cX \in \cat{C}_0$ and $X$ is $F$-complete;
  \item for every $K$-coalgebra $A \in \cat{C}_0$, $u\mathfrak{C}(A) \in \cat{A}_0$ and the canonical enrichment map
  \[ Fu\Tot(G'(F'G')^\bullet A) \to \Tot(FuG'(F'G')^\bullet A) \]
  is a weak equivalence in $\cat{B}$.
\end{enumerate}
Then the functor $\mathsf{Ho}(F) : \mathsf{Ho}(\cat{A}) \to \mathsf{Ho}(\tilde{\cat{B}}_{K})$ of Definition~\ref{def:ho} restricts to an equivalence between the full subcategories of these homotopy categories determined by the collections of objects $\cat{A}_0$ and $\cat{C}_0$, respectively.
\end{prop}
\begin{proof}
By Corollary~\ref{cor:descent}, the first condition implies that $\mathsf{Ho}(F)$ is a fully faithful embedding of the homotopy category of $\cat{A}_0$ into the homotopy category of $\cat{C}_0$. The second condition implies, using Proposition~\ref{prop:Tot}, that every object $A$ of $\cat{C}_0$ is isomorphic in the homotopy category to one of the form $\mathsf{Ho}(F)(X)$ for $X \in \cat{A}_0$ (namely $X = u\mathfrak{C}(A)$) and hence that $\mathsf{Ho}(F)$ is essentially surjective on objects.
\end{proof}

\begin{remark}
We have focused here on \emph{descent} for the Quillen adjunction (\ref{eq:adj-FG}), but our approach can be easily dualized to give a theory of \emph{codescent} in which the category $\cat{B}$ is compared with the category of algebras over the monad $GF$ (or, rather $G'F'$) associated to the adjunction.
\end{remark}

\section{Taylor towers and descent} \label{sec:taylor}

We now turn to Goodwillie calculus and set up a general framework to describe how the Taylor tower of a functor can be recovered from its layers.

\begin{definition}[Spaces and spectra] \label{def:general}
Let $\based$ be the category of based compactly-generated topological spaces, and $\spectra$ the category of $S$-modules of EKMM \cite{elmendorf/kriz/mandell/may:1997}. We refer to the objects of $\spectra$ as \emph{spectra} instead of $S$-modules.

We use the letters $\C$ and $\D$ to denote either of the categories $\based$ and $\spectra$. These are cofibrantly generated pointed simplicial model categories in which every object is fibrant. We use $\Hom(-,-)$ to denote simplicial mapping objects with a subscript $\C$ or $\D$ to denote the underlying category. Since $\C$ and $\D$ are pointed categories, these mapping objects are pointed simplicial sets.

In $\spectra$ we also have internal mapping objects, and we write $\Map(-,-)$ for the spectrum of maps between two spectra.

We write $\Cfin$ for the full subcategory of $\C$ whose objects are the finite cell complexes in $\C$ with respect to the chosen generating cofibrations. In particular, $\Cfin$ is a skeletally small category.
\end{definition}

\begin{definition}[Functor categories] \label{def:functors}
Let $[\Cfin,\D]$ be the category of pointed simplicially-enriched functors $F: \Cfin \to \D$ and simplicial natural transformations. Here \emph{pointed} means that $F(*) = *$. The category $[\Cfin,\D]$ has a projective model structure that is also simplicial, cofibrantly generated and pointed. Note that in each case the resulting model category is Quillen equivalent to a combinatorial model category (for example, by replacing spaces with simplicial sets and EKMM $S$-modules with symmetric spectra based on simplicial sets), and so satisfies condition (1) of Hypothesis~\ref{hyp:adjunction}.

We write $\Nat(F,G)$ for the simplicial set of natural transformations between two pointed simplicial functors $F,G: \Cfin \to \D$. We also write $\hNat(F, G)$ for the simplicial set of derived natural transformations from $F$ to $G$ which can be defined as $\Nat(cF, fG)$, where $c$ and $f$ denote cofibrant and fibrant replacement in $[\Cfin,\D]$, respectively. Note that $\pi_0(\hNat(F, G))$ is the set of maps from $F$ to $G$ in the homotopy category of functors.
\end{definition}

\begin{lemma} \label{lem:homotopy-functor}
Let $F: \Cfin \to \D$ be a pointed simplicial functor. Then $F$ preserves weak equivalences.
\end{lemma}
\begin{proof}
The objects in $\Cfin$ are both fibrant and cofibrant so a weak equivalence $f:X \weq Y$ in $\Cfin$ has a homotopy inverse $g:Y \to X$. Since $F$ is simplicial, it follows that $Ff$ has $Fg$ as a homotopy inverse and so is also a weak equivalence.
\end{proof}

\begin{definition}[Taylor tower] \label{def:taylor}
Let $F: \Cfin \to \D$ be a pointed simplicial functor. Then there exists a sequence of natural transformations
\[ F \to \dots \to P_nF \to P_{n-1}F \to \dots \to P_1F \to * \]
in which $P_nF: \Cfin \to \D$ is \emph{$n$-excisive} (takes strongly homotopy cocartesian $(n+1)$-cubes in $\Cfin$ to homotopy cartesian cubes in $\D$), and $p_n: F \to P_nF$ is, up to homotopy, the initial natural transformation from $F$ to an $n$-excisive functor. This is the \emph{Taylor tower of $F$} (expanded at the trivial object in $\C$). See \cite{goodwillie:2003} for the details of these constructions.
\end{definition}

\begin{definition}[Derivatives] \label{def:coefficients}
The \emph{layers} of the Taylor tower are the functors
\[ D_nF := \hofib(P_nF \to P_{n-1}F). \]
Goodwillie shows in \cite{goodwillie:2003} that these are of the following form:
\[ D_nF(X) \homeq (\Omega^\infty)(\der_nF \smsh ((\Sigma^\infty) X^{\smsh n}))_{h\Sigma_n} \]
where $\der_nF$ is a spectrum with $\Sigma_n$-action, and $\Omega^\infty$, respectively $\Sigma^\infty$, is present if $\D$, respectively $\C$, is equal to $\based$, and absent if equal to $\spectra$. We refer to $\der_nF$ as the \emph{\ord{n} derivative of $F$}. Note that the above formula determines only the homotopy type of the spectrum $\der_nF$. We choose specific models for these spectra later in this paper.
\end{definition}

In previous work \cite{arone/ching:2011}, we have identified additional structure on the derivatives of a functor of based spaces or spectra. This additional structure is stated in the language of operads which we now briefly recall. Note that all our operads, modules and bimodules are in the category of spectra and do not have a \ord{0} term.

\begin{definition}[Symmetric sequences of spectra, operads and their modules] \label{def:operads}
Let $\symseq$ denote the category of \emph{symmetric sequences of spectra}, that is, functors $\mathsf{\Sigma} \to \spectra$ where $\mathsf{\Sigma}$ is the category of nonempty finite sets and bijections. The category $\symseq$ has a cofibrantly generated pointed simplicial model structure in which fibrations and weak equivalences are detected termwise. All objects are fibrant in this model structure.

For a symmetric sequence $A$, we write $A_n$ for the value of $A$ on the finite set $\{1,\dots,n\}$. The spectrum $A_n$ then has an action of the symmetric group $\Sigma_n$ and $A$ is determined, up to isomorphism, by the sequence $A_1,A_2,A_3,\dots$ together with these actions.

An \emph{operad} consists of a symmetric sequence $P$ together with a unit map $S \to P_1$ (where $S$ is the sphere spectrum) and composition maps
\[ P_r \smsh P_{n_1} \smsh \dots \smsh P_{n_r} \to P_{n_1+\dots+n_r} \]
that are suitably equivariant, associative and unital.

For an operad $P$, a \emph{right $P$-module} consists of a symmetric sequence $A$ with suitable right $P$-action maps
\[ A_r \smsh P_{n_1} \smsh \dots \smsh P_{n_r} \to A_{n_1+\dots+n_r} \]
and a \emph{left $P$-module} consists of a symmetric sequence $M$ with suitable left $P$-action maps
\[ P_r \smsh A_{n_1} \smsh \dots \smsh A_{n_r} \to A_{n_1+\dots+n_r}. \]
For two operads $P,P'$ we also have the notion of a $(P,P')$-bimodule, that is a symmetric sequence with a left $P$-action and right $P'$-action that commute.

Each of the categories of operads, left, right and bi- modules has a projective model structure in which weak equivalences and fibrations are detected in the underlying category of symmetric sequences. These model structures are cofibrantly generated, pointed and simplicial, and every object is fibrant. See \cite[Appendix]{arone/ching:2011} for details.
\end{definition}

\begin{definition}[Truncated symmetric sequences] \label{def:truncation}
A symmetric sequence $A$ is said to be \emph{$N$-truncated} if $A_k = *$ for $k > N$, and we say that  $A$ is \emph{bounded} if it is $N$-truncated for some $N$.

We write $A_{\leq N}$ for the \emph{$N$-truncation} of the symmetric sequence $A$ given by setting equal to $*$ the terms $A_k$ for $k > N$. If $A$ is an operad, module or bimodule then $A_{\leq N}$ inherits this structure and there is a natural map $A \to A_{\leq N}$ that preserves the structure. In fact there is a tower (of symmetric sequences, operads, modules or bimodules)
\[ A \to \dots \to A_{\leq N} \to A_{\leq (N-1)} \to \dots \to A_{\leq 1}. \]
\end{definition}

\begin{example} \label{ex:d_*I}
The derivatives of a pointed simplicial functor $F: \Cfin \to \D$ form a symmetric sequence $\der_*F$.

If $I$ is the identity functor on based spaces, then there is a model for the symmetric sequence of derivatives of $I$ that has an operad structure. This model is formed by the Spanier-Whitehead duals of the partition poset complexes, see \cite{ching:2005}. We denote a cofibrant replacement for this operad by $\der_*I$.

In \cite{arone/ching:2011} the authors constructed models for functors to or from based spaces that are modules over the operad $\der_*I$. For $F: \finbased \to \spectra$, $\der_*F$ is a right $\der_*I$-module, for $F: \finspec \to \based$, $\der_*F$ is a left $\der_*I$-module, and for $F: \finbased \to \based$, $\der_*F$ is a $(\der_*I,der_*I)$-bimodule. These structures allow us to define, in each case, a functor
\[ \der_*: [\Cfin,\D] \to \cat{M} \]
where $\cat{M}$ is the appropriate category of modules or bimodules. When $\C = \D = \spectra$, we take $\cat{M}$ to be the category of symmetric sequences.
\end{example}

One of the key observations of this paper is that this functor $\der_*$ has a right adjoint. This is well-known in the case $\D = \spectra$ but is new when $\D = \based$. We leave the construction of such adjunctions to later sections that deal with each combination of $\C$ and $\D$ in turn. The anxious reader may turn to Proposition \ref{prop:der-sp} (for the case $\D = \spectra$) or Proposition \ref{prop:der-top} (for the case $\D = \based$).

For the remainder of this section we make the following assumption about the existence of the right adjoint to $\der_*$.

\begin{hypothesis} \label{hyp:coefficients}
Let $\C$ and $\D$ each be either of the categories $\based$ or $\spectra$, and let $\cat{M}$ be some category of symmetric sequences, modules or bimodules, as in Definition~\ref{def:operads}. We assume given a simplicial Quillen adjunction
\[ \der_*: [\Cfin,\D] \rightleftarrows \cat{M} : \Phi \]
such that, for cofibrant $F$, the symmetric sequence $\der_*F$ is a natural model for the sequence of Goodwillie derivatives of $F$.
\end{hypothesis}

\begin{remark} \label{rem:M}
As indicated in Example~\ref{ex:d_*I} there is a natural choice for the category $\cat{M}$ of (bi)modules for each combination of choices of the categories $\C$ and $\D$. However, it turns out that the choice of right module structure is irrelevant. For example, when $\D = \spectra$, we can choose $\cat{M}$ to be the category of symmetric sequences regardless of whether $\C$ is $\based$ or $\spectra$. In fact, when we analyze this case in Section~\ref{sec:sp}, it will be convenient to do exactly that. Similarly, when $\D = \based$, we can take $\cat{M}$ to be the category of left $\der_*I$-modules for either choice of $\C$. The underlying reason for this is that the calculation of colimits in $\cat{M}$ is not affected by a choice of right module structure.
\end{remark}

The adjunction $(\der_*,\Phi)$ satisfies the conditions of Hypothesis~\ref{hyp:adjunction} and so we are in a position to apply the descent theory of Section~\ref{sec:descent}. Recall that this involves choosing a simplicial Quillen equivalence
\[ u: [\Cfin,\D]_c \rightleftarrows [\Cfin,\D] : c \]
such that every object in $[\Cfin,\D]_c$ is cofibrant, and then provides a comparison between the category $[\Cfin,\D]$ and the category of $K$-coalgebras, where $K$ is the comonad $\der_* uc \Phi$.

In order to simplify the exposition, we henceforth drop the functors $u$ and $c$ from our notation. The assiduous reader will work out all the places where these should be inserted.

This section has two main results. Theorem~\ref{thm:main} identifies the functors $F: \Cfin \to \D$ that are $\der_*$-complete in the sense of Definition~\ref{def:complete-F}, and hence those which can be recovered from the $K$-coalgebra structure on their derivatives. Theorem~\ref{thm:n-excisive} applies our Homotopic Barr-Beck Theorem (\ref{prop:barr-beck}) to deduce that there is an equivalence between the homotopy category of $N$-excisive functors and that of $N$-truncated $K$-coalgebras. In the proofs of both of these results, the following lemma concerning the right adjoint $\Phi$ plays a key role.

\begin{lemma} \label{lem:phi}
Let $A$ be an $N$-truncated object in $\cat{M}$. Then the functor $\Phi A$ is $N$-excisive.
\end{lemma}
\begin{proof}
We use the following characterization of $N$-excisive functors:
\begin{claim*}
A functor $G: \Cfin \to \D$ is $N$-excisive if and only if any natural transformation $F \to G$ factors, in the homotopy category of $[\Cfin,\D]$, as
\[ F \arrow{e,t}{p_N} P_NF \to G. \]
\end{claim*}
\begin{proof}[Proof of Claim]
The `only if' direction is part of the universal property enjoyed by $P_N$ (see \cite{goodwillie:2003}). For the `if' direction, suppose that $G$ is a functor such that every natural transformation $F \to G$ factors as claimed. Then in particular, the identity natural transformation $G \to G$ factors, in the homotopy category, as
\[ G \arrow{e,t}{p_N} P_NG \arrow{e,t}{\alpha} G. \]
Thus $p_N$ has a left inverse in the homotopy category. But composing the above sequence again with $p_N$, and applying the uniqueness part of the universal property enjoyed by $P_N$, it follows that $p_N \circ \alpha$ is the identity on $P_NG$. Thus $p_N$ also has a right inverse in the homotopy category, so is an equivalence. Hence $G$ is $N$-excisive.
\end{proof}
Now suppose $A$ is $N$-truncated and apply the claim to $\Phi A$. A natural transformation $F \to \Phi A$ corresponds under the adjunction of \ref{hyp:coefficients} to a map $\der_*F \to A$ in $\cat{M}$. Since $A$ is $N$-truncated, this map factors, in the homotopy category of $\cat{M}$, as
\[ \der_*(F) \to \der_*(P_NF) \to A \]
which corresponds, under the Quillen adjunction of \ref{hyp:coefficients} to a factorization
\[ F \to P_NF \to \Phi A. \]
By the Claim, $\Phi A$ is $N$-excisive.
\end{proof}

\begin{definition} \label{def:K}
Under the assumptions of \ref{hyp:coefficients}, let $K$ be the comonad on the category $\cat{M}$ given by
\[ KM := \der_*\Phi M. \]
(Recall that we really mean $K = \der_* uc \Phi$ but are dropping $u$ and $c$ from the notation.) Then, for any pointed simplicial functor $F: \Cfin \to \D$, the symmetric sequence $\der_*F$ (or, really, $\der_*ucF$) has the structure of a $K$-coalgebra in $\cat{M}$. The functor $F$ is \emph{$\der_*$-complete} if the map
\[ F \to \Tot(\Phi K^{\bullet}\der_*F) \]
is a weak equivalence, that is, if $F$ can be recovered from its derivatives together with their $K$-coalgebra structure.
\end{definition}

\begin{theorem} \label{thm:main}
Let $\C$ and $\D$ be either $\based$ or $\spectra$ and $\der_*$ as in \ref{hyp:coefficients}. For a pointed simplicial functor $F: \Cfin \to \D$ the $\der_*$-completion map
\[ \eta: F \to \Tot(\Phi K^{\bullet}\der_*F) \]
is, up to homotopy, a retract of the map
\[ p_{\infty}: F \to \holim_n P_nF \]
associated to the Taylor tower of $F$. In particular, if $F$ is equivalent to the limit of its Taylor tower, then $F$ is $\der_*$-complete.
\end{theorem}
\begin{proof}
Goodwillie constructs, in \cite[2.2]{goodwillie:2003}, a functor
\[ R_k: [\Cfin,\D] \to [\Cfin,\D] \]
such that, for any $F \in [\Cfin,\D]$, $R_kF$ is $k$-homogeneous, and there is a natural fibration sequence of functors:
\[ P_k F \to P_{k-1}F \to R_kF. \]
In particular, $R_kF$ is a natural delooping of $D_kF$.

Our proof of the theorem is based on the following commutative diagrams in which each column is a fibration sequence of functors
\begin{equation} \label{eq:thm-main} \begin{diagram}
  \node{P_k(F)} \arrow{s} \arrow{e} \node{\Tot(P_k(\Phi K^\bullet \der_*F))} \arrow{s} \\
  \node{P_{k-1}(F)} \arrow{e} \arrow{s} \node{\Tot(P_{k-1}(\Phi K^\bullet \der_*F))} \arrow{s} \\
  \node{R_k(F)} \arrow{e} \node{\Tot(R_k(\Phi K^\bullet \der_*F))}
\end{diagram} \end{equation}
We show first, by induction on $k$, that all the horizontal maps in these diagrams are equivalences. For this, it is sufficient, by the induction hypothesis, to show that the bottom horizontal map is an equivalence.

The key point here is that the functor $R_k$ factors via $\der_*$. We define a functor $\Psi_k: \cat{M} \to [\Cfin,\D]$ by
\[ \Psi_k(A)(X) := [\Omega^\infty](A_k \smsh X^{\smsh k})_{h\Sigma_k}. \]
That is, $\Psi_k$ is the functor that recovers $D_kF$ from $\der_*F$. There is then a natural equivalence
\[ R_kG \homeq R_k\Psi_k\der_*G. \]
The bottom horizontal map of (\ref{eq:thm-main}) is therefore equivalent to the map
\[ R_k\Psi_k(\der_*F) \to \Tot(R_k\Psi_k(K^{\bullet+1}\der_*F)). \]
This is the coaugmentation map for a cosimplicial object with extra codegeneracies (given by the counit map for the extra copy of $K$ that now appears) and so is a simplicial homotopy equivalence, hence a weak equivalence in $[\Cfin,\D]$.

We deduce then by induction that all the horizontal maps in (\ref{eq:thm-main}) are natural weak equivalences.

The claim that $\eta$ is a retract, up to homotopy, of $p_{\infty}$ then follows from the existence of the following commutative diagram:
\[ \begin{diagram}
  \node{F} \arrow{e,t}{\eta} \arrow{s,l}{p_{\infty}}
    \node{\Tot(\Phi K^{\bullet}\der_*F)} \arrow{e,t}{\sim} \arrow{s}
    \node{\holim_n \Tot(\Phi(K^{\bullet}\der_*F)_{\leq n})} \arrow{s,r}{\sim} \\
  \node{\holim_n P_n(F)} \arrow{e,t}{\sim}
    \node{\holim_n \Tot P_n(\Phi K^{\bullet}\der_*F)} \arrow{e}
    \node{\holim_n \Tot P_n(\Phi (K^{\bullet}\der_*F)_{\leq n})}
\end{diagram} \]
Since we have shown that the horizontal maps in (\ref{eq:thm-main}) are equivalences, the bottom-left horizontal map is an equivalence. The top-right horizontal map is an equivalence because the homotopy limit commutes with both totalization and the right Quillen functor $\Phi$, and because $\holim_n A_{\leq n} \homeq A$ for any $A \in \cat{M}$. The right-hand vertical map is an equivalence by Lemma \ref{lem:phi}.

Note that we do not claim the bottom-right horizontal map in the above diagram to be a weak equivalence in general. If the Taylor tower of $F$ converges, that is $p_{\infty}$ is an equivalence, then all maps in this diagram are weak equivalences.
\end{proof}

It follows from Theorem \ref{thm:main} that if $F$ is $N$-excisive for some $N$, then $F$ can be recovered from the $K$-coalgebra $\der_*F$. More generally, for any $F$, we can recover $P_NF$ from the $N$-truncation $\der_{\leq N}F$, together with its $K$-coalgebra structure.

\begin{definition} \label{def:truncated-coalgebra}
We define a comonad $K_{\leq N}$ on the subcategory of $N$-truncated objects in $\cat{M}$ by setting
\[ K_{\leq N}(A) := (KA)_{\leq N}. \]
We refer to a coalgebra over $K_{\leq N}$ is an \emph{$N$-truncated $K$-coalgebra}. Notice that if $A$ is a $K$-coalgebra in $\cat{M}$, then the $N$-truncation $A_{\leq N}$ is an $N$-truncated $K$-coalgebra.

When $A$ is $N$-truncated object in $\cat{M}$, Lemma \ref{lem:phi} tells us that
\[ (KA)_r \homeq * \]
for $r > N$ and so $K_{\leq N}A \homeq KA$. We therefore usually drop the distinction and just write $K$ instead of $K_{\leq N}$.
\end{definition}

\begin{corollary} \label{cor:truncated}
For $F \in [\Cfin,\D]$, we have
\[ P_nF \homeq \Tot(\Phi K^{\bullet} \der_{\leq n}F). \]
With respect to these equivalences, the map $P_nF \to P_{n-1}F$ is induced by the truncation map
\[ \der_{\leq n}F \to \der_{\leq(n-1)}F. \]
\end{corollary}
\begin{proof}
Theorem \ref{thm:main} tells us that $P_nF$ is $\der_*$-complete which means that
\[ P_nF \homeq \Tot(\Phi K^{\bullet} \der_*(P_nF)). \]
It is therefore sufficient to notice that there is a levelwise equivalence of cosimplicial objects
\[ \Phi K^{\bullet} \der_{\leq n}F \homeq \Phi K^{\bullet} \der_*(P_nF). \]
\end{proof}

\begin{corollary} \label{cor:nat}
For $F,G \in [\Cfin,\D]$ with $G$ analytic in the sense of Goodwillie~\cite{goodwillie:1991}, there are natural weak equivalences of simplicial sets
\[ \hNat(F,\holim_n P_nG) \homeq \widetilde{\Hom}_{K}(\der_*F,\der_*G) \]
where the left-hand side is the simplicial set of derived natural transformations of Definition~\ref{def:functors} and the right-hand side is the simplicial set of derived $K$-coalgebra maps of Definition~\ref{def:derived-ss-hom}.
\end{corollary}
\begin{proof}
When $G$ is analytic, the functor $\holim_n P_nG$ has Taylor tower equivalent to that of $G$ and so is $\der_*$-complete by Theorem~\ref{thm:main}. The result is then Proposition \ref{prop:descent} in this context.
\end{proof}

Corollary \ref{cor:nat} now implies in particular that $\der_*$ determines a fully-faithful embedding of the homotopy theory of $N$-excisive functors into that of $N$-truncated $K$-coalgebras. We now show that this embedding is in fact an equivalence by verifying the hypothesis of our homotopic Barr-Beck Theorem (Proposition \ref{prop:barr-beck}).

\begin{lemma} \label{lem:barr-beck-derivs}
For any $K$-coalgebra $A$ such that $A_k \homeq *$ for $k > N$, the map
\[ \der_*\Tot(\Phi K^\bullet A) \to \Tot(\der_*\Phi K^\bullet A) \]
is a weak equivalence of symmetric sequences.
\end{lemma}
\begin{proof}
This is similar to the proof of \ref{thm:main}. Consider the following diagram in which the columns are fibration sequences. (It is important here that $\der_*$ preserves fibration sequences.)
\[ \begin{diagram}
  \node{\der_*\Tot(P_{k}\Phi K^{\bullet}A)} \arrow{e} \arrow{s} \node{\Tot(\der_*P_{k}\Phi K^{\bullet}A)} \arrow{s} \\
  \node{\der_*\Tot(P_{k-1}\Phi K^{\bullet}A)} \arrow{e} \arrow{s} \node{\Tot(\der_*P_{k-1}\Phi K^{\bullet}A)} \arrow{s} \\
  \node{\der_*\Tot(R_{k}\Phi K^{\bullet}A)} \arrow{e} \node{\Tot(\der_*R_{k}\Phi K^{\bullet}A)}
\end{diagram} \]
where $R_k$ is a natural delooping of $D_k$. Since $\Phi K^{\bullet}A$ is $N$-excisive by Lemma \ref{lem:phi}, it is sufficient, by induction on $k$, to show that the bottom-horizontal map is an equivalence for any $k$. But we have
\[ R_k = R_k \Psi_k \der_* \]
so this bottom map takes the form
\[ \der_*\Tot(R_k \Psi_k K^{\bullet+1}A) \to \Tot(\der_*R_k\Psi_k K^{\bullet+1}A). \]
These cosimplicial objects both have extra codegeneracies and this map is therefore equivalent to the identity on
\[ \der_*R_k\Psi_kA \]
so is a weak equivalence.
\end{proof}

\begin{theorem} \label{thm:n-excisive}
There is an equivalence between the homotopy theory of $N$-excisive pointed simplicial functors $F: \Cfin \to \D$ and that of $N$-truncated $K$-coalgebras. Letting $N \to \infty$ we obtain an equivalence between the homotopy category of all polynomial functors and that of all bounded $K$-coalgebras.
\end{theorem}
\begin{proof}
This now follows from \ref{lem:barr-beck-derivs} and \ref{prop:barr-beck}.
\end{proof}

We conclude this section with a useful strengthening of Lemma~\ref{lem:phi}. Suppose that $A$ is an $N$-truncated object of $\cat{M}$. Then not only is $\Phi A$ an $N$-excisive functor, but the counit map $\der_* \Phi A \to A$ induces an equivalence of \ord{N} terms, that is, of the highest non-trivial terms. We also have a dual result, involving the unit of the adjunction $(\der_*, \Phi)$.
\begin{proposition}\label{prop:unit}
Suppose $F$ is an $N$-excisive functor. Then the unit map
\[
F \longrightarrow \Phi\der_* F
\]
induces an equivalence of \ord{N} derivatives.
\end{proposition}

\begin{proposition}\label{prop:counit}
Suppose $A$ is an $N$-truncated object of $\cat{M}$. The counit map
\[
\der_*  \Phi A \longrightarrow A
\]
induces an equivalence of \ord{N} terms of these symmetric sequences.
\end{proposition}
To prove the two propositions, we need a few preparatory lemmas that help us to detect $D_N$-equivalences.
\begin{lemma}
Let $F, G\in [\Cfin, \D]$ be functors and assume that $G$ is $N$-excisive. The map $p_N \colon F \to P_NF$ induces a weak homotopy equivalence
\[
\hNat(P_NF, G) \longrightarrow \hNat(F, G).
\]
\end{lemma}
\begin{proof}
Again we recall that the universal property of $P_N$ is that  any natural transformation from $F$ to an $N$-excisive functor factors through $p_N$, uniquely in the homotopy category of functors. It follows that the map $\hNat(P_NF, G) \longrightarrow \hNat(F, G)$ is an isomorphism on $\pi_0$. To prove that it is an isomorphism on every homotopy group, note that for every $i$, $\pi_i(\hNat(F, G))\cong \pi_0(\hNat(F, \Omega^i G))$ and that if $G$ is $N$-excisive then so is $\Omega^iG$.
\end{proof}
\begin{corollary}\label{cor:trivial}
Let $F, G\in [\Cfin, \D]$. Assume that $P_NF \simeq *$ and that $G$ is $N$-excisive. Then
\[
\hNat(F, G) \simeq *.
\]
\end{corollary}
\begin{lemma}\label{lem:detect}
Let $F, G$ be $N$-excisive. Let $\eta\colon F\longrightarrow G$ be a natural transformation. The induced natural transformation $D_N(\eta)\colon D_N F\longrightarrow D_NG$ is a weak equivalence if and only if the induced map
\[
\hNat(L, \eta)\colon \hNat(L, F)\longrightarrow \hNat(L, G)
\]
is a weak equivalence for all $N$-homogeneous functors $L$.
\end{lemma}
\begin{proof}
For every $L$ and $F$ there is a fibration sequence
\[
\hNat(L, D_N F)\longrightarrow \hNat(L, P_N F) \longrightarrow \hNat(L, P_{N-1} F).
\]
If $L$ is $N$-homogeneous then, by Corollary~\ref{cor:trivial}, $\hNat(L, P_{N-1} F)\simeq *$. It follows that if $F$ is $N$-excisive there is a natural equivalence
\[
\hNat(L, D_N F)\simeq \hNat(L, F).
\]
The same holds for $G$. Thus, if $D_N(\eta)\colon D_NF \longrightarrow D_NG$ is an equivalence then so is $\hNat(L, D_N(\eta))$, and therefore so is $\hNat(L, \eta)$, as required.

Conversely suppose $\hNat(L, \eta)$ is an equivalence for every $N$-homogenous functor $L$. Then so is $\hNat(L, D_N(\eta))$. It follows by a standard categorical argument that the map $D_N(\eta)$ is an isomorphism in the homotopy category of $N$-homogeneous functors, which is a sub-category of the homotopy category of all functors. Thus $D_N(\eta)$ is a weak equivalence.
\end{proof}
Recall that $(\der_*, \Phi)$ is an adjoint pair of functors between $[\Cfin, \D]$ and $\cat{M}$. Given two objects $A, B\in \cat{M}$ let $\hM(A, B)$ be the derived space of morphisms from $A$ to $B$. For example, it can be defined as the space of morphisms from a cofibrant replacement of $A$ to a fibrant replacement of $B$. (It is true that all objects are fibrant in $\cat{M}$, so taking fibrant replacement is not necessary here.)
\begin{lemma}\label{lem:derivative detects}
As before, let $L$ be an $N$-homogeneous functor and let $F$ be $N$-excisive. Then $\der_*$ induces an equivalence
\[
\hNat(L, F) \stackrel{\simeq}{\longrightarrow} \hM(\der_* L, \der_* F).
\]
\end{lemma}
\begin{proof}
Consider the diagram
\[ \begin{diagram}
  \node{\hNat(L, D_NF)} \arrow{e} \arrow{s,l}{\der_*} \node{\hNat(L, F)} \arrow{s,r}{\der_*} \\
  \node{\hM(\der_*L, \der_* D_NF)} \arrow{e}  \node{\hM(\der_*L, \der_*F)}
\end{diagram} \]
Our goal is to prove that the right vertical map is an equivalence. We do this by showing that the other three maps are equivalences. We saw that the top horizontal map is an equivalence in the proof of Lemma~\ref{lem:detect}. The fact that the left vertical map is an equivalence is a consequence of Goodwillie's theorem that $\der_*$ induces an equivalence between the category of $N$-homogeneous functors and the category of spectra with an action of $\Sigma_N$. For detailed calculations of this kind see for example~\cite{arone/dwyer/lesh:2008}. Finally the fact that the bottom horizontal map is an equivalence is an easy calculation: because $\der_*L$ is concentrated in degree $N$, and $\der_*F$ is truncated at $N$, morphisms from $\der_* L$ to $\der_* F$ depend only on $\der_N F= \der_* D_N F$.
\end{proof}
\begin{proof}[Proof of Proposition~\ref{prop:unit}]
Let $L\in [\Cfin, \D]$ be an $N$-homogeneous functor. Let $F$ be an $N$-excisive functor. By Lemma~\ref{lem:phi}, $\Phi\der_* F$ is $N$-excisive. By Lemma~\ref{lem:detect} it is enough to show that the map
\[
\hNat(L, F) \longrightarrow \hNat(L, \Phi\der_* F)
\]
is a weak equivalence. Consider the commutative diagram
\[ \begin{diagram}
  \node{\hNat(L, F)} \arrow{e} \arrow{s,l}{\der_*} \node{\hNat(L, \Phi\der_* F)} \arrow{s} \\
  \node{\hM(\der_*L, \der_* F)} \arrow{e,t}{=} \node{\hM(\der_*L, \der_*F)}
\end{diagram}. \]
Here the right vertical map is an equivalence expressing the adjunction between $\der_*$ and $\Phi$. The left vertical map is an equivalence by Lemma~\ref{lem:derivative detects}. The bottom horizontal map is the identity. It follows that the top horizontal map is an equivalence.
\end{proof}
\begin{proof}[Proof of Proposition~\ref{prop:counit}]
First assume that $A$ is concentrated in degree $N$. There exists a functor $L$ such that $\der_*L=A$. Indeed, $L(X)=(\Omega^\infty)(A_N \smsh X^{\smsh N})_{h\Sigma_N}$. Consider the maps in $\cat{M}$
\[
A=\der_*L \longrightarrow \der_*\Phi \der_* L \longrightarrow \der_*L = A.
\]
Here the first map is $\der_*$ applied to the unit of the adjunction $L\to \Phi\der_* L$. By proposition~\ref{prop:unit} this map induces an equivalence $A_N=\der_N L \stackrel{\simeq}{\longrightarrow} \der_N \Phi\der_* L$. On the other hand, the composed map is the identity $A = A$. It follows that the second map induces an equivalence $\der_N \Phi A \stackrel{\simeq}{\longrightarrow} A_N.$

Now let $A$ be any $N$-truncated object of $\cat{M}$. Let $A_{<N}$ be the truncation of $A$ at $N-1$ and let $A_N$ be the \ord{N} term of $A$, considered as a symmetric sequence concentrated in degree $N$. There is a fibration sequence in $\cat{M}$
\[
A_N \longrightarrow A \longrightarrow A_{<N}.
\]
Both functors $\Phi$ and $\der_*$ preserve fibration sequences. Therefore we have a diagram in $\cat{M}$ where both rows are fibration sequences
\[ \begin{diagram}
  \node{\der_* \Phi A_N} \arrow{e} \arrow{s}\node{\der_* \Phi A} \arrow{e}\arrow{s} \node{\der_*\Phi A_{<N}}\arrow{s}\\
  \node{A_N} \arrow{e} \node{A} \arrow{e} \node{A_{<N}}
\end{diagram}. \]
Consider the restriction of this diagram of the $N$-terms of all the sequences. Then the right vertical map is a map between contractible objects and therefore is an equivalence. The left vertical map is an equivalence by the special case of the proposition that we proved already. It follows that the middle map is an equivalence.
\end{proof}

This concludes our description of the general theory. We now turn to more specific cases and discuss how the adjunction of Hypothesis \ref{hyp:coefficients} can be constructed, and what the $K$-coalgebra structure amounts to, in each case.

\section{Functors with values in spectra} \label{sec:sp}

We now focus on functors $F: \Cfin \to \spectra$ where $\C$ is either $\based$ or $\spectra$. Our first goal is to establish the existence of an adjunction of the form described in Hypothesis \ref{hyp:coefficients}, so that the theory of the previous section can be applied. We get this adjunction by showing that the derivatives functor $\der_*$ can be obtained by left Kan extension from its values on representable functors. Note that in either of the cases $\C = \based,\spectra$ we take the category $cat{M}$ of \ref{hyp:coefficients} to be just the category $\symseq$ of symmetric sequences.

\begin{definition} \label{def:reps}
For $X \in \Cfin$, we let $R_X: \Cfin \to \sset$ denote the representable functor defined by the simplicial enrichment of $\C$, that is,
\[ R_X(Y) := \Hom_{\C}(X,Y). \]
We fix models for the derivatives of the functors $\Sigma^\infty R_X : \Cfin \to \spectra$. That is, we fix a simplicially-enriched functor
\[ \der_*(\Sigma^\infty R_\bullet): (\Cfin)^{op} \to \symseq \]
whose value at $X$ is a model for the symmetric sequence of derivatives of $\Sigma^\infty R_X$. We also choose this such that $\der_*(\Sigma^\infty R_X)$ is a cofibrant symmetric sequence for each $X \in \Cfin$.
\end{definition}

\begin{definition} \label{def:der-sp}
We define $\der_*F$ for arbitrary $F \in [\Cfin,\spectra]$ by
\[ \der_*F := F(X) \smsh_{X \in \Cfin} \der_*(\Sigma^\infty R_X). \]
This is an enriched coend calculated over the simplicial category $\Cfin$ and $\smsh$ here denotes the termwise smash product of a symmetric sequence with a spectrum. This definition produces a simplicial functor
\[ \der_*: [\Cfin,\spectra] \to \symseq. \]
Note that when $F = \Sigma^\infty R_X$, this definition is canonically isomorphic (by an enriched Yoneda Lemma) to the choice of $\der_*(\Sigma^\infty R_X)$ in Definition \ref{def:reps} so our notation is consistent. We can also view $\der_*$ as the enriched left Kan extension of $\der_*(\Sigma^\infty R_\bullet)$ along the functor
\[ (\Cfin)^{op} \to [\Cfin,\spectra]; \quad X \mapsto \Sigma^\infty R_X. \]
\end{definition}

The following proposition is due to Peter Oman~\cite{oman:2010} in the case $\C=\based$.

\begin{proposition} \label{prop:der-sp}
For cofibrant $F \in [\Cfin,\spectra]$, the symmetric sequence $\der_*F$ is naturally weakly equivalent to the symmetric sequence of derivatives of $F$.
\end{proposition}
\begin{proof}
Let $\der_*^G F$ be a natural model for the actual derivatives of $F$. There is an assembly map $\der_*F \to \der_*^G F$ that is natural in $F$. (More precisely, it could be a zigzag of maps.) The claim is that this is an equivalence for all cofibrant $F$. By the Yoneda Lemma the claim is true when $F$ is of the form $\Sigma^\infty R_X$ and any cofibrant functor can be built, up to weak equivalence, as a homotopy colimit of these representable functors (for example, by the small object argument in the cofibrantly-generated model category $[\Cfin,\spectra]$). Therefore, since both functors $\der_*$ and $\der_*^G$  preserve homotopy colimits of, and weak equivalences between, cofibrant functors, the claim is true for all cofibrant $F$.
\end{proof}

\begin{definition} \label{def:phi}
The definition of the functor $\der_*: [\Cfin,\spectra] \to \symseq$ as a left Kan extension ensures that it has a right adjoint. This right adjoint is the simplicial functor $\Phi: \symseq \to [\Cfin,\spectra]$ given by
\[ \Phi(A)(X) := \Map_{\mathsf{\Sigma}}(\der_*(\Sigma^\infty R_X),A) \isom \prod_{n} \Map(\der_n(\Sigma^\infty R_X),A_n)^{\Sigma_n}. \]
\end{definition}

\begin{proposition} \label{prop:derphisp}
The adjunction
\[ \der_*: [\Cfin,\spectra] \leftrightarrows \symseq : \Phi \]
satisfies the conditions of Hypothesis \ref{hyp:coefficients}.
\end{proposition}
\begin{proof}
We have already seen that $\der_*$ provides a model for the derivatives of a cofibrant functor. The functor $\Phi$ preserves fibrations and trivial fibrations because these are detected objectwise in both $[\Cfin,\spectra]$ and $\symseq$. Therefore $(\der_*,\Phi)$ is a Quillen adjunction.
\end{proof}

The results of Section \ref{sec:taylor} now apply. In particular, we conclude that the derivatives of a functor $F: \Cfin \to \spectra$ possess the structure of a $K$-coalgebra where $K$ is the comonad $\der_* \Phi$ on the category of symmetric sequences. The Taylor tower of $F$ can be recovered from this $K$-coalgebra structure by the formulas in Corollary \ref{cor:truncated}.

In sections \ref{sec:specspec} and \ref{sec:topspec} below, we give an explicit (but only up to homotopy) description of the comonad $K$ in each of the cases where $\C$ is $\spectra$ and $\based$. There are, however, some common features to these descriptions which we now outline.

\begin{lemma} \label{lem:K-limits}
The comonad $K: \symseq \to \symseq$ preserves finite homotopy limits. Thus, finite homotopy limits of $K$-coalgebras can be calculated in the underlying category of symmetric sequences, that is, termwise.
\end{lemma}
\begin{proof}
The right adjoint $\Phi$ preserves all homotopy limits. The \ord{n} derivative functor $\der_n$ can be constructed from $F$ by applying $D_n$, taking the \ord{n} cross-effect, and evaluating at the $0$-sphere spectrum. Each of these constructions preserves finite homotopy limits, so $\der_*$ does too.
\end{proof}

Now any symmetric sequence $A$ is equivalent to the product of its individual terms (viewed as one-term symmetric sequences). It therefore follows from \ref{lem:K-limits} that when $K$ is applied to a bounded symmetric sequence, the result splits up into terms corresponding to the pieces of the input. This motivates the following definition.

\begin{definition} \label{def:derAn}
For a $\Sigma_n$-spectrum $A_n$, we also denote by $A_n$ the symmetric sequence consisting of $A_n$ in the \ord{n} term, with the trivial spectrum in all other terms. Then, for $r \leq n$, we write
\[ K_r A_n := (KA_n)_r = \der_r \left( X \mapsto \Map(\der_n(\Sigma^\infty R_X),A_n)^{\Sigma_n} \right). \]
This is the \ord{r} term in the symmetric sequence given by applying $K$ to $A_n$. Notice that the corresponding spectra are trivial for $r > n$ by Lemma \ref{lem:phi}. For an $N$-truncated symmetric sequence $A$, we then have a $\Sigma_r$-equivariant equivalence of spectra
\[ K(A)_r \homeq \prod_{n = r}^{N} K_r(A_n). \]
For $r \leq s \leq n$ the comonad structure map $K \to KK$ determines a map
\[ \delta_{r,s}: K_r A_n \to K_r K_s A_n \]
and the counit $K \to 1_{\symseq}$ determines
\[ \epsilon_r: K_r A_r \to A_r. \]
\end{definition}

We can now describe a $K$-coalgebra structure in terms of these individual maps.

\begin{lemma} \label{lem:derAn}
Let $A$ be a $K$-coalgebra with structure map $\theta: A \to KA$. Then, for positive integers $r \leq n$, $\theta$ determines $\Sigma_r$-equivariant maps
\[ \theta_{r,n}: A_r \to K_r A_n \]
such that the diagrams
\[ \begin{diagram}
  \node{A_r} \arrow{e,t}{\theta_{r,s}} \arrow{s,l}{\theta_{r,n}} \node{K_r A_s} \arrow{s,r}{\theta_{s,n}} \\
  \node{K_r A_n} \arrow{e,t}{\delta_{r,s}} \node{K_r K_s A_n}
\end{diagram} \]
and
\[ \begin{diagram}
  \node{A_r} \arrow{e,t}{\theta_{r,r}} \arrow{se,=} \node{K_r A_r} \arrow{s,r}{\epsilon_r} \\
  \node[2]{A_r}
\end{diagram} \]
commute.
\end{lemma}
\begin{proof}
The map $\theta_{r,n}$ is the composite
\[ A_r \arrow{e,t}{\theta} K(A)_r \to K(A_n)_r = K_r A_n \]
where the second map is induced by the projection from a symmetric sequence $A$ to $A_n$ considered as a symmetric sequence concentrated in the \ord{n} term. The given diagrams then follow from the coassociativity and counit axioms for a $K$-coalgebra in Definition \ref{def:comonad}.
\end{proof}

\begin{corollary} \label{cor:derAn}
Let $F: \Cfin \to \spectra$ be a pointed simplicial functor. Then there are $\Sigma_r$-equivariant maps
\[ \theta_{r,n}: \der_rF \to K_r \der_nF \]
for $r \leq n$, that are natural in $F$, such that the diagrams in Lemma \ref{lem:derAn} commute.
\end{corollary}


\begin{lemma} \label{lem:K-counit}
The counit map $\epsilon_r: K_r A_r \to A_r$ of \ref{def:derAn} is a weak equivalence for any $\Sigma_r$-spectrum $A_r$.
\end{lemma}
\begin{proof}
This is Proposition~\ref{prop:counit} applied to $A=A_r$ (with the role of $N$ in that proposition played here by $r$).
\end{proof}

Lemma \ref{lem:K-counit} removes the need to consider the maps $\theta_{r,r}$ as part of the coalgebra structure. The $K$-coalgebra structure on $\der_*F$ for $F \in [\Cfin,\spectra]$ is determined by the maps $\theta_{r,n}$ for $r < n$ subject only to the coassociativity conditions in the first diagram of Lemma \ref{lem:derAn}.

As a consequence of Lemma \ref{lem:K-counit} we get a description of the $\mathbb{E}^1$-page of the Bousfield-Kan spectral sequence associated to the cosimplicial space that calculates the simplicial sets $\hNat(F,G)$.

\begin{proposition} \label{prop:BKSS}
Let $F,G \in [\Cfin,\spectra]$ be pointed simplicial functors and let $E$ be any spectrum. Then there is a spectral sequence $\mathbb{E}^*_{*,*}$ with
\[ \mathbb{E}^1_{-s,t} \isom \bigoplus_{1 \leq r_0 < \dots < r_s} E_t \left( \Hom(\der_{r_0}F,K_{r_0} K_{r_1} \dots K_{r_{s-1}}\der_{r_s}G)^{\Sigma_{r_0}} \right) \]
and differential $d^1$ equal to the alternating sum of maps induced by the comonad structure maps $\delta_{r_i,r_{i+1}}$ and the $K$-coalgebra structure maps $\theta_{r_0,r_1}$ for $\der_*F$ and $\theta_{r_{s-1},r_s}$ for $\der_*G$.

If $G$ is $N$-excisive for some $N$, then the spectral sequence collapses at the $\mathbb{E}^n$-page and converges to the $E$-homology
\[ E_{t-s} \hNat(F,G). \]
\end{proposition}
\begin{proof}
The decomposition of the comonad $K$ via the constructions $K_r$ allows us to write down a Reedy fibrant model for the cosimplicial simplicial set
\begin{equation} \label{eq:spseq} \Hom_{\mathsf{\Sigma}}(\der_*F,K^{\bullet} \der_*G). \end{equation}
We do this by replacing $K^s(\der_*G)$ with the equivalent object
\[ \prod_{r_0 \leq \dots \leq r_s} K_{r_0} \dots K_{r_{s-1}} \der_{r_s}G \]
where $K_r$ is now defined to be the identity on a $\Sigma_r$-spectrum.

The codegeneracies are now inclusions and the coface maps given by the relevant comonad and coalgebra structure maps $\delta$ and $\theta$. It is easy to see this is Reedy fibrant since the matching maps are projections from a product onto one of its terms and all objects are fibrant in the underlying category. The given formula now follows by the standard form for the Bousfield-Kan spectral sequence.

If $G$ is $N$-excisive then the spectral sequence is restricted to the region where $0 \leq s \leq N-1$ so collapses at the $\mathbb{E}^N$ page. The cosimplicial simplicial set (\ref{eq:spseq}) is degenerate above the \ord{N} level and so the spectral sequence converges to the $E$-homology of its totalization. By Corollary \ref{cor:nat}, this is equivalent to the $E$-homology of $\Nat(F,G)$.
\end{proof}

\begin{remark}
There is a similar spectral sequence based on the cosimplicial cobar construction $\Phi K^\bullet \der_*G$ that, if $G$ is $N$-excisive for some $N$, converges to the $E$-homology of the spectrum $G(X)$.
\end{remark}

Before turning to the specific cases $[\finspec,\spectra]$ and $[\finbased,\spectra]$, we look at what our results say about classifying the extensions in the Taylor tower of a functor $F: \Cfin \to \spectra$.

Goodwillie proved in~\cite{goodwillie:2003} that the fibration sequence
\[
D_n F \longrightarrow P_n F \longrightarrow P_{n-1}F
\]
can be extended to the right. That is, there exists a functor $R_nF$ and a natural fibration sequence
\[
 P_n F \longrightarrow P_{n-1}F \longrightarrow R_n F.
\]
In particular, $R_nF$ is a delooping of $D_nF$. This theorem is important for understanding Taylor towers of space-level functors. On the other hand, for spectrum-valued functors it is a triviality. This raises the question whether for spectrum-valued functors the map $P_n F \longrightarrow P_{n-1} F$ can be classified in an interesting way by means of a universal fibration. For functors from $\spectra$ to $\spectra$ this was essentially answered by McCarthy~\cite{mccarthy:2001} (see also~\cite{kuhn:2004}). We now recover McCarthy's result with our methods, and show how it can be extended to functors from $\based$ to $\spectra$.

\begin{proposition}\label{prop:pullback}
For $F\in[\Cfin, \spectra]$ and $n \geq 2$, there is a homotopy pullback square, natural in $X$, of the form
\[\begin{diagram}
\node{P_n F(X)} \arrow{e}\arrow{s} \node{(\Phi \der_nF)(X)}\arrow{s} \\
\node{P_{n-1} F(X)} \arrow{e} \node{P_{n-1}(\Phi \der_nF)(X).}
\end{diagram}\]
where $\Phi \der_nF$ denotes the value of the functor $\Phi$ applied to $\der_nF$ viewed as a symmetric sequence concentrated in its \ord{n} term.
\end{proposition}
\begin{proof}
By Lemma \ref{lem:phi}, $\Phi \der_nF$ is $n$-excisive. It is sufficient to show that there is a natural transformation
\[ P_nF \to \Phi \der_nF \]
that becomes an equivalence after applying $D_n$.

By Proposition \ref{prop:unit} we have a $D_n$-equivalence
\begin{equation} \label{eq:Dn} P_nF \to \Phi \der_*(P_nF) \homeq \Phi \der_{\leq n}F. \end{equation}
Now $\Phi$ commutes with products so we have
\[ \Phi \der_{\leq n} F \homeq \prod_{k =1}^{n} \Phi \der_kF. \]
By Lemma \ref{lem:phi} again each $\Phi \der_kF$ is $k$-excisive, so the projection map
\[ \Phi \der_{\leq n}F \to \Phi \der_nF \]
is a $D_n$-equivalence. Composing this with (\ref{eq:Dn}) gives the required natural transformation.
\end{proof}
\begin{remark}
We only proved the proposition for $\C$ being $\based$ or $\spectra$, but it seems likely that there is a corresponding result for a more general class of categories.
\end{remark}
\begin{corollary}\label{cor:universal}
Let $\C$ be $\based$ or $\spectra$ and let $F\in[\Cfin, \spectra]$. Let $X\in \Cfin$. If $\C=\spectra$ then for every $n\ge 1$ there is a homotopy pullback square, natural in $X$
\[\begin{diagram}
\node{P_n F(X)} \arrow{e}\arrow{s} \node{\left(\der_n F \smsh X^{\smsh n}\right)^{h\Sigma_n}}\arrow{s} \\
\node{P_{n-1} F(X)} \arrow{e} \node{\Tate_{\Sigma_n}\left(\der_n F \smsh X^{\smsh n}\right).}
\end{diagram}\]
If $\C=\based$ then the pullback square has the following form
\[\begin{diagram}
\node{P_n F(X)} \arrow{e}\arrow{s} \node{\left(\der_n F \smsh X^{\smsh n}/\Delta^n X\right)_{h\Sigma_n}}\arrow{s} \\
\node{P_{n-1} F(X)} \arrow{e} \node{\left(\der_n F \smsh \Sigma \Delta^n X\right)_{h\Sigma_n}.}
\end{diagram}\]
where $\Delta^n X$ is the fat diagonal inside $X^{\smsh n}$.
\end{corollary}
\begin{proof}
In Corollary \ref{cor:spsp-phi} below we see that when $\C = \spectra$, we have
\[ \Phi \der_nF(X) \homeq (\der_nF \smsh X^{\smsh n})^{h\Sigma_n}. \]
So the assertion here is that there is an equivalence
\[
P_{n-1}\left(\der_n F \smsh X^{\smsh n}\right)^{h\Sigma_n} \homeq \Tate_{\Sigma_n}\left(\der_n F \smsh X^{\smsh n}\right).
\]
It is well-known, and is easy to prove that the norm map
\[
\left(\der_n F \smsh X^{\smsh n}\right)_{h\Sigma_n} \arrow{e,t}{N} \left(\der_n F \smsh X^{\smsh n}\right)^{h\Sigma_n}
\]
induces an equivalence of \ord{n} cross-effects. Since the target of this norm map is an $n$-excisive functor, the norm map is equivalent to the map
\[
D_n\left(\der_n F \smsh X^{\smsh n}\right)^{h\Sigma_n}\longrightarrow \left(\der_n F \smsh X^{\smsh n}\right)^{h\Sigma_n}.
\]
It follows that the Tate construction, which is the homotopy cofibre of the norm map, is the $(n-1)$-excisive approximation to $\left(\der_n F \smsh X^{\smsh n}\right)^{h\Sigma_n}$.

In the case $\C=\based$, Corollary \ref{cor:topsp-phi} below tells us that
\[ \Phi \der_nF(X) \homeq (\der_nF \smsh X^{\smsh n}/\Delta^n X)^{h\Sigma_n}. \]
When $X$ is a finite based complex, the space $X^{\smsh n}/\Delta^n X$ can be built from finitely many free $\Sigma_n$-cells. It follows that the norm map
\[
(\der_n F \smsh  X^{\smsh n}/\Delta^n X)_{h\Sigma_n}\longrightarrow (\der_n F \smsh  X^{\smsh n}/\Delta^n X)^{h\Sigma_n}
\]
is an equivalence. This establishes the top-right corner of the required square. It remains to prove that
\[
P_{n-1}\left(X\mapsto (\der_n F \smsh  X^{\smsh n}/\Delta^n X)_{h\Sigma_n}\right)\simeq (\der_n F \smsh \Sigma \Delta^n X)_{h\Sigma_n}.
\]
To see this, consider the homotopy fibration/cofibration sequence
\[
\Sigma^\infty X^{\smsh n} \longrightarrow \Sigma^\infty X^{\smsh n}/\Delta^n X \longrightarrow \Sigma^\infty \Sigma \Delta^n X.
\]
Smashing it with $\der_n F$ and taking homotopy orbits we obtain a fibration/cofibration sequence
\[
(\der_nF \smsh X^{\smsh n})_{h\Sigma_n} \longrightarrow (\der_nF \smsh X^{\smsh n}/\Delta^n X)_{h\Sigma_n} \longrightarrow (\der_n F\smsh  \Sigma \Delta^n X)_{h\Sigma_n}.
\]
Clearly, the fibre in this sequence is an $n$-homogeneous functor. It is easy to prove that the base is $(n-1)$-excisive (it is the homotopy colimit of $(n-1)$-excisive functors). It follows that the map from the total space to the base is projection on the \ord{(n-1)} Taylor polynomial.
\end{proof}

\section{Functors from spectra to spectra} \label{sec:specspec}

We now turn to the specific case where $\C = \spectra$, that is, we look at Taylor towers of pointed simplicial functors $F: \finspec \to \spectra$. Our goal in this section is to calculate the objects $K_r A_n$ and the maps $\delta_{r,s}$ from Definition \ref{def:derAn} in this case. These are given in Propositions \ref{prop:derAn-spsp} and \ref{prop:delta-spsp} respectively.

Our description relies on certain symmetric group actions on sphere spectra.

\begin{definition} \label{def:spheres-sym}
Let $L$ be a non-negative integer and let $S^L$ denote the suspension spectrum of the based topological $L$-sphere. For positive integers $n,r$ we write
\[ S^{L(n-r)} := \Map((S^L)^{\smsh r},(S^L)^{\smsh n}). \]
This has commuting actions of the symmetric groups $\Sigma_r$ and $\Sigma_n$ by permutation on the two smash powers respectively.

For an integer $n \geq 1$ we denote the set $\{1,\dots,n\}$ by $\un{n}$. Now suppose we are given a surjection $\alpha: \un{n} \epi \un{r}$.

For positive integers $L < L'$ there is an isomorphism
\[ S^{L'} \isom S^{L} \smsh S^{L'-L} \]
arising from the canonical inclusion of $\mathbb{R}^L$ in $\mathbb{R}^{L'}$ via the first $L$ coordinates. We then have a map
\[ i^{L,L'}_{\alpha}: S^{L(n-r)} \to S^{L'(n-r)} \]
that is adjoint to the composite
\[ \dgTEXTARROWLENGTH=3em \begin{split} (S^{L'})^{\smsh r} \smsh \Map((S^L)^{\smsh r},(S^L)^{\smsh n})
    &\arrow{e,t}{\isom} (S^{L'-L})^{\smsh r} \smsh (S^L)^{\smsh r} \smsh \Map((S^L)^{\smsh r},(S^L)^{\smsh n}) \\
    &\arrow{e} (S^{L'-L})^{\smsh r} \smsh (S^L)^{\smsh n} \\
    &\arrow{e,t}{\alpha^{\#}} (S^{L'-L})^{\smsh n} \smsh (S^L)^{\smsh n} \\
    &\arrow{e,t}{\isom} (S^{L'})^{\smsh n}
\end{split} \]
The second map is the canonical evaluation, and the third is constructed from the diagonal map on the suspension spectrum $S^{L'-L}$ via the surjection $\alpha$.
\end{definition}

We then have the following description of the pieces of the comonad $K: \symseq \to \symseq$. The proof of this result starts with Lemma \ref{lem:Rx-spsp} later in this section.

\begin{prop} \label{prop:derAn-spsp}
In the classification of Taylor towers of functors $F: \finspec \to \spectra$ we have the following calculation. For a $\Sigma_n$-spectrum $A_n$, there is a $\Sigma_r$-equivariant equivalence of spectra
\[ K_r A_n \homeq \hocolim_{L \to \infty} \left[ \left( \prod_{\un{n} \epi \un{r}} A_n \smsh S^{L(n-r)} \right)^{h\Sigma_n} \right]. \]
The product is over the set of surjections $\alpha: \un{n} \epi \un{r}$. The symmetric group $\Sigma_n$ acts on the product by composition and via the diagonal of its actions on $A_n$ and $S^{L(n-r)}$. The homotopy colimit is formed over the diagram formed by the maps $i^{L,L'}_{\alpha}$ of Definition \ref{def:spheres-sym}.
\end{prop}

\begin{remark} \label{rem:derAn-spsp}
The orbits of the action of $\Sigma_n$ on the set of surjections $\un{n} \epi \un{r}$ are in one-to-one correspondence with ordered partitions $n = n_1+\dots+n_r$ of $n$ into positive integers $n_1,\dots,n_r$. The stabilizer for the orbit corresponding to such a partition is isomorphic to the subgroup
\[ \Sigma_{n_1} \times \dots \times \Sigma_{n_r} \subseteq \Sigma_n. \]
It follows that the formula for $K_r A_n$ can be rewritten as
\[ K_r A_n \homeq \prod_{n = n_1+\dots+n_r} \hocolim_{L \to \infty} \left[ \left( A_n \smsh S^{L(n-r)} \right)^{h\Sigma_{n_1} \times \dots \times \Sigma_{n_r}} \right] \]
where the product is over the set of ordered partitions of $n$ into positive integers $n_1,\dots,n_r$.
\end{remark}

\begin{corollary} \label{cor:derF-spsp}
For a pointed simplicial functor $F: \finspec \to \spectra$, and $r < n$, there are natural $\Sigma_r$-equivariant maps
\[ \theta_{r,n}: \der_rF \to \hocolim_{L \to \infty} \left[ \left( \prod_{\un{n} \epi \un{r}} A_n \smsh S^{L(n-r)} \right)^{h\Sigma_n} \right] \]
that encode the $K$-coalgebra structure map for $\der_*F$. Equivalently, these can be described in terms of a map
\[ \theta_{(n_1,\dots,n_r)}: \der_rF \to \hocolim_{L \to \infty} \left[ \left( A_n \smsh S^{L(n-r)} \right)^{h\Sigma_{n_1} \times \dots \times \Sigma_{n_r}} \right] \]
for each partition $n = n_1+\dots+n_r$.
\end{corollary}

We now turn to the maps $\delta_{r,s}$ that encode the comonad structure. To describe these it is convenient to observe the following formula for $K_r K_s A_n$.

\begin{proposition} \label{prop:spsp-KK}
For positive integers $r < s < n$, we have
\[ K_r K_s A_n \homeq \hocolim_{M \to \infty} \left[ \hocolim_{N \to \infty} \left[ \prod_{\un{n} \epi \un{s} \epi \un{r}} A_n \smsh S^{N(n-s)} \smsh S^{M(s-r)} \right]^{h\Sigma_n} \right]^{h\Sigma_s}. \]
\end{proposition}
\begin{proof}
This follows from Proposition \ref{prop:derAn-spsp} by observing that finite products commute with both homotopy fixed points and smash products in the category of spectra.
\end{proof}

We also need to understand how the symmetric group actions on the spheres $S^{L(n-r)}$ fit together.

\begin{definition} \label{def:sphere-cooperad}
For positive integers $L$, $n$, $r$ and $s$, there are natural composition maps
\[ c: \Map((S^L)^{\smsh r},(S^L)^{\smsh s}) \smsh \Map((S^L)^{\smsh s},(S^L)^{\smsh n}) \to \Map((S^L)^{\smsh r},(S^L)^{\smsh n}) \]
or, in the notation of \ref{def:spheres-sym},
\[ c: S^{L(s-r)} \smsh S^{L(n-s)} \to S^{L(n-r)}. \]
Note that these maps are weak equivalences of spectra and are equivariant with respect to actions of $\Sigma_n$, $\Sigma_s$ and $\Sigma_r$. (The $\Sigma_s$-action is the diagonal action on the source and is trivial on the target.)

For each $c$ we fix an inverse
\[ c^{-1}: S^{L(n-r)} \to S^{L(s-r)} \smsh S^{L(n-s)} \]
in the homotopy category of spectra with $\Sigma_n \times \Sigma_s \times \Sigma_r$ action.
\end{definition}

\begin{lemma}
For surjections $\alpha: \un{n} \epi \un{s}$ and $\beta: \un{s} \epi \un{r}$, the following diagram commutes in the homotopy category:
\[ \begin{diagram}
  \node{S^{L(n-r)}} \arrow{s,r}{i^{L,L'}_{\beta \circ \alpha}} \arrow{e,t}{c^{-1}} \node{S^{L(s-r)} \smsh S^{L(n-s)}} \arrow{s,r}{i^{L,L'}_\beta \smsh i^{L,L'}_\alpha} \\
  \node{S^{L'(n-r)}} \arrow{e,t}{c^{-1}} \node{S^{L'(s-r)} \smsh S^{L'(n-s)}}
\end{diagram} \]
where the vertical maps are as in Definition \ref{def:spheres-sym}.
\end{lemma}
\begin{proof}
The maps $c$ commute strictly with the vertical maps so their inverses commute in the homotopy category.
\end{proof}

We then have the following description of the comonad structure maps.

\begin{prop} \label{prop:delta-spsp}
For positive integers $r < s < n$, and with respect to the equivalences of \ref{prop:derAn-spsp} and \ref{prop:spsp-KK}, the comonad structure map
\[ \delta_{r,s}: K_r A_n \to K_r K_s A_n \]
takes the form
\[ \begin{diagram} \dgARROWLENGTH=1em
  \node{\hocolim_{L \to \infty} \left[ \left( \prod_{\un{n} \epi \un{r}} A_n \smsh S^{L(n-r)} \right)^{h\Sigma_n} \right]} \arrow{e,t}{\delta_{r,s}}
  \node{\hocolim_{M \to \infty} \left[ \hocolim_{N \to \infty} \left[ \prod_{\un{n} \epi \un{s} \epi \un{r}} A_n \smsh S^{N(n-s)} \smsh S^{M(s-r)} \right]^{h\Sigma_n} \right]^{h\Sigma_s}}
\end{diagram} \]
and, in the homotopy category, is induced by the map
\[ \prod_{\un{n} \epi \un{r}} A_n \smsh S^{L(n-r)} \to \left[ \prod_{\un{n} \epi \un{s} \epi \un{r}} A_n \smsh S^{L(n-s)} \smsh S^{L(s-r)} \right]^{h\Sigma_s} \]
which composes surjections and applies $c^{-1}$ of Definition \ref{def:sphere-cooperad}. The $\Sigma_s$-equivariance of $c^{-1}$ implies that this lands in the $\Sigma_s$-fixed points and hence the homotopy fixed points.

The map $\delta_{r,s}$ is then obtained by applying the $\Sigma_n$-homotopy fixed points, commuting them with the $\Sigma_s$-homotopy fixed points, and including into the homotopy colimits as the terms $M = L$ and $N = L$.
\end{prop}

\begin{remark} \label{rem:pro-operad}
Propositions~ \ref{prop:derAn-spsp} and \ref{prop:delta-spsp} can be interpreted in terms of a certain {\it pro-operad} of spectra. The pro-operad is a tower of operads, where the \ord{n} spectrum of the \ord{L} operad is $S^{L(1-n)}$. Associated to each operad $S^{L(1-*)}$ is a comonad $K^L$ whose coalgebras are precisely the right $S^{L(1-*)}$-modules. The comonad $K$ is then equivalent (as a comonad) to the homotopy colimit of the $K^L$.
\end{remark}

We now turn to the proofs of Propositions \ref{prop:derAn-spsp} and \ref{prop:delta-spsp}. Following the general approach of Section \ref{sec:sp}, this involves first understanding the derivatives of the representable functors in $[\finspec,\spectra]$. These are given by the following lemma.

\begin{lemma} \label{lem:Rx-spsp}
For $X \in \finspec$, we have
\[ \der_n(\Sigma^\infty \Hom_{\spectra}(X,-)) \homeq \dual(X^{\smsh n}) \]
with $\Sigma_n$-action given by permuting the factors of $X$.
\end{lemma}
\begin{proof}
This result is well-known. It follows, for example, from the models for the derivatives of functors from spectra to spectra described in \cite[3.1.4]{arone/ching:2011}.
\end{proof}

\begin{corollary} \label{cor:spsp-phi}
The functor $\der_*: [\finspec,\spectra] \to \symseq$ has right adjoint $\Phi: \symseq \to [\finspec,\spectra]$ that satisfies
\[ \Phi(A)(X) \homeq \prod_{n = 1}^{\infty} (A_n \smsh X^{\smsh n})^{h\Sigma_n}. \]
\end{corollary}

\begin{proof}[Proof of Proposition \ref{prop:derAn-spsp}]
It follows from Corollary \ref{cor:spsp-phi} that the spectrum $K_r A_n$ is given by the \ord{r} derivative of the functor
\[ X \mapsto (A_n \smsh X^{\smsh n})^{h\Sigma_n}. \]
We calculate this using cross-effects. The \ord{r} cross-effect of $(A_n \smsh X^{\smsh n})$ is
\[ (X_1,\dots,X_r) \mapsto \prod_{\alpha: \un{n} \epi \un{r}} A_n \smsh X_1^{n_1} \smsh \dots \smsh X_r^{n_r} \]
where we are writing $n_j := |\alpha^{-1}(j)|$.

Since cross-effects commute with homotopy fixed points, the \ord{r} cross-effect of $(A_n \smsh X^{\smsh n})^{h\Sigma_n}$ is then
\[ (X_1,\dots,X_r) \mapsto \left[ \prod_{\alpha: \un{n} \epi \un{r}} A_n \smsh X_1^{n_1} \smsh \dots \smsh X_r^{n_r} \right]^{h\Sigma_n}. \]
According to \cite{goodwillie:2003}, we can now recover the \ord{r} derivative by multilinearizing. This gives us
\[ K_r A_n \homeq \hocolim_{L \to \infty} \Map \left( (S^L)^{\smsh r},  \left[ \prod_{\alpha: \un{n} \to \un{r}} A_n \smsh (S^L)^{\smsh n_1} \smsh \dots \smsh (S^L)^{\smsh n_r}\right]^{h\Sigma_{n}} \right). \]
The terms in this homotopy colimit are clearly equivalent to
\[ \left( \prod_{\un{n} \epi \un{r}} A_n \smsh S^{L(n-r)} \right)^{h\Sigma_n} \]
so it remains to show that the maps between these terms are given by the $i^{L,L'}_{\alpha}$ of Definition \ref{def:spheres-sym}.

For this we note that the stabilization maps are given by the topological enrichment of the cross-effect functor which yields maps
\[ \creff_n(F)(S^L,\dots,S^L) \smsh S^{L'-L} \smsh \dots \smsh S^{L'-L} \to \creff_n(F)(S^{L'}, \dots, S^{L'}). \]
(Here we are using the fact that the spectra $S^{L'-L}$ are suspension spectra. The cross-effect functors are not strictly enriched over based spaces, but are enriched in an up-to-homotopy sense that is sufficient for us.) The topological enrichment of the functor $X \mapsto X^{\smsh n}$ is given by the diagonal of the space being tensored with. It follows that the stabilization map in our case is given by the map
\[ (A_n \smsh (S^L)^{\smsh n_1} \smsh \dots \smsh (S^L)^{\smsh n_r}) \smsh S^{L'-L} \smsh \dots \smsh S^{L'-L} \to (A_n \smsh (S^{L'})^{\smsh n_1} \smsh \dots \smsh (S^{L'})^{\smsh n_r}) \]
that applies the diagonal
\[ \alpha^{\#}: (S^{L'-L})^{\smsh r} \to (S^{L'-L})^{\smsh n}. \]
This is equivalent to the description of $i^{L,L'}_{\alpha}$ in Definition \ref{def:spheres-sym}.
\end{proof}

We now turn to Proposition \ref{prop:delta-spsp}. The comonad structure maps $\delta_{r,s}: K_r A_n \to K_r K_s A_n$ are derived from the unit map of the $(\der_*,\Phi)$ adjunction, that is, the map
\[ \eta: F \to \Phi \der_*F. \]
It follows from Corollary \ref{cor:spsp-phi} that $\eta$ is built from natural transformations
\[ \eta_s: F(X) \to (\der_sF \smsh X^{\smsh s})^{h\Sigma_s}. \]
The structure map $\delta_{r,s}$ is given by applying $\der_r$ to the natural transformation $\eta_s$ associated to the functor $F(X) = (A_n \smsh X^{\smsh n})^{h\Sigma_n}$. The first step in proving Proposition \ref{prop:delta-spsp} is therefore to get a better understanding of the maps $\eta_s$. In order to express these we use the `co-cross-effects'.

\begin{definition}[Co-cross-effect] \label{def:cocross}
For a pointed simplicial functor $F: \finspec \to \spectra$, the \emph{\ord{n} co-cross-effect} is the functor of $n$-variables, denoted by $\creff^n(F)$, given by
\[ \creff^n(F)(X_1,\dots,X_n) := \thocofib_{S \subseteq \un{n}} F \left( \prod_{i \in S} X_i \right) \]
This is the total homotopy cofibre of an $n$-cube, as defined by Goodwillie \cite[1.4]{goodwillie:1991}.

For functors from spectra to spectra, the \ord{n} co-cross-effect is naturally equivalent to the \ord{n} cross-effect \cite[2.2]{ching:2010} and hence its multilinearization can be used to calculate the \ord{n} derivative. We therefore have
\[ \der_nF \homeq \hocolim_{L \to \infty} \Map((S^L)^{\smsh n}, \creff^n(F)(S^L,\dots,S^L)). \]
\end{definition}

\begin{definition}[Assembly map equivalence] \label{def:assembly}
Goodwillie established in \cite[\S5]{goodwillie:2003} the relationship between the cross-effects and the derivatives. Central to this is the existence of an equivalence (for finite cell spectra $X_1,\dots,X_n$)
\[ \der_nF \smsh X_1 \smsh \dots \smsh X_n \arrow{e,t}{\sim} \hocolim_{L \to \infty} \Map((S^L)^{\smsh n}, \creff^n(F)(\Sigma^L X_1,\dots,\Sigma^L X_n)). \]
We refer to this as the \emph{assembly map equivalence} and it describes how a symmetric multilinear functor (in this case the multilinearized co-cross-effect) can be described with a single coefficient spectrum (in this case $\der_nF$).

In the case that $X_1,\dots,X_n$ are suspension spectra, the assembly map equivalence can be described at the level of co-cross-effects (i.e. before multilinearizing). In this case there are natural maps
\[ \creff^n(F)(S^L,\dots,S^L) \smsh X_1 \smsh \dots \smsh X_n \to \creff^n(S^L \smsh X_1, \dots, S^L \smsh X_n) \]
that arise from the simplicial enrichment of the co-cross-effect functor in each variable. (Strictly speaking, this map only exists up to certain inverse equivalences which we suppress.) Looping and taking the homotopy colimit as $L \to \infty$ we get the assembly map equivalence as described above.
\end{definition}

With this in mind, we can describe the maps $\eta_n$ as follows.

\begin{prop} \label{prop:spsp-eta}
For a functor $F: \finspec \to \spectra$, the map $\eta_n$ is equivalent to the composite
\[ \dgTEXTARROWLENGTH=3em \begin{split}
    F(X) &\arrow{e,t}{\Delta} F(X \times \dots \times X)^{h\Sigma_n} \\
        &\arrow{e,t}{\iota} \creff^n(F)(X,\dots,X)^{h\Sigma_n} \\
        &\arrow{e,t}{i} \left[\hocolim_{L \to \infty} \Map\left((S^L)^{\smsh n}, \creff^n(F)(\Sigma^L X, \dots, \Sigma^L X)\right)\right]^{h\Sigma_n} \\
        &\arrow{e,t}{\sim} (\der_nF \smsh X^{\smsh n})^{h\Sigma_n}
\end{split} \]
where $\Delta$ is induced by the diagonal $X \to X \times \dots \times X$, $\iota$ is the map from the terminal object of a cube to its total homotopy cofibre, $i$ is the inclusion as the $L = 0$ term and the final equivalence is an inverse to the assembly map equivalence of \ref{def:assembly}.
\end{prop}
\begin{proof}
We show first that this calculation is valid for representable functors. Suppose $F = \Sigma^\infty R_Y$ for some $Y \in \finspec$. Then the map
\[ \eta_n: \Sigma^\infty R_Y(X) \to \Phi \der_n(\Sigma^\infty R_Y)(X) = \Map(\der_n(\Sigma^\infty R_X), \der_n(\Sigma^\infty R_Y))^{\Sigma_n} \homeq \Map(Y^{\smsh n},X^{\smsh n})^{\Sigma_n} \]
is given, for formal reasons, by the simplicial enrichment of the functor $X \mapsto X^{\smsh n}$. This is adjoint to the diagonal map
\[ \Hom(Y,X) \arrow{e,t}{\Delta} \left[\Hom(Y,X)^{\smsh n}\right]^{\Sigma_n} \to \Hom(Y^{\smsh n},X^{\smsh n})^{\Sigma_n} \isom \Omega^\infty \Map(Y^{\smsh n},X^{\smsh n})^{\Sigma_n}. \]
On the other hand, the composite in the statement of the lemma is equivalent to
\[ \dgTEXTARROWLENGTH=3em \begin{split} \Sigma^\infty \Hom(Y,X) &\arrow{e,t}{\Delta} [\Sigma^\infty \Hom(Y,X)^{n}]^{\Sigma_n} \\
    &\arrow{e,t}{\iota} [\Sigma^\infty \Hom(Y,X)^{\smsh n}]^{\Sigma_n} \\
    &\arrow{e,t}{i} [\Map(Y,X)^{\smsh n}]^{\Sigma_n} \\
    &\arrow{e,t}{\sim} \Map(Y^{\smsh n},X^{\smsh n})^{\Sigma_n}
\end{split} \]
which is therefore equivalent to $\eta_n$.

Now take an arbitrary $F: \finspec \to \spectra$. The composite
\[ F(X) \arrow{e,t}{\eta_n} (\der_nF \smsh X^{\smsh n})^{h\Sigma_n} \to (\der_nF \smsh X^{\smsh n}) \]
is equivalent, by naturality of $\eta_n$, to
\[ F(Y) \smsh_{Y \in \finspec} R_Y(X) \arrow{e,t}{\eta_n} F(Y) \smsh_{Y \in \finspec} [\der_n(R_Y) \smsh X^{\smsh n}]^{h\Sigma_n} \to F(Y) \smsh_{Y \in \finspec} \der_n(R_Y) \smsh X^{\smsh n} \]
which, from the calculation for $R_Y$, we know to be equivalent to the composite
\[ \dgTEXTARROWLENGTH=3em \begin{split} F(Y) \smsh_{Y} R_Y(X) &\to F(Y) \smsh_{Y} R_Y(X \times \dots \times X) \\
    &\to F(Y) \smsh_{Y} \creff^n(R_Y)(X,\dots,X) \\
    &\to F(Y) \smsh_{Y} \der_n(R_Y) \smsh X^{\smsh n}
\end{split} \]
which is equivalent to
\[ F(X) \to F(X \times \dots \times X) \to \creff^n(F)(X,\dots,X) \to (\der_nF \smsh X^{\smsh n}). \]
The factorization via homotopy fixed points is unique up to homotopy, so this completes the proof.
\end{proof}

\begin{remark}
The maps $\eta_n$ have been previously studied by McCarthy \cite{mccarthy:2001}. He used them to show how the $n$-excisive approximation $P_nF$ can be built from $P_{n-1}F$ and $D_nF$ via a certain Tate spectrum. Since our goal here is also to understand how $P_nF$ is built from its pieces, it should not be a surprise to see this map appearing here too.
\end{remark}

\begin{example}
When $F$ is the $n$-homogeneous functor given by
\[ F(X) = (A_n \smsh X^{\smsh n})_{h\Sigma_n} \]
the map $\eta_n$ is equivalent to the norm map of Greenlees-May \cite{greenlees/may:1995}
\[ N: (A_n \smsh X^{\smsh n})_{h\Sigma_n} \to (A_n \smsh X^{\smsh n})^{h\Sigma_n} \]
from homotopy orbits to homotopy fixed points.
\end{example}

We are now in a position to calculate the comonad structure maps $\delta_{r,s}: K_r A_n \to K_r K_s A_n$.

\begin{proof}[Proof of Proposition \ref{prop:delta-spsp}]
Recall that $\delta_{r,s}$ is given by applying $\der_r$ to the unit map $\eta_s$ of Proposition \ref{prop:spsp-eta} in the case that $F(X) = (A_n \smsh X^{\smsh n})^{h\Sigma_n}$. Recalling the form of the \ord{s} cross-effect of $F$ from the proof of \ref{prop:derAn-spsp} we see that $\eta_s$ is the composite
\[ \dgTEXTARROWLENGTH=3em \begin{split}
  (A_n \smsh X^{\smsh n})^{h\Sigma_n}
    &\arrow{e,t}{\Delta} \left[ \left[ \prod_{\un{n} \epi \un{s}} A_n \smsh X^{\smsh n} \right]^{h\Sigma_n} \right]^{h\Sigma_s} \\
    &\arrow{e} \left[ \hocolim_{N \to \infty} \left[ \prod_{\un{n} \epi \un{s}} A_n \smsh S^{N(n-s)} \smsh X^{\smsh n} \right]^{h\Sigma_n} \right]^{h\Sigma_s} \\
    &\lweq \left[ \hocolim_{N \to \infty} \left[ \prod_{\un{n} \epi \un{s}} A_n \smsh S^{N(n-s)} \right]^{h\Sigma_n} \smsh X^{\smsh s} \right]^{h\Sigma_s}
\end{split} \]
where the first map $\Delta$ is the diagonal, the second is inclusion as the $N = 0$ term of the homotopy colimit, and the third is the inverse of the assembly map equivalence of \ref{def:assembly}, applied to the \ord{s} cross-effect of the functor $F(X)=(A_n\smsh X^{\smsh n})^{h\Sigma_n}$. This equivalence is valid for all finite cell spectra $X$, but is easiest to describe in the case that $X$ is a suspension spectrum. In that case, each surjection $\alpha: \un{n} \epi \un{s}$ determines a diagonal map
\[ \alpha^{\#}: X^{\smsh s} \to X^{\smsh n} \]
and the assembly map equivalence above is built from these.

To get $\delta_{r,s}$ we now apply $\der_r$ to this composite. First we take the \ord{r} cross-effect. Under the identification again of the \ord{r} cross-effect of $X^{\smsh n}$ in terms of surjections, and noting that cross-effects commute with homotopy colimits, products, homotopy fixed points and smash products, we get the composite
\[ \dgTEXTARROWLENGTH=3em \begin{split}
  \left[ \prod_{\un{n} \epi \un{r}} A_n \smsh X^{\smsh n} \right]^{h\Sigma_n}
    &\arrow{e} \left[ \left[ \prod_{\un{n} \epi \un{s}} \prod_{\un{n} \epi \un{r}} A_n \smsh X^{\smsh n} \right]^{h\Sigma_n} \right]^{h\Sigma_s} \\
    &\arrow{e} \left[ \hocolim_{N \to \infty} \left[ \prod_{\un{n} \epi \un{s}} \prod_{\un{n} \epi \un{r}} A_n \smsh S^{N(n-s)} \smsh X^{\smsh n} \right]^{h\Sigma_n} \right]^{h\Sigma_s} \\
    &\lweq \left[ \prod_{\un{s} \epi \un{r}} \hocolim_{N \to \infty} \left[ \prod_{\un{n} \epi \un{s}} A_n \smsh S^{N(n-s)} \right]^{h\Sigma_n} \smsh X^{\smsh s} \right]^{h\Sigma_s}
\end{split} \]

To get the map $\delta_{r,s}$ we now multilinearize the cross-effects and apply to the sphere spectrum. (Note that the expressions in the above formula are really functors of $r$ variables that have been diagonalized, and that we are multilinearizing the underlying functors of $r$ variables.) This therefore tells us that $\delta_{r,s}$ is equivalent to the composite
\[ \dgTEXTARROWLENGTH=1.5em \begin{split}
  \hocolim_{L \to \infty} \left[ \prod_{\un{n} \epi \un{r}} A_n \smsh S^{L(n-r)} \right]^{h\Sigma_n}
    &\arrow{e} \hocolim_{L \to \infty} \left[ \left[ \prod_{\un{n} \epi \un{s}} \prod_{\un{n} \epi \un{r}} A_n \smsh S^{L(n-r)} \right]^{h\Sigma_n} \right]^{h\Sigma_s} \\
    &\arrow{e} \hocolim_{L \to \infty} \left[ \hocolim_{N \to \infty} \left[ \prod_{\un{n} \epi \un{s}} \prod_{\un{n} \epi \un{r}} A_n \smsh S^{N(n-s)} \smsh S^{L(n-r)} \right]^{h\Sigma_n} \right]^{h\Sigma_s} \\
    &\lweq \hocolim_{L \to \infty} \left[ \prod_{\un{s} \epi \un{r}} \hocolim_{N \to \infty} \left[ \prod_{\un{n} \epi \un{s}} A_n \smsh S^{N(n-s)} \right]^{h\Sigma_n} \smsh S^{L(s-r)} \right]^{h\Sigma_s}
\end{split} \]
where the first map is a diagonal and the second is inclusion as the term $N = 0$.

The third map in the above composite is based on the assembly map equivalence and hence on the diagonal map
\[ \alpha^{\#}: (S^L)^{\smsh s} \to (S^L)^{\smsh n} \]
associated to a surjection $\alpha: \un{n} \epi \un{s}$. Applying the \ord{r} cross-effect to this diagonal map gives a map
\[ \alpha_*: \prod_{\un{s} \epi \un{r}} (S^L)^{\smsh s} \to \prod_{\un{n} \epi \un{r}} (S^L)^{\smsh n} \]
for which the component associated to a surjection $\beta: \un{n} \epi \un{r}$ depends on whether $\beta$ factors as
\[ \un{n} \arrow{e,t,>>}{\alpha} \un{s} \arrow{e,t,>>}{\gamma} \un{r} \]
for some surjection $\gamma$ (which is unique if it exists). If it does factor thus, then the required component of $\alpha_*$ is projection onto the term corresponding to $\gamma$ followed by the diagonal. If it does not factor thus, then that component is the trivial map. The third map in the composite above is therefore determined by $\alpha_*$.

Our claim then is that this composite is induced, in the homotopy category, by the maps
\[ c^{-1}: S^{L(n-r)} \to S^{L(s-r)} \smsh S^{L(n-s)} \]
of Definition \ref{def:sphere-cooperad} together with various natural inclusions and projections. The truth of this boils down to the following claim: that for each surjection $\alpha: \un{n} \epi \un{s}$ the diagram
\[ \begin{diagram}
  \node{\Map(S^{Ls},S^{Ln}) \smsh \Map(S^{Lr},S^{Ls})} \arrow{se,b}{(1,\Delta)} \arrow{e,t}{c} \node{\Map(S^{Lr},S^{Ln})} \arrow{s,r}{i^{0,L}_\alpha} \\
  \node[2]{\Map(S^{Ls},S^{Ln}) \smsh \Map(S^{Lr},S^{Ln})}
\end{diagram} \]
commutes up to a homotopy that is compatible with the action of the group $\Sigma_r\times \Sigma_{n_1} \times\cdots\times \Sigma_{n_s}$, where $n_1, \ldots, n_s$ are the inverse images of $\alpha$. The homotopy also needs to be natural in $L$. In this diagram, $c$ is the composition map, $i_\alpha^{0, L}$ is the composed map
\[
\Map(S^{Lr},S^{Ls}) \stackrel{\simeq}{\longrightarrow} S^0 \smsh \Map(S^{Lr},S^{Ls})  \longrightarrow \Map(S^{Ls},S^{Ln}) \smsh\Map(S^{Lr},S^{Ls})
\]
where the second map uses the map $S^0 \longrightarrow \Map(S^{Ls},S^{Ln})$, which is adjoint to the diagonal map $\alpha^\sharp\colon S^{Ls}\to S^{Ln}$ induced by $\alpha$. The map $(1, \Delta)$ is the identity on the factor $\Map(S^{Ls},S^{Ln})$ and, again, uses the map $\alpha^\sharp\colon S^{Ls}\to S^{Ln}$ to get a map
\[
\Map(S^{Lr},S^{Ls}) \longrightarrow \Map(S^{Lr},S^{Ln}).
\]
We need to show that the resulting two maps
\[
\Map(S^{Ls},S^{Ln})\smsh \Map(S^{Lr},S^{Ls}) \longrightarrow \Map(S^{Ls},S^{Ln})\smsh \Map(S^{Lr},S^{Ln})
\]
are naturally homotopic. If one desuspends the two maps by $S^{L(s-r)}$ then they are naturally equivalent to two maps
\[
\Map(S^{Ls},S^{Ln}) \longrightarrow \Map(S^{Ls},S^{Ln}) \smsh \Map(S^{Ls},S^{Ln})
\]
given by $f\mapsto f \smsh \alpha^\sharp$ and $f\mapsto \alpha^\sharp\smsh f$. To see that these two maps are homotopic we note that these maps, which were defined as stable maps, are in fact (naturally equivalent to) stabilizations of space level maps. Namely, $S^{L(n-s)}$ is naturally equivalent to the suspension spectrum of the one-point compactification of the orthogonal complement of $\R^{Ls}$ in $\R^{Ln}$. The fact that the two maps are homotopic now boils down to the following elementary assertion: Let $W$ be a vector space. The two natural inclusions $S^W\to S^W\smsh S^W$, given by $w\mapsto w \smsh S^0$ and $w\mapsto S^0 \smsh w$ are naturally homotopic. The homotopy is given by the one-point compactification of the linear map $w\mapsto (w\cos t, w\sin t)$, where $0\le t\le \frac{\pi}{2}$.
\end{proof}

This completes our description of the comonad $K$ in the context of functors from spectra to spectra. We now turn to more explicit descriptions of the maps $\theta_{r,n}$ and $\theta_{(n_1,\dots,n_r)}$ of Corollary \ref{cor:derF-spsp} for certain functors $F: \finspec \to \spectra$.

\subsubsection*{$2$-excisive functors}

A $2$-excisive functor $F: \finspec \to \spectra$ is determined by a $2$-term symmetric sequence $A_1,A_2$ and a single map
\[ \theta_{1,2} = \theta_{(2)}: A_1 \to \hocolim_{L \to \infty} [(S^{L(2-1)} \smsh A_2)^{h\Sigma_2}]. \]
The $\Sigma_2$-spectrum
\[ S^{L(2-1)} := \Map(S^L,(S^L)^{\smsh 2}) \]
is equivariantly equivalent to the suspension spectrum of the based space $S^L$ with $\Sigma_2$-action given by reflection in each of the $L$ coordinates. This action is free away from a copy of $S^0$. We therefore have a commutative diagram
\[ \begin{diagram}
  \node{(S^0 \smsh A_2)_{h\Sigma_2}} \arrow{s} \arrow{e} \node{(S^L \smsh A_2)_{h\Sigma_2}} \arrow{s} \arrow{e} \node{(S^L/S^0 \smsh A_2)_{h\Sigma_2}} \arrow{s,r}{\sim} \\
  \node{(S^0 \smsh A_2)^{h\Sigma_2}} \arrow{e} \node{(S^L \smsh A_2)^{h\Sigma_2}} \arrow{e} \node{(S^L/S^0 \smsh A_2)^{h\Sigma_2}}
\end{diagram} \]
in which the vertical maps are the standard norm maps and the rows are cofibration sequences. It follows that the left-hand square is a homotopy pushout square. Now taking the homotopy colimit as $L \to \infty$ we get a homotopy pushout square
\[ \begin{diagram}
  \node{(A_2)_{h\Sigma_2}} \arrow{e} \arrow{s} \node{*}  \arrow{s} \\
  \node{(A_2)^{h\Sigma_2}} \arrow{e} \node{\hocolim_{L \to \infty} (S^L \smsh A_2)^{h\Sigma_2}}
\end{diagram} \]
which implies that the target of the map $\theta_{1,2}$ is equivalent to the Tate spectrum $\Tate_{\Sigma_2}(A_2)$. Thus, the homotopy theory of $2$-excisive functors $F: \finspec \to \spectra$ is equivalent to that of triples $(A_1,A_2,\theta)$ where $A_2$ has a $\Sigma_2$-action and $\theta$ is any map
\[ \theta: A_1 \to \Tate_{\Sigma_2}(A_2). \]

\subsubsection*{$3$-excisive functors}

A $3$-excisive functor $F: \finspec \to \spectra$ is determined by a $3$-term symmetric sequence $A_1,A_2,A_3$, together with a map $\theta_{(2)}$ as above, a map
\[ \theta_{(3)}: A_1 \to \hocolim_{L \to \infty} \left[(A_3 \smsh S^{L(3-1)})^{h\Sigma_3}\right], \]
and a map
\[ \theta_{(1,2)}: A_2 \to \hocolim_{L \to \infty} \left[\left(A_3 \smsh S^{L(3-2)} \right)^{h\Sigma_1 \times \Sigma_2} \right]. \]
By the same analysis as for $\theta_{(2)}$, the target of $\theta_{(1,2)}$ is equivalent to the Tate spectrum $\Tate_{\Sigma_1 \times \Sigma_2}(A_3)$.

A similar analysis for $\theta_{(3)}$ implies that its target is equivalent to the Tate spectrum $\Tate_{\Sigma_3}(L_3 \smsh A_3)$ where $L_3$ is the subspace of $S^2$ on which $\Sigma_3$ fails to acts freely. (This space $L_3$ is the union of three semi-circles glued at their endpoints. It is non-equivariantly equivalent to a wedge of two $1$-spheres.)

These maps must make the following square commute up to homotopy
\[ \begin{diagram}
  \node{A_1} \arrow{e} \arrow{s} \node{\Tate_{\Sigma_2}(A_2)} \arrow{s} \\
  \node{\Tate_{\Sigma_3}(L_3 \smsh A_3)} \arrow{e} \node{\Tate_{\Sigma_2}(\Tate_{\Sigma_1 \times \Sigma_2}(A_3) \times \Tate_{\Sigma_2 \times \Sigma_1}(A_3))}
\end{diagram} \]
but the bottom-right corner is trivial since $\Sigma_2$ permutes the two terms in $K_2(A_3)$. There is therefore no compatibility condition required of the three maps.

We conclude then that there is an equivalence between the homotopy theory of $3$-excisive functors $F: \finspec \to \spectra$ and that of $3$-truncated symmetric sequences $A$ together with maps
\[ \begin{split}
    \theta_{(2)}: &A_1 \to \Tate_{\Sigma_2}(A_2) \\
    \theta_{(3)}: &A_1 \to \Tate_{\Sigma_3}(L_3 \smsh A_3) \\
    \theta_{(1,2)}: &A_2 \to \Tate_{\Sigma_2}(A_3).
\end{split} \]

\subsubsection*{Representable functors}

For $X \in \finspec$, the derivatives of $\Sigma^\infty \Hom_{\spectra}(X,-)$ have a $K$-coalgebra structure that encodes the Taylor tower of this functor. This structure consists of maps
\[ \theta_{r,n}: \dual(X^{\smsh r}) \to \hocolim_{L \to \infty} \left[ \left( \prod_{ \un{n} \epi \un{r}} S^{L(n-r)} \smsh \dual(X^{\smsh n}) \right)^{h\Sigma_n} \right] \]
If we set $X = \dual Y$, the Spanier-Whitehead dual of a finite spectrum $Y$, then these maps take the form
\begin{equation} \label{eq:diag} \theta_{r,n}: Y^{\smsh r} \to \hocolim_{L \to \infty} \left[ \left( \prod_{ \un{n} \epi \un{r}} S^{L(n-r)} \smsh Y^{\smsh n} \right)^{h\Sigma_n} \right]. \end{equation}
In particular, taking $r = 1$ we get maps
\[ \theta_{1,n}: Y \to \hocolim_{L \to \infty} \left[ \left( S^{L(n-1)} \smsh Y^{\smsh n} \right)^{h\Sigma_n} \right] \]
These can be thought of as generalized diagonal maps that exists for any finite spectrum $Y$.

\subsubsection*{The functor $\Sigma^\infty \Omega^\infty$}

The functor $\Sigma^\infty \Omega^\infty: \spectra \to \spectra$ is a special case of the representable functor in the previous section in the case $Y = S^0$. Thus the $K$-coalgebra structure maps for the derivatives of $\Sigma^\infty \Omega^\infty$  take the form
\begin{equation} \label{eq:SO} \theta_{(n_1,\dots,n_r)}: S^0 \to \hocolim_{L \to \infty} \left[ (S^{L(n-r)})^{h\Sigma_{n_1} \times \dots \times \Sigma_{n_r}} \right]. \end{equation}
We can then identify these maps as follows.

\begin{lemma}
The map $\theta_{(n_1,\dots,n_r)}$ associated to the $K$-coalgebra structure on $\der_*(\Sigma^\infty \Omega^\infty)$ is determined by the map
\[ S^0 \to \Map((B\Sigma_{n_1} \times \dots \times B\Sigma_{n_r})_+,S^0) = (S^0)^{h\Sigma_{n_1} \times \dots \times \Sigma_{n_r}} \]
dual to the collapse map
\[ B\Sigma_{n_1} \times \dots \times B\Sigma_{n_r} \to * \]
followed by the inclusion of $(S^0)^{h\Sigma_{n_1} \times \dots \times \Sigma_{n_r}}$ as the $L = 0$ term in the homotopy colimit (\ref{eq:SO}).
\end{lemma}
\begin{proof}
The map $\theta_{r,n}$ is given by applying $\der_r$ to the map $\eta_n$ of Proposition \ref{prop:spsp-eta} associated to $\Sigma^\infty \Omega^\infty$. In this case, $\eta_n$ is the map
\[ \Sigma^\infty \Omega^\infty X \to (X^{\smsh n})^{h\Sigma_n} \]
that is adjoint to the diagonal on the space $\Omega^\infty X$. Taking the \ord{r} cross-effect of this we get
\[ (\Sigma^\infty \Omega^\infty X_1) \smsh \dots \smsh (\Sigma^\infty \Omega^\infty X_r) \to  \left[ \prod_{\un{n} \epi \un{r}} X_1^{n_1} \smsh \dots \smsh X_r^{n_r} \right]^{h\Sigma_n} \]
adjoint to the diagonals on the spaces $\Omega^\infty X_j$. To get $\theta_{r,n}$ we now multilinearize and evaluate at $(S^0,\dots,S^0)$. But $\Sigma^\infty \Omega^\infty$ splits off the linear part on the suspension spectrum $S^0$ and so $\theta_{r,n}$ factors via the cross-effect itself evaluated at $(S^0,\dots,S^0)$ that is, via the composite
\[ S^0 = (S^0)^{\smsh r} \to (\Sigma^\infty \Omega^\infty S^0)^{\smsh r} \to \left[ \prod_{\un{n} \epi \un{r}} S^0 \right]^{h\Sigma_n} \]
followed by the inclusion of the right-hand side as the $L = 0$ term in the homotopy colimit (\ref{eq:SO}). In the above composite the first map is the inclusion of $S$ in $\Sigma^\infty \Omega^\infty S^0$ and the second is adjoint to the diagonal on $\Omega^\infty S^0$. The composite is the natural map into the homotopy fixed points which is dual to the collapse map as claimed.
\end{proof}

\subsubsection*{Functors with split Taylor tower}

We say that the Taylor tower of a functor $F: \finspec \to \spectra$ \emph{splits} if, for every $n$, we have
\[ P_n(F) \homeq \prod_{j = 1}^{n} D_j(F). \]
Since the comonad $K$ preserves finite products, the $K$-coalgebra structure on the product functor on the right-hand side is just the product of the $K$-coalgebra structures on the individual functors $D_j(F)$. Since $D_j(F)$ is homogeneous, this $K$-coalgebra structure is trivial. We can therefore express the splitting condition in the following way.

\begin{proposition}
The Taylor tower of $F: \finspec \to \spectra$ splits if and only if, for every $n$, the $K$-coalgebra structure on $\der_{\leq n}F$ is equivalent (in the homotopy category of derived $K$-coalgebras) to the trivial $K$-coalgebra structure in which
\[ \theta_{r,n}: \der_rF \to K_r \der_nF \]
is the trivial map for all $r < n$.
\end{proposition}
\begin{proof}
This follows from Theorem \ref{thm:n-excisive}.
\end{proof}

It should be noted that this condition is more subtle that just asking each map $\theta_{r,n}$ to be nullhomotopic though it certainly implies that fact. The nullhomotopies have to be coherent up to higher nullhomotopies. We do, however, recover a well known sufficient condition for the Taylor tower of $F$ to split.

\begin{proposition}\label{prop:splits}
Let $F: \finspec \to \spectra$ be a pointed simplicial functor with the property that, for each $n$, $\der_nF$ can be built from finitely many free $\Sigma_n$-cells. Then the Taylor tower of $F$ splits.
\end{proposition}
\begin{proof}
This and more general splitting results were studied by Chaoha in \cite{chaoha:2004} based on work of McCarthy \cite{mccarthy:2001}, but it follows from our work in the following way.

The condition on $F$ implies that $K_r \der_nF \homeq *$ for all $r < n$, so the counit map $\epsilon: K(\der_{\leq n}F) \to \der_{\leq n}F$ is an equivalence. Therefore in the cobar construction for recovering $P_nF$ we have an equivalence
\[ P_nF \weq \Tot(\Phi \der_{\leq n}F) \]
where the right-hand side is a constant cosimplicial object. The totalization of this is just equivalent again to
\[ (\Phi \der_{\leq n}F)(X) \homeq \prod_{j = 1}^{n} (\der_jF \smsh X^{\smsh j})^{h\Sigma_j} \]
which by the freeness of each $\der_jF$ is equivalent to
\[ \prod_{j = 1}^{n} (\der_jF \smsh X^{\smsh j})_{h\Sigma_j} = \prod_{j = 1}^{n} D_jF(X). \]
\end{proof}

\section{Functors from based spaces to spectra} \label{sec:topspec}

We now turn to the analysis of the Taylor towers of functors $F: \finbased \to \spectra$. We have a similar aim as in Section \ref{sec:specspec}, that is, to give a better description of what a $K$-coalgebra structure amounts to in this case, at least up to homotopy.

Recall that the derivatives of such a functor $F$ possess the structure of a right module over the operad $\der_*I$ formed by the derivatives of the identity functor on $\based$. We show here that the $K$-coalgebra structure on those coefficients includes and extends that module structure.

In this section, then, $K$ denotes the comonad on the category of symmetric sequences associated to the adjunction
\[  \der_*: [\finbased,\spectra] \rightleftarrows \symseq : \Phi \]
of Proposition \ref{prop:derphisp}. Our calculation of this comonad is given by the following result which we prove starting with \ref{def:bar} later in this section.

\begin{prop} \label{prop:topsp-KrAn}
In the classification of Taylor towers of functors $F: \finbased \to \spectra$, we have the following calculation. For a $\Sigma_n$-spectrum $A_n$, there are natural $\Sigma_r$-equivalences
\[ K_r A_n \homeq \left[ \prod_{\un{n} \epi \un{r}} \Map(\der_{n_1}I \smsh \dots \smsh \der_{n_r}I, A_n) \right]_{h\Sigma_n} \]
where the product is taken over the set of surjections $\alpha: \un{n} \epi \un{r}$ and we are writing $n_i := |\alpha^{-1}(i)|$ with the dependence on $\alpha$ understood.
\end{prop}

\begin{remark} \label{rem:topsp-KrAn}
Note the similarities between the above formula and the corresponding calculation for functors from spectra to spectra in Proposition \ref{prop:derAn-spsp}. As in that case, our formula breaks up into a product indexed by ordered partitions $n = n_1+\dots+n_r$ of $n$ into a sum of positive integers:
\[ K_r A_n \homeq \prod_{n = n_1+\dots+n_r} \left[ \Map(\der_{n_1}I \smsh \dots \smsh \der_{n_r}I, A_n) \right]_{h\Sigma_{n_1}\times\dots\times\Sigma_{n_r}}. \]
\end{remark}

\begin{definition} \label{def:psi}
It follows from Proposition \ref{prop:topsp-KrAn} that the $K$-coalgebra structure on the derivatives of a pointed simplicial functor $F: \finbased \to \spectra$ takes the form of a $\Sigma_r$-equivariant map
\[ \theta_{r,n}: \der_rF \to \left[ \prod_{\un{n} \epi \un{r}} \Map(\der_{n_1}I \smsh \dots \smsh \der_{n_r}I, \der_nF) \right]_{h\Sigma_n}. \]
for each pair of positive integers $r < n$.

On the other hand, the right $\der_*I$-module structure on $\der_*F$ can be encoded by $\Sigma_r$-equivariant maps
\[ \psi_{r,n}: \der_rF \to \left[ \prod_{\un{n} \epi \un{r}} \Map(\der_{n_1}I \smsh \dots \smsh \der_{n_r}I, \der_nF) \right]^{h\Sigma_n}. \]
These are adjoint to the usual right module structure maps $\der_rF \smsh \der_{n_1}I \smsh \dots \smsh \der_{n_r}I \to \der_nF$.
\end{definition}

The following result expresses the connection between the maps $\theta_{r,n}$ and $\psi_{r,n}$. Again the proof is given later in the section.

\begin{prop} \label{prop:norm}
For a pointed simplicial functor $F: \finbased \to \spectra$, there is a commutative diagram (in the stable homotopy category)
\[ \begin{diagram}
  \node[2]{ \left[ \prod_{\un{n} \epi \un{r}} \Map(\der_{n_1}I \smsh \dots \smsh \der_{n_r}I, \der_nF) \right]_{h\Sigma_n}} \arrow{s,r}{N} \\
  \node{\der_rF} \arrow{ne,t}{\theta_{r,n}} \arrow{e,b}{\psi_{r,n}}
    \node{ \left[ \prod_{\un{n} \epi \un{r}} \Map(\der_{n_1}I \smsh \dots \smsh \der_{n_r}I, \der_nF) \right]^{h\Sigma_n}}
\end{diagram} \]
where the vertical map $N$ is the norm map from homotopy orbits to homotopy fixed points. (See Greenlees-May \cite{greenlees/may:1995}.)
\end{prop}

\begin{remark} \label{rem:norm}
Thus we can think of a $K$-coalgebra in this case as consisting of a right $\der_*I$-module together with (compatible) lifts up the norm map. The situation is analogous (but dual) to the relationship between divided power algebras and ordinary commutative algebras, in which the divided power structure can be expressed as an extension along the norm map. (See, for example, Fresse \cite{fresse:2000}.) We can therefore view the structure of a $K$-coalgebra as providing a rigid notion of \emph{divided power right $\der_*I$-module}. Further justification for this is given by Proposition \ref{prop:delta-topsp} where we show that the comonad structure map for $K$ is determined (at least up to homotopy) by the operad composition map for $\der_*I$. The main conclusion of this section can now be expressed as saying that $n$-excisive functors from based spaces to spectra are classified by the divided power right $\der_*I$-module structure on their derivatives.
\end{remark}

Let us turn then to an explicit description of the comonad structure on the functor $K$ with respect to the equivalences of Proposition \ref{prop:topsp-KrAn}. To state our answer, it is convenient to give a formula for iterates of $K$. For this we need some new notation.

\begin{definition} \label{def:seqs}
For integers $r_0 < r_1 < \dots < r_t$ we put an equivalence relation on the set of sequences of surjections
\[ \un{r_t} \epi \un{r_{t-1}} \epi \dots \epi \un{r_0}. \]
We say that two such sequences $(\alpha_t,\dots,\alpha_1)$ and $(\beta_t,\dots,\beta_1)$ are equivalent if they fit into a commutative diagram
\[ \begin{diagram}
  \node[2]{\un{r_{t-1}}} \arrow{e,t}{\alpha_{t-1}} \arrow[2]{s,r}{\isom} \node{\dots} \arrow{e,t}{\alpha_{2}} \node{\un{r_1}} \arrow[2]{s,l}{\isom} \arrow{se,t}{\alpha_{1}} \\
  \node{\un{r_t}} \arrow{ne,t}{\alpha_{t}} \arrow{se,b}{\beta_{t}} \node[4]{\un{r_0}} \\
  \node[2]{\un{r_{t-1}}} \arrow{e,b}{\beta_{t-1}} \node{\dots} \arrow{e,b}{\beta_{2}} \node{\un{r_1}} \arrow{ne,b}{\beta_1}
\end{diagram} \]
where the vertical maps are bijections. We write
\[ [\un{r_t} \epi \dots \epi \un{r_0}] \]
for the set of equivalence classes of such sequences. In the case $t = 1$ notice that $[\un{n} \epi \un{r}]$ is just the set of surjections $\alpha: \un{n} \epi \un{r}$ as previously.
\end{definition}

\begin{proposition} \label{prop:topsp-KK}
For a $\Sigma_n$-spectrum $A_n$, and integers $r = r_0 < r_1 < \dots < r_t = n$, we have
\[ K_{r_0}K_{r_1} \dots K_{r_{t-1}}A_n \homeq \left[ \prod_{[\un{r_t} \epi \dots \epi \un{r_0}]} \Map \left(\Smsh_{i = 1}^{t} (\der_{r_{i1}}I \smsh \dots \smsh \der_{r_{ir_i}}I), A_n \right) \right]_{h\Sigma_n} \]
where the product is taken over equivalence classes of sequences $(\alpha_t,\dots,\alpha_1)$ as in Definition \ref{def:seqs}, and we are writing
\[ r_{ij} := |\alpha_i^{-1}(j)|. \]
Note that these numbers are, up to reordering, independent of the representative chosen for the equivalence class. In particular
\[ K_r K_s A_n \homeq \left[ \prod_{[\un{n} \epi \un{s} \epi \un{r}]} \Map(\der_{s_1}I \smsh \dots \smsh \der_{s_r}I \smsh \der_{n_1}I \smsh \dots \smsh \der_{n_s}I, A_n) \right]_{h\Sigma_n} \]
where the product is over equivalence classes of pairs of surjections
\[ \un{n} \arrow{e,t,>}{\beta} \un{s} \arrow{e,t,>}{\gamma} \un{r} \]
and we are writing $n_i := |\beta^{-1}(i)|$ and $s_j := |\gamma^{-1}(j)|$.
\end{proposition}
\begin{proof}
This follows by iterating the formula given in Proposition \ref{prop:topsp-KrAn}, and using the fact that $\Sigma_s$ acts freely on the set of surjections $\un{n} \epi \un{s}$ (and that homotopy colimits commute).
\end{proof}

\begin{prop} \label{prop:delta-topsp}
With respect to the equivalences of \ref{prop:topsp-KrAn} and \ref{prop:topsp-KK}, the comultiplication map
\[ \delta_{r,s}: K_r A_n \to K_r K_s A_n \]
associated to the comonad $K$, takes the form
\[ \begin{diagram}
  \node{ \left[ \prod_{\un{n} \epi \un{r}} \Map(\der_{n'_1}I \smsh \dots \smsh \der_{n'_r}I, A_n) \right]_{h\Sigma_n} } \arrow{e,t}{\delta_{r,s}}
  \node{ \left[ \prod_{[\un{n} \epi \un{s} \epi \un{r}]} \Map(\der_{s_1}I \smsh \dots \smsh \der_{n_1}I \smsh \dots , A_n) \right]_{h\Sigma_n}}
\end{diagram} \]
where $n'_i := |\alpha^{-1}(i)|$. This map is given by applying $\Sigma_n$-homotopy orbits to a map between the products constructed in the following way. We can compose a sequence of surjections
\[ \un{n} \epi \un{s} \epi \un{r} \]
to get a single surjection $\un{n} \epi \un{r}$ and this composite is the same for all representatives of an element in $[\un{n} \epi \un{s} \epi \un{r}]$. Associated to this composition are operad composition maps
\[ \der_{s_j}I \smsh \der_{n_{s_1+\dots+s_{j-1}+1}}I \smsh \dots \smsh \der_{n_{s_1+\dots+s_j}}I \to \der_{n'_j}I \]
which yield the required map $\delta_{r,s}$.
\end{prop}

Before turning to the proofs, we show that the relationship between $K$-coalgebras and right $\der_*I$-module can be described in terms of a map of comonads that is given, up to homotopy, by the norm maps of Proposition \ref{prop:norm}.

\begin{definition} \label{def:module-comonad}
We can define a comonad $K'$ on the category of symmetric sequences whose coalgebras are precisely the right $\der_*I$-modules. On the symmetric sequence $A$, $K'$ is given by
\[ K'(A)_r := \prod_{n} \left[ \prod_{\un{n} \epi \un{r}} \Map(\der_{n_1}I \smsh \dots \smsh \der_{n_r}I, A_n) \right]^{\Sigma_n}. \]
This functor is right adjoint to the free right $\der_*I$-module monad and so for formal reasons inherits a comonad structure whose coalgebras are the right $\der_*I$-modules. The $K'$-coalgebra structure on $\der_*F$ is a map of symmetric sequences $\der_*F \to K'(\der_*F)$ that is captured exactly by the maps $\psi_{r,n}$ of Definition \ref{def:psi}. (We use a cofibrant model for $\der_*I$ so that the strict fixed points in the definition of $K'$ are equivalent to the homotopy fixed points.)

Now the comonad $K = \der_* \Phi$ takes values in right $\der_*I$-modules (because $\der_*$ does). This gives us a map
\[ K \to K'K \]
and composing with the counit for $K$ we get a natural transformation
\[ \nu: K \to K'. \]
Notice that $K'$ commutes with products so that it is determined by constructions $K'_r A_n$ for $r < n$ analogous to those of Definition \ref{def:derAn} for $K$. Then $\nu$ restricts to maps $\nu_r: K_rA_n \to K'_rA_n$.
\end{definition}

\begin{lemma} \label{lem:map-comonads}
The map $\nu: K \to K'$ is a morphism of comonads and the right $\der_*I$-module structure on a $K$-coalgebra $A$ is encoded in the composite map
\[ A \to KA \arrow{e,t}{\nu} K'A. \]
\end{lemma}
\begin{proof}
This is a diagram chase using the naturality of the $K'$-coalgebra structure on the values of $K$.
\end{proof}

\begin{proposition} \label{prop:comonad-norm}
With respect to the equivalence of Proposition \ref{prop:topsp-KrAn}, the map
\[ \nu_r: K_rA_n \to K'_rA_n \]
is given, in the homotopy category, by the norm maps
\[ N: \left[ \prod_{\un{n} \epi \un{r}} \Map(\der_{n_1}I \smsh \dots \smsh \der_{n_r}I, A_n) \right]_{h\Sigma_n} \to \left[ \prod_{\un{n} \epi \un{r}} \Map(\der_{n_1}I \smsh \dots \smsh \der_{n_r}I, A_n) \right]^{h\Sigma_n} \]
that appear in Proposition \ref{prop:norm}.
\end{proposition}
\begin{proof}
We apply Proposition \ref{prop:norm} to the functor $F = \Phi A_n$. This tells us that the $K'$-coalgebra structure on $KA_n$ is given, up to homotopy, by the composite
\[ KA_n \to KKA_n \to K'KA_n \]
where the first map is the $K$-coalgebra structure on $KA_n$ (i.e. the comonad structure map for $K$) followed by the norm map $N$. By definition then, $\nu$ is given by following the above composite with the counit for $K$. This gives the composite
\[ KA_n \to KKA_n \to K'KA_n \to K'A_n. \]
Now Proposition \ref{prop:delta-topsp} tells us what the first map is (up to homotopy), the second map is $N$ and the third is projection on to the terms with $s = n$. That is, we get
\[ \begin{split}
  \left[ \prod_{\alpha: \un{n} \epi \un{r}} \Map(\der_{n'_1}I \smsh \dots \smsh \der_{n'_r}I, A_n) \right]_{h\Sigma_n} &\arrow{e,t}{\delta_{r,s}}
  \left[ \prod_{[\un{n} \epi \un{s} \epi \un{r}]} \Map(\der_{s_1}I \smsh \dots \smsh \der_{n_1}I \smsh \dots \smsh \der_{n_s}I, A_n) \right]_{h\Sigma_n} \\
    &\arrow{e,t}{N}
  \left[ \prod_{[\un{n} \epi \un{s} \epi \un{r}]} \Map(\der_{s_1}I \smsh \dots \smsh \der_{n_1}I \smsh \dots \smsh \der_{n_s}I, A_n) \right]^{h\Sigma_n} \\
    &\arrow{e}
  \left[ \prod_{\alpha: \un{n} \epi \un{r}} \Map(\der_{n'_1}I \smsh \dots \smsh \der_{n'_r}I, A_n) \right]^{h\Sigma_n}
\end{split} \]
By naturality of $N$, this composite is just the norm map $N$ as claimed.
\end{proof}

We now turn to the proofs of Propositions \ref{prop:topsp-KrAn}, \ref{prop:norm} and \ref{prop:delta-topsp}. The first step is to understand the derivatives of the representable functors $\Sigma^\infty \Hom_{\based}(X,-)$ for $X \in \finbased$. These were calculated by the first author in \cite{arone:1999}. The right $\der_*I$-module structure was calculated in \cite{arone/ching:2011}. We recall both descriptions here.

\begin{definition} \label{def:fat-diagonal}
For $X \in \finbased$ we have the \emph{fat diagonal} $\Delta^n X \subseteq X^{\smsh n}$ consisting of $n$-tuples $(x_1,\dots,x_n)$ of points in $X$ with some $x_i = x_j$ for $i \neq j$. In \cite{arone:1999}, the first author showed that
\[ \der_n(\Sigma^\infty \Hom_{\based}(X,-)) \homeq \dual (X^{\smsh n}/\Delta^n X) \]
where $\dual$ denotes the Spanier-Whitehead dual.
\end{definition}

\begin{definition} \label{def:bar}
To describe the $\der_*I$-module structure on these derivatives, we recall some facts about operadic bar constructions. For $X \in \finbased$, the symmetric sequence $\Sigma^\infty X^{\smsh *}$ forms a right module over the commutative operad $\mathsf{Com}$. (This is the operad of spectra all of whose terms are the sphere spectrum.) By \cite{ching:2005}, the one-sided bar construction
\[ B(\Sigma^\infty X^{\smsh *},\mathsf{Com},1) \]
has the structure of a right comodule over the cooperad $B(1,\mathsf{Com},1)$. Applying Spanier-Whitehead duality, we get a right module
\[ \dual B(X^{\smsh *},\mathsf{Com},1) \]
over the operad $\dual B(1,\mathsf{Com},1)$. The terms of this operad are equivalent to the derivatives of the identity functor on based spaces.

Here we write $\der_*I$ for a cofibrant model of the operad $\dual B(1,\mathsf{Com},1)$ (in the projective model structure), and $M(X)$ for a cofibrant replacement of $\dual B(X^{\smsh *},\mathsf{Com},1)$ in the projective model structure on right $\der_*I$-modules. In fact, we choose $M(-)$ to be a cofibrant object in the category of simplicially-enriched functors from $(\finbased)^{op}$ to right $\der_*I$-modules (with its projective model structure).
\end{definition}

\begin{lemma} \label{lem:der-topsp}
For $X \in \finbased$, the right $\der_*I$-module $M(X)$ is equivalent to the right $\der_*I$-module formed by the derivatives of the representable functor $\Sigma^\infty \Hom_{\based}(X,-)$.
\end{lemma}
\begin{proof}
This calculation is done in \cite[4.2.28]{arone/ching:2011}. Note that it relies on a $\Sigma_n$-equivariant equivalence
\[ B(\Sigma^\infty X^{\smsh *},\mathsf{Com},1)_n \homeq \Sigma^\infty X^{\smsh n}/\Delta^n X \]
which we use later in this section.
\end{proof}

We now follow the usual pattern to construct the required adjunction $(\der_*,\Phi)$ from our choice of derivatives for the representable functors.

\begin{definition} \label{def:module-der-topsp}
We define
\[ \der_*: [\finbased,\spectra] \to \symseq \]
by
\[ \der_*F := M(X) \smsh_{X \in \finbased} F(X) \]
and its right adjoint
\[ \Phi: \symseq \to [\finbased, \spectra] \]
is then given by
\[ \Phi(A): X \mapsto \Map_{\Sigma}(M(X),A) = \prod_{n} \Map(M(X)_n,A_n)^{\Sigma_n}. \]
\end{definition}

\begin{corollary} \label{cor:topsp-phi}
The right adjoint $\Phi: \symseq \to [\finbased,\spectra]$ satisfies
\[ \Phi(A)(X) \homeq \prod_{n} (A_n \smsh X^{\smsh n}/\Delta^n X)^{h\Sigma_n} \]
where $\Delta^n X$ denotes the fat diagonal inside $X^{\smsh n}$.
\end{corollary}
\begin{proof}
This follows from the equivalence mentioned in the proof of Lemma \ref{lem:der-topsp}.
\end{proof}

\begin{remark}
The category of right $\der_*I$-modules is enriched, tensored and cotensored over $\spectra$ with the tensoring given by termwise smash product. Since colimits are also calculated at the symmetric sequence level, it follows that $\der_*F$, as defined above, inherits a right $\der_*I$-module structure from that on $M(X)$. This is equivalent to the right $\der_*I$-module structure constructed in \cite{arone/ching:2011} because that construction preserves homotopy colimits.
\end{remark}

The pieces of the comonad $K$ can now be identified as the spectra
\[ K_r A_n := M(X)_r \smsh_{X \in \finbased} \Map(M(X)_n,A_n)^{\Sigma_n} \]
and since $M(X)_n$ is a cofibrant $\Sigma_n$-spectrum, the fixed points here are equivalent to the homotopy fixed points.

Our next step is to use the above expression for $K_r A_n$ to prove Proposition \ref{prop:topsp-KrAn}. The required equivalence arises from the following commutative diagram:
\begin{equation} \label{eq:norm}
  \begin{diagram}
    \node{M(X)_r \smsh_{X} [\Map(M(X)_n,A_n)_{h\Sigma_n}]} \arrow{e,tb}{m}{\sim} \arrow{s,lr}{N}{\sim}
      \node{\left[\prod_{\alpha: \un{n} \epi \un{r}} \Map(\der_{n_1}I \smsh \dots \smsh \der_{n_r}I, A_n) \right]_{h\Sigma_n}} \arrow{s,r}{N} \\
    \node{M(X)_r \smsh_{X} [\Map(M(X)_n,A_n)^{h\Sigma_n}]} \arrow{e,t}{m}
      \node{\left[\prod_{\alpha: \un{n} \epi \un{r}} \Map(\der_{n_1}I \smsh \dots \smsh \der_{n_r}I, A_n) \right]^{h\Sigma_n}}
  \end{diagram}
\end{equation}

The maps marked $N$ are the norm maps from homotopy orbits to homotopy fixed points and those marked $m$ are determined by the composite
\begin{equation} \begin{split} \label{eq:m} M(X)_r \smsh \Map(M(X)_n,A_n) &\to M(X)_r \smsh \Map(M(X)_r \smsh \der_{n_1}I \smsh \dots \smsh \der_{n_r}I, A_n) \\
    &\to \Map(\der_{n_1}I \smsh \dots \smsh \der_{n_r}I, A_n) \end{split}
\end{equation}
where the first map is the right module composition map associated to a surjection $\alpha: \un{n} \epi \un{r}$ and the second is the canonical evaluation map.

Proposition \ref{prop:topsp-KrAn} follows from the claim that the top and left-hand maps in diagram (\ref{eq:norm}) are equivalences. We prove these facts now.

\begin{lemma} \label{lem:free-finite}
For any $\Sigma_n$-spectrum $A_n$ and $X \in \finbased$, the norm map
\[ N: \Map(M(X)_n,A_n)_{h\Sigma_n} \to \Map(M(X)_n,A_n)^{h\Sigma_n} \]
is a weak equivalence.
\end{lemma}
\begin{proof}
By the equivalence of $M(X)_n \homeq \dual (X^n/\Delta^n X)$ and the finiteness of $X$, this map can be rewritten as
\[ N: (A_n \smsh X^{\smsh n}/\Delta^n X)_{h\Sigma_n} \to (A_n \smsh X^{\smsh n}/\Delta^n X)^{h\Sigma_n} \]
But $X^{\smsh n}/\Delta^n X$ is a finite free cell-$\Sigma_n$-space in the sense that it is built from finitely many free $\Sigma_n$-cells. It follows that the norm map $N$ above is an equivalence as claimed.
\end{proof}

\begin{notation} \label{not:Tn}
To analyze the top horizontal map $m$ in (\ref{eq:norm}) we introduce the following notation. We write
\[ T_n := B(1,\mathsf{Com},1)_n. \]
This is the suspension spectrum of the \ord{n} partition poset complex. Recall that the symmetric sequence $T_*$ forms a cooperad whose Spanier-Whitehead dual is $\der_*I$, and that we have a comodule structure map
\[ B(\Sigma^\infty X^{\smsh *},\mathsf{Com},1)_n \to B(\Sigma^\infty X^{\smsh *},\mathsf{Com},1)_r \smsh T_{n_1} \smsh \dots \smsh T_{n_r} \]
for each surjection $\un{n} \epi \un{r}$. These maps define the right $\der_*I$-module structure on $M(X)$.
\end{notation}

\begin{lemma} \label{lem:topsp1}
The map
\[  B(\Sigma^\infty X^{\smsh *},\mathsf{Com},1)_n \to \left[ \prod_{\un{n} \epi \un{r}} B(\Sigma^\infty X^{\smsh *},\mathsf{Com},1)_r \smsh T_{n_1} \smsh \dots \smsh T_{n_r} \right]^{h\Sigma_r} \]
induced by the comodule structure maps, determines an equivalence of \ord{r} derivatives.
\end{lemma}
\begin{proof}
Since taking derivatives commutes with the various colimit operations performed (and since $\Sigma_r$ is acting freely on the product here), the result can be reduced to the claim that the map
\[  B(\der_r(\Sigma^\infty X^{\smsh *}),\mathsf{Com},1)_n \to \left[ \prod_{\un{n} \epi \un{r}} B(\der_r(\Sigma^\infty X^{\smsh *}),\mathsf{Com},1)_r \smsh T_{n_1} \smsh \dots \smsh T_{n_r} \right]^{h\Sigma_r} \]
is an equivalence. But the right $\mathsf{Com}$-module $\der_r(\Sigma^\infty X^{\smsh *})$ is trivial and concentrated in degree $r$ where it equals $\Sigma^\infty (\Sigma_r)_+$. The claim then follows by a straightforward calculation with the bar construction.
\end{proof}

\begin{lemma} \label{lem:topsp2}
There is a $\Sigma_r$-equivariant equivalence
\[ \epsilon: M(X)_r \smsh_{X \in \finbased} B(X^{\smsh *},\mathsf{Com},1)_r \weq \prod_{\Sigma_r} S   \]
where the target has the regular $\Sigma_r$-action. The component corresponding to the identity element in $\Sigma_r$ is made up of the canonical evaluation maps
\[ \dual B(X^{\smsh *},\mathsf{Com},1)_r \smsh B(X^{\smsh *},\mathsf{Com},1)_r \to S. \]
\end{lemma}
\begin{proof}
Let $[r]$ denote the finite pointed set $\{0,1,\dots,r\}$ viewed as an object of $\finbased$. Our strategy is to construct an equivalence
\[ \phi: M([r])_r \weq M(X)_r \smsh_{X \in \finbased} B(X^{\smsh *},\mathsf{Com},1)_r \]
and show that the composite $\epsilon \phi$ is an equivalence $M([r])_r \weq \prod_{\Sigma_r} S$.

First note that since the bar construction preserves colimits in its right module variable, so the target of $\phi$ is isomorphic to
\[ B(M(X)_r \smsh_{X} X^{\smsh *},\mathsf{Com},1)_r. \]
Here $M(X)_r \smsh_{X} X^{\smsh *}$ inherits a right $\mathsf{Com}$-module structure from that on $X^{\smsh *}$. Notice also that to calculate the \ord{r} term in the bar construction we only care about the $r$-truncated part of this right $\mathsf{Com}$-module.

We now observe that the smash product $X^{\smsh s}$ can be written (as a colimit of representable functors) as:
\[ X^{\smsh s} \isom \tcof_{J \subseteq [s]} \Hom(J_+,X). \]
The right-hand side is the total cofibre of an $s$-dimensional cube whose morphisms are given by extending functions to the basepoint in $X$. It follows by an enriched dual Yoneda Lemma that
\[ M(X)_r \smsh_{X} X^{\smsh s} \isom \tcof_{J \subseteq [s]} M(J_+)_r. \]
The key observation is now that
\[ M(J_+)_r \homeq * \quad \text{if $|J| < r$}. \]
(To see this recall that $M(X)_r$ is equivalent to the dual of $X^{\smsh r}/\Delta^r X$ which is the set of configurations of $r$ points in $J$ (plus a disjoint basepoint) when $X$ is the finite set $J_+$. There are no such configurations when $|J| < r$.)

It follows that there is a natural equivalence of $r$-truncated right $\mathsf{Com}$-modules
\[ M([r])_r \weq M(X)_r \smsh_{X} X^{\smsh *} \]
where the source here is considered to be an $r$-truncated symmetric sequence that is trivial except in the \ord{r} term. Applying the bar construction to this we get an equivalence
\[ B(M([r])_r, \mathsf{Com},1)_r \weq B(M(X)_r \smsh_{X} X^{\smsh *}, \mathsf{Com},1)_r \]
but the left-hand side is again just isomorphic to $M([r])_r$.

Altogether then we have constructed an equivalence
\[ \phi: M([r])_r \weq M(X)_r \smsh_{X} B(X^{\smsh *},\mathsf{Com},1)_r. \]
The map $\phi$ can also be expressed as the composite
\[ M([r])_r \to M([r])_r \smsh B([r]^{\smsh *},\mathsf{Com},1)_r \to M(X)_r \smsh_{X} B(X^{\smsh *},\mathsf{Com},1)_r \]
where the first map is inclusion via the identity map on $\{1,\dots,r\}$ viewed as a point $e$ in $[r]^{\smsh r}$ included in the $0$-simplices of the bar construction, and the second is the natural map into the coend.

It remains to analyze the composite of $\phi$ with $\epsilon$:
\[ M([r])_r \to \prod_{\Sigma_r} S. \]
But this composite is precisely the composite
\[ M([r])_r \weq \dual B([r]^{\smsh *},\mathsf{Com},1)_r \weq \dual([r]^{\smsh r}/\Delta^r [r]) \isom \prod_{\Sigma_r} S \]
so is an equivalence.
\end{proof}

The heavy lifting is now done and we can complete the proofs of our main results.

\begin{proof}[Proof of \ref{prop:topsp-KrAn}]
Lemmas \ref{lem:topsp1} and \ref{lem:topsp2} together imply that the map top horizontal map in (\ref{eq:norm}) is a weak equivalence. Lemma \ref{lem:free-finite} implies that the left-hand vertical map in that diagram is an equivalence. Together these yield the claimed formula.
\end{proof}

\begin{proof}[Proof of \ref{prop:norm}]
We have just seen that the equivalence of Proposition \ref{prop:topsp-KrAn} is based on the diagram in (\ref{eq:norm}). Consider now the following diagram for some fixed surjection $\un{n} \epi \un{r}$.
\[ \begin{diagram} \dgARROWLENGTH=1em
  \node{M(X)_r \smsh_{X} FX} \arrow{e} \arrow{s}
    \node{M(X)_r \smsh_{X} \Map(M(X)_n, M(Y)_n \smsh_{Y} FY)} \arrow{s} \\
  \node{M(X)_r \smsh_{X} \Map(M(X)_r, M(Y)_r \smsh_{Y} FY)} \arrow{e} \arrow{s}
    \node{M(X)_r \smsh_{X} \Map(M(X)_r \smsh \der_{n_1}I \smsh \dots, M(Y)_n \smsh_{Y} FY)} \arrow{s} \\
  \node{M(Y)_r \smsh_{Y} FY} \arrow{e}
    \node{\Map(\der_{n_1}I \smsh \dots \smsh \der_{n_r}I, M(Y)_n \smsh_{Y} FY)}
\end{diagram} \]
where
\begin{itemize}
  \item the top horizontal and top-left vertical maps are the unit of the $(\der_*,\Phi)$-adjunction;
  \item the top-right vertical and middle/bottom horizontal maps are the $\der_*I$-module structure on $M(X)$ and $M(Y)$ respectively;
  \item the bottom-left and bottom-right vertical maps are canonical evaluations.
\end{itemize}
The top square commutes because a map $X \to Y$ of spaces induces a map $M(X) \to M(Y)$ of right $\der_*I$-modules, and the bottom square commutes by naturality of the maps involved.

The composite $N \circ \theta_{r,n}$ appearing in the statement of Proposition \ref{prop:norm} is essentially the composite of the top and right-hand maps. (Strictly speaking, to get $N \circ \theta_{r,n}$ we consider all surjections $\un{n} \epi \un{r}$ and map into the $\Sigma_n$-homotopy fixed points of the product over these.) On the other hand the composite of the left-hand maps is the identity (by a triangle identity for the $(\der_*,\Phi)$ adjunction), and the bottom map is the definition of the right $\der_*I$-module structure on $\der_*F$. The commutativity of this diagram thus implies the Proposition.
\end{proof}

\begin{proof}[Proof of \ref{prop:delta-topsp}]
The comonad structure map $\delta_{r,s}$ is exactly the map $\theta_{r,s}$ associated to the functor $F(X) = \Map(M(X)_n,A_n)_{\Sigma_n}$. By naturality this is given by applying the $\Sigma_n$-homotopy orbits to the map $\theta_{r,s}$ associated to $\Map(M(X)_n,A_n)$. The proof of Proposition \ref{prop:topsp-KrAn} shows us that the derivatives of this functor are given by
\begin{equation} \label{eq:ds} \der_s \Map(M(X)_n,A_n) \homeq \prod_{\un{n} \epi \un{s}} \Map(\der_{n_1}I \smsh \dots \smsh \der_{n_s}I, A_n) \end{equation}
and $\Sigma_s$ acts freely on this. The norm map in Proposition \ref{prop:norm} is therefore an equivalence and so the map $\theta_{r,s}$ for this functor is given by the right $\der_*I$-module structure on these derivatives. Now recall that the equivalence (\ref{eq:ds}) is built from the maps $m$ of (\ref{eq:m}). From this it follows that the right $\der_*I$-module structure on these derivatives is determined by the operad structure on $\der_*I$ in the claimed manner.
\end{proof}

We now turn to calculations of the $K$-coalgebra structure maps $\theta_{r,n}$ for specific functors $F: \finbased \to \spectra$.

\subsubsection*{$2$-excisive functors}
A $2$-excisive pointed simplicial functor $F: \finbased \to \spectra$ is determined by a $2$-term symmetric sequence $A_1,A_2$ together with a single map of spectra
\[ \theta_{1,2}: A_1 \to (T_2 \smsh A_2)_{h\Sigma_2} = (\Sigma A_2)_{h\Sigma_2} \]
where $\Sigma_2$ acts trivially on the suspension coordinate and $T_n$ denotes the \ord{n} partition poset complex. According to Proposition \ref{prop:norm}, $\theta_{1,2}$ is determined by a right $\der_*I$-module structure map
\[ A_1 \smsh \der_2I \to A_2 \]
together with a nullhomotopy of the composite
\[ A_1 \to \Map(\der_2I,A_2)^{h\Sigma_2} \to \Tate_{\Sigma_2}\Map(\der_2I,A_2) = \Tate_{\Sigma_2}(\Sigma A_2) \]
that yields a lift up the norm map.

Notice that this structure map can also be recovered directly form the fibre sequence
\[ D_2F \to P_2F \to P_1F. \]
Delooping this we see that $P_2F$ is the fibre of a map
\[ (A_1 \smsh X) = P_1F \to \Sigma D_2F = (\Sigma A_2 \smsh X^{\smsh 2})_{h\Sigma_2}. \]
Evaluating at $X = S^0$ we get the map $\theta_{1,2}$.

\subsubsection*{$3$-excisive functors}
A $3$-excisive pointed simplicial functor $F: \finbased \to \spectra$ is determined by a $3$-term symmetric sequence $A_1,A_2,A_3$, the map $\theta_{1,2}$ described above, and two further maps
\[ \theta_{2,3}: A_2 \to \left[ \prod_{\un{3} \epi \un{2}} T_1 \smsh T_2 \smsh A_3 \right]_{h\Sigma_3} \]
and
\[ \theta_{1,3}: A_1 \to (T_3 \smsh A_3)_{h\Sigma_3}. \]
These maps make the following diagram commute
\[ \begin{diagram}
  \node{A_1} \arrow{e,t}{\theta_{1,2}} \arrow{s,l}{\theta_{1,3}} \node{(T_2 \smsh A_2)_{h\Sigma_2}} \arrow{s,r}{\theta_{2,3}} \\
  \node{(T_3 \smsh A_3)_{h\Sigma_3}} \arrow{e,t}{\delta_{2,3}} \node{ \left[ \prod_{[\un{3} \epi \un{2} \epi \un{1}]} T_2 \smsh T_1 \smsh T_2 \smsh A_3 \right]_{h\Sigma_3}}
\end{diagram} \]
where the bottom horizontal map is given by the cooperad structure maps $T_3 \to T_2 \smsh T_1 \smsh T_2$.

\subsubsection*{Representable functors}

For $X \in \finbased$, the derivatives of the functor $\Sigma^\infty \Hom_{\based}(X,-)$ are given by the right $\der_*I$-module
\[ M(X) \homeq \dual B(X^{\smsh *},\mathsf{Com},1) \homeq \dual X^{\smsh *}/\Delta^* X. \]
By Proposition \ref{prop:norm} the $K$-coalgebra structure on these coefficients is a lift up the norm map of this right $\der_*I$-module structure. In this case, however, the norm map
\[ \left[ \prod_{\un{n} \epi \un{r}} \Map(\der_{n_1}I \smsh \dots \smsh \der_{n_r}I, M(X)_n) \right]_{h\Sigma_n} \arrow{e,t}{N} \left[ \prod_{\un{n} \epi \un{r}} \Map(\der_{n_1}I \smsh \dots \smsh \der_{n_r}I, M(X)_n) \right]^{h\Sigma_n} \]
is an equivalence by a similar argument to that of Lemma \ref{lem:free-finite}. The $K$-coalgebra structure is therefore completely determined by the right $\der_*I$-module structure on $M(X)$.

\subsubsection*{Functors of the form $F\Sigma^\infty$}

A pointed simplicial functor $F: \finspec \to \spectra$ determines a pointed simplicial functor $\finbased \to \spectra$ by precomposing with $\Sigma^\infty$. The functor $F\Sigma^\infty$ has the same derivatives as $F$. The symmetric sequence $\der_*F \homeq \der_*(F\Sigma^\infty)$ therefore has actions by both the comonads described in this and the previous section. To avoid confusion we denote these here by $K^{\based}$ and $K^{\spectra}$ respectively.

\begin{example} (First observed by Bill Dwyer)
For any finite spectrum $Y$, we described in (\ref{eq:diag}) the $K^{\spectra}$-coalgebra structure on the derivatives of the functor
\[ \Sigma^\infty \Omega^\infty (Y \smsh -) \homeq \Sigma^\infty \Hom_{\spectra}(\dual Y,-). \]
Considering this now as a functor $\finbased \to \spectra$, these derivatives inherit a $K^{\based}$-coalgebra structure. This coalgebra structure consists of maps
\[ \theta_{r,n}: Y^{\smsh r} \to \Map(\der_{n_1}I \smsh \dots \smsh \der_{n_r}I, Y^{\smsh n})_{h\Sigma_{n_1} \times \dots \times \dots \Sigma_{n_r}}. \]
and so, in particular, we have
\[ \theta_{1,n}: Y \to \Map(\der_nI,Y^{\smsh n})_{h\Sigma_n}. \]
Comparing with the terminology of Remark \ref{rem:norm}, we can refer to this structure as making any finite spectrum $Y$ into a \emph{divided power right $\der_*I$-coalgebra}.
\end{example}

\subsubsection*{Functors with vanishing Tate data}

Suppose $G$ is a functor whose derivatives have the following property: the $\Sigma_n$-spectrum $\der_nG$ can be built from finitely many free $\Sigma_n$-cells. It follows that the norm maps of Proposition \ref{prop:norm} are equivalences and so the $K$-coalgebra structure on $\der_*G$ is determined by the right $\der_*I$-module structure. The Taylor tower is thus also determined by this information. The following theorem gives an explicit expression for $P_nG$ in terms of the right $\der_*I$-module $\der_*G$.

\begin{theorem} \label{thm:tate-topsp}
Let $\Map_{\der_*I}(-,-)$ denote the derived mapping spectrum for right $\der_*I$-modules. Suppose $G: \finbased \to \spectra$ is a pointed simplicial functor such that $\der_nG$ can be built from finitely many free $\Sigma_n$-cells. Then
\[ P_nG(X) \homeq \Map_{\der_*I}(M(X),\der_{\leq n}G) \]
with the maps in the Taylor tower given by the sequence of truncation maps
\[ \dots \to \der_{\leq n}G \to \der_{\leq (n-1)}G \to \dots . \]
Moreover, if the Taylor tower of $G$ converges, then there is an equivalence of spectra
\[ \Nat_{X \in \finbased}(F,G) \homeq \Map_{\der_*I}(\der_*F,\der_*G). \]
The left-hand side is the \emph{spectrum} of natural transformations for two functors $\finbased \to \spectra$.
\end{theorem}

\begin{remark}
Proposition~\ref{prop:splits} and Theorem~\ref{thm:tate-topsp} are examples of a general phenomenon: if $G$ is a functor whose derivatives have vanishing Tate homology, then the Taylor tower of $G$ is determined by the $\der_*I$-module structure on $\der_*G$. This observation also applies to space-valued functors and we intend to develop this remark in another paper.
\end{remark}

\begin{example}
For any $Y \in \finbased$, the representable functor $G = \Sigma^\infty R_Y$ satisfies the hypothesis of Theorem \ref{thm:tate-topsp}. It follows that
\[ P_n(\Sigma^\infty \Hom_{\based}(Y,-))(X) \homeq \Map_{\der_*I}(M(X),M(Y)_{\leq n}) \]
and, if the connectivity of $X$ is larger than the dimension of $Y$:
\[ \Sigma^\infty \Hom_{\based}(Y,X) \homeq \Map_{\der_*I}(M(X),M(Y)). \]
This description of the Taylor tower of the stable mapping functors is Koszul dual to that given in \cite{arone:1999} in which the stages of this tower are described as mapping objects for the right $\mathsf{Com}$-modules $X^{\smsh *}$.
\end{example}

\begin{proof}[Proof of \ref{thm:tate-topsp}]
Recall the comonad $K'$ of Definition \ref{def:module-comonad} whose coalgebras are precisely the right $\der_*I$-modules. The derived mapping spectrum $\Map_{\der_*I}(M(X),\der_{\leq n}G)$ can then by calculated as the derived mapping spectrum of $K'$-coalgebras, namely
\[ \widetilde{\Map}_{K'}(M(X),\der_{\leq n}G) \]
in the sense of \ref{def:derived-ss-hom} (but a spectrum not just a simplicial set because $K'$ is enriched over spectra).

It is now sufficient to construct a levelwise equivalence of cosimplicial spectra
\[ \Phi K^\bullet \der_{\leq n}G \to \Map_{\mathsf{\Sigma}}(M(X), (K')^\bullet \der_{\leq n}G) \]
since the corresponding totalization will then be the required equivalence
\[ P_nG \to \Map_{\der_*I}(M(X),\der_{\leq n}G) \]
by Corollary~\ref{cor:truncated}.

Since the right adjoint $\Phi$ is given by $\Map_{\mathsf{\Sigma}}(M(X),-)$, it is sufficient to show that the given condition on $\der_*G$ implies that, for each $s \geq 1$, the comonad map $\nu: K \to K'$ of Definition~\ref{def:module-comonad} induces equivalences
\[ K^s \der_{\leq n}G \to (K')^s \der_{\leq n}G. \]

We first show that if the bounded symmetric sequence $A$ has each term $A_n$ built from finitely many free $\Sigma_n$-cells, then the same is true of the symmetric sequence $K(A)$. We have a $\Sigma_n \times \Sigma_r$-equivariant equivalence
\[ \prod_{\un{n} \epi \un{r}} \Map(\der_{n_1}I \smsh \dots \smsh \der_{n_r}I, A_n) \homeq \left( \Wdge_{\un{n} \epi \un{r}} T_{n_1} \smsh \dots \smsh T_{n_r} \right) \smsh A_n. \]
The right-hand side here can be built from finitely many free $\Sigma_n \times \Sigma_r$-cells because $A_n$ can be built from finitely many free $\Sigma_n$-cells and the wedge can be built from finitely many free $\Sigma_r$-cells. Taking $\Sigma_n$ homotopy orbits and using the fact that $A$ is truncated, we see that $K(A)_r$ can be built from finitely many free $\Sigma_r$-cells as required.

Now we observe that if $A$ has the finiteness property above, then the map
\[ \nu: K(A) \to K'(A) \]
is a weak equivalence. Again, this follows from the fact that
\[ \prod_{\un{n} \epi \un{r}} \Map(\der_{n_1}I \smsh \dots \smsh \der_{n_r}I, A_n) \]
can be built from finitely many free $\Sigma_n$-cells. Therefore the norm map for this $\Sigma_n$-spectrum (which by Proposition \ref{prop:comonad-norm} is equivalent to $\nu$) is an equivalence.

It now follows by induction that the map
\[ K^s \der_{\leq n}G \to (K')^s \der_{\leq n}G \]
induced by $\nu$ is a weak equivalence for all $s \geq 1$, as required.
\end{proof}

\section{Functors with values in based spaces} \label{sec:top}

We now apply the general theory of section \ref{sec:taylor} to pointed simplicial functors $\Cfin \to \based$, where $\based$ is, as before, the category of based compactly-generated topological spaces, and $\C$ is either $\based$ or $\spectra$. As with spectrum-valued functors, our models for the derivatives are left Kan extended from those of representable functors. The key difference is that these derivatives can be endowed with the structure of a left module over the operad $\der_*I$. The left Kan extension is performed in the category of left $\der_*I$-modules.

\begin{definition} \label{def:der-top}
Let $\cat{M}$ denote the category of left $\der_*I$-modules. We fix a simplicially-enriched functor
\[ \der_*(R_\bullet): (\Cfin)^{op} \to \cat{M} \]
such that, for each $X \in \Cfin$, $\der_*(R_X)$ is a model for the derivatives of the representable functor $\Hom_{\C}(X,-): \Cfin \to \based$ together with the left module structure on those derivatives as described in \cite{arone/ching:2011}. We also ensure that $\der_*(R_X)$ is a cofibrant left $\der_*I$-module for each $X$.

For an arbitrary pointed simplicial functor $F: \Cfin \to \based$ we now define
\[ \der_*F := F(X) \otimes_{X \in \Cfin} \der_*(R_X). \]
This is an enriched coend over the simplicial category $\Cfin$ calculated in the category $\cat{M}$ of left $\der_*I$-modules, and $\otimes$ denotes the tensoring of $\cat{M}$ over based spaces. We thus obtain a simplicial functor
\[ \der_* : [\Cfin,\based] \to \cat{M}. \]
\end{definition}

\begin{proposition} \label{prop:der-top}
For cofibrant $F \in [\Cfin,\based]$, the left $\der_*I$-module $\der_*F$ is equivalent, as a left module, to the derivatives of $F$ with the module structure defined in \cite{arone/ching:2011}.
\end{proposition}
\begin{proof}
For the purposes of this proof, we write $\der^G_*F$ for the model for the derivatives of $F$ constructed in \cite{arone/ching:2011}. We focus first on the case where $F$ is a finite cell functor, that is a finite cell object with respect to the generating cofibrations in $[\Cfin,\based]$.

From Proposition \ref{prop:der-sp} we have an equivalence
\[ F(X) \smsh_{X \in \Cfin} \der_*(\Sigma^\infty R_X) \weq \der_*(\Sigma^\infty F). \]
Taking Spanier-Whitehead duals, we get
\begin{equation} \label{eq:kos} \der^*(\Sigma^\infty F) \weq \Map_{X \in \Cfin}(F(X),\der^*(\Sigma^\infty R_X)) \end{equation}
where $\der^*$ denotes the Spanier-Whitehead dual of the derivatives for a spectrum-valued functor. In \cite{arone/ching:2011} it is shown that this is in fact an equivalence of left $\mathsf{Com}$-modules where $\mathsf{Com}$ is the commutative operad of spectra. The right-hand side in (\ref{eq:kos}) involves the cotensoring of left $\mathsf{Com}$-modules over $\based$.

According to \cite{arone/ching:2011}, the derivatives of $F$ are given by the `Koszul dual' of the left $\mathsf{Com}$-module $\der^*(\Sigma^\infty F)$, that is
\[ \der^G_*(F) := \dual B(1,\mathsf{Com},\der^*(\Sigma^\infty F)). \]
Now we have a natural assembly map
\[ f: F(X) \otimes_{X \in \Cfin} \der^G_*(R_X) \to \der^G_*(F) \]
and our claim is that $f$ is a weak equivalence of left $\der_*(I)$-modules. To show this, it is sufficient by \cite[20.2]{arone/ching:2011} to show that
it induces a weak equivalence on taking Koszul duals again. The bar construction $B(1,\der_*(I),-)$ is equivalent to the left derived functor of the indecomposables of a left $\der_*(I)$-module, so preserves the tensoring and coends. Thus the Koszul dual of $f$ is equivalent to the natural map
\[ \dual B(1,\der_*(I),\der'_*(F)) \to \Map_{X \in \Cfin}(F(X),\dual B(1,\der_*(I),\der'(R_X))). \]
But this is equivalent, again by \cite[20.2]{arone/ching:2011}, to the map (\ref{eq:kos}) above, so is a weak equivalence. Thus $f$ is a weak equivalence.

The result for an arbitrary cell functor $F$ now follows by taking a filtered homotopy colimit, and for an arbitrary cofibrant functor by taking retracts.
\end{proof}

\begin{definition} \label{def:phi-top}
The right adjoint to $\der_*: [\Cfin,\based] \to \cat{M}$ is the simplicial functor $\Phi: \cat{M} \to [\Cfin,\based]$ given by
\[ \Phi A(X) := \Hom_{\der_*I}(\der_*(R_X),A) \]
where $\Hom_{\der_*I}(-,-)$ denotes the enrichment of the category $\cat{M}$ of left $\der_*I$-modules over $\based$.
\end{definition}

\begin{proposition} \label{prop:derphitop}
The adjunction $(\der_*,\Phi)$ satisfies the condition of \ref{hyp:coefficients} with $\cat{M}$ equal to the category of left $\der_*I$-modules.
\end{proposition}
\begin{proof}
Since $\der_*(R_X)$ is a cofibrant left $\der_*I$-module, the right adjoint $\Phi$ preserves fibrations and trivial fibrations. Therefore $(\der_*,\Phi)$ is a Quillen adjunction.
\end{proof}

The general theory of Section \ref{sec:taylor} then gives us the following result.

\begin{theorem} \label{thm:Ktop}
Let $K: \cat{M} \to \cat{M}$ denote the comonad $\der_* \Phi$ on the category of left $\der_*I$-modules associated to the adjunction $(\der_*,\Phi)$. Let $F: \Cfin \to \based$ be a pointed simplicial functor. Then the derivatives $\der_*F$ have the structure of a $K$-coalgebra, and the Taylor tower of $F$ can be reconstructed from this coalgebra by the cobar constructions of Corollary \ref{cor:truncated}.
\end{theorem}

An analysis of the structure of the comonad $K$, for a given source category $\mathsf{C}$, is much harder than in the case of spectrum-valued functors. The pleasant properties of symmetric sequences, that they are equivalent to the products of their individual terms, and that finite products are equivalent to finite coproducts, do not apply to left $\der_*I$-modules. We are therefore unable to get as explicit descriptions of the comonad $K$ as we did in the previous sections.

\subsection{Functors from spectra to based spaces} \label{sec:sptop}

For a finite cell spectrum $X$, the representable functor
\[ R_X: \finspec \to \based \]
is equivalent to the linear functor
\[ \Omega^\infty \dual X \smsh - \]
whose only non-trivial derivative is
\[ \der_1(R_X) \homeq \dual X. \]
The left $\der_*I$-module structure on $\der_*(R_X)$ is of course trivial. It follows that the comonad $K$ on the category of left $\der_*I$-modules controlling Taylor towers of functors $\spectra \to \based$ is given by
\[ K(A) = \der_* \left[ X \mapsto \widetilde{\Hom}_{\der_*I}(\dual X, A) \right] \]
where $\dual X$ is a left $\der_*I$-module concentrated in its first term, and $\widetilde{\Hom}$ denotes the derived mapping space for left $\der_*I$-modules. We can describe this derived mapping space as the totalization of a cosimplicial space
\[ \widetilde{\Hom}_{\der_*I}(\dual X, A) \homeq \Tot \Hom_{\mathsf{\Sigma}}(\der_*I \circ \dots \circ \der_*I \circ \dual X, A). \]
In general, this is hard to calculate, but for $2$-truncated $A$ we can use this description to understand $2$-excisive functors from spectra to based spaces.

\subsubsection*{$2$-excisive functors}

Let $A$ be a $2$-truncated left $\der_*I$-module. Thus $A$ consists of spectra $A_1$ and $A_2$, with a $\Sigma_2$-action on $A_2$, and a $\Sigma_2$-equivariant map
\[ m: \der_2I \smsh A_1 \smsh A_1 \to A_2. \]
In this case the cosimplicial space that calculates the derived mapping space
\[ \widetilde{\Hom}_{\der_*I}(\dual X, A) \]
is degenerate above degree $2$ and its totalization is equivalent to the homotopy fibre of the map
\[ \Hom_{\spectra}(\dual X,A_1) \to \Hom_{\spectra}(\der_2I \smsh \dual X \smsh \dual X, A_2)^{h\Sigma_2} \]
given by the diagonal map
\[ \Hom_{\spectra}(\dual X,A_1) \to \Hom_{\spectra}(\dual X,A_1)^{\smsh 2} \homeq \Hom_{\spectra}(\dual X \smsh \dual X, A_1 \smsh A_1) \]
followed by the canonical map that smashes source and target with $\der_2I$, and the module structure map $m$. We can write this instead as
\[ \Omega^\infty(A_1 \smsh X) \to \Omega^\infty (\Sigma A_2 \smsh X^{\smsh 2})^{h\Sigma_2} \]
where $\Sigma_2$ acts trivially on the suspension coordinate.

We know how to calculate the derivatives of these terms from our analysis of functors from spectra to spectra. Thus we get
\[ K(A)_2 \homeq A_2 \]
and $K(A)_1$ is the homotopy fibre of the composite
\[ A_1 \arrow{e,t}{\delta} \mathsf{Tate}_{\Sigma_2}(A_1 \smsh A_1) \arrow{e,t}{m} \mathsf{Tate}_{\Sigma_2}(\Sigma A_2) \]
where $\delta$ is the generalized diagonal map $\theta_{1,2}$ of (\ref{eq:diag}) for the spectrum $A_1$, and $m$ is induced by the module structure map for $A$. The left module structure on $K(A)$ is determined by the fact that the counit $K(A) \to A$ must be a map of modules.

A $K$-coalgebra structure on $A$ consists of an appropriate map of left modules $\theta: A \to KA$. The unit condition implies that the composite
\[ A_1 \to K(A)_1 \to A_1 \]
is the identity and it follows that the composite map $m\delta$ must be nullhomotopic. Conversely, a choice of nullhomotopy for this composite determines the required map $\theta$. Any such map is automatically a module map.

In conclusion, therefore, our classification of $2$-excisive functors $\finspec \to \based$ is as follows. Such a functor corresponds to spectra $A_1$, $A_2$ and
\[ m: A_1 \smsh A_1 \to \Sigma A_2 \]
as above, together with a nullhomotopy of the composite
\[ A_1 \arrow{e,t}{\delta} \mathsf{Tate}_{\Sigma_2}(A_1 \smsh A_1) \arrow{e,t}{m} \mathsf{Tate}_{\Sigma_2}(\Sigma A_2). \]

\subsubsection*{Functors of the form $\Omega^\infty F$}

Now consider $F: \finspec \to \spectra$. Then the derivatives $\der_*(\Omega^\infty F)$ are equivalent to those of $F$. It follows that any $K^{\spectra}$-coalgebra (in the sense of Section \ref{sec:specspec}) possesses a canonical $K$-coalgebra structure (in the sense of this section).

To see this, first note that the left $\der_*I$-module structure on $\der_*(\Omega^\infty F)$ is trivial. We therefore have
\[ \widetilde{\Hom}_{\der_*I}(\dual X, \der_*F) \homeq \Hom_{\mathsf{\Sigma}}(B(1,\der_*I,\dual X),\der_*F) \homeq \prod_{n} \Omega^\infty (\der_nF \smsh X^{\smsh n})^{h\Sigma_n}. \]
It follows that
\[ K(\der_*F) \homeq K^{\spectra}(\der_*F) \]
and that $K(\der_*F)$ also has a trivial left $\der_*I$-module structure. The $K$-coalgebra structure map is then given precisely by the $K^{\spectra}$-coalgebra structure map associated to $F$.

\subsection{Functors from based spaces to based spaces} \label{sec:toptop}

For a finite cell complex $X$, the representable functor $\Hom_{\based}(X,-)$ has derivatives
\[ \Map(X,\der_*I) \]
with left $\der_*I$-module structure given by the diagonal of $X$ together with the operad structure on $\der_*I$. (This is the cotensoring of the left $\der_*I$-module $\der_*I$ by the space $X$.)

The comonad $K$ on the category of left $\der_*I$-modules that controls Taylor towers of functors $\finbased \to \based$ is therefore given by
\[ K(A) = \der_* \left[ X \mapsto \widetilde{\Hom}_{\der_*I}(\Map(X,\der_*I),A) \right]. \]

\subsubsection*{$2$-excisive functors}

Let $A$ be a $2$-truncated left $\der_*I$-module with structure map
\[ m: A_1 \smsh A_1 \to \Sigma A_2. \]
Then the cosimplicial object that calculates
\[ \widetilde{\Hom}_{\der_*I}(\Map(X,\der_*I),A) \]
is again degenerate above degree $2$. The totalization can be identified with the homotopy pullback
\[ \begin{diagram}
   \node[2]{\Omega^\infty(\Sigma A_2 \smsh X)^{h\Sigma_2}} \arrow{s} \\
   \node{\Omega^\infty(A_1 \smsh X)} \arrow{e} \node{\Omega^\infty(\Sigma A_2 \smsh X^{\smsh 2})^{h\Sigma_2}}
\end{diagram} \]
Here the vertical map is the diagonal on the space $X$, and the horizontal map is the same as that described in Section \ref{sec:sptop}. We deduce that
\[ K(A)_2 \homeq A_2 \]
and that $K(A)_1$ is equivalent to the homotopy pullback of the diagram
\[ \begin{diagram}
  \node[2]{\Sigma A_2^{h\Sigma_2}} \arrow{s} \\
  \node{A_1} \arrow{e} \node{\mathsf{Tate}_{\Sigma_2}(\Sigma A_2)}
\end{diagram} \]
where the vertical map is the canonical map from the homotopy fixed points to the Tate construction, and the bottom map is as in Section \ref{sec:sptop}.

A $K$-coalgebra structure consists of an appropriate map $A \to K(A)$. This amounts to a map $m': A_1 \to \Sigma A_2^{h\Sigma_2}$ together with a homotopy between the two composites of
\[ \begin{diagram}
  \node{A_1} \arrow{s,=} \arrow{e,t}{m'} \node{\Sigma A_2^{h\Sigma_2}} \arrow{s} \\
  \node{A_1} \arrow{e} \node{\mathsf{Tate}_{\Sigma_2}(\Sigma A_2)}
\end{diagram} \]
Note that the map $m'$ can be interpreted as a right $\der_*I$-module structure on the symmetric sequence $A$ (thus making it into a $\der_*I$-bimodule). This diagram thus gives a certain additional compatibility between the right and left module structures.

Our classification of $2$-excisive functors $\finbased \to \based$ is therefore as follows. Such a functor corresponds to a symmetric sequence $A_1,A_2$ together with $\Sigma_2$-equivariant maps
\[ m: A_1 \smsh A_1 \to \Sigma A_2, \quad m': A_1 \to \Sigma A_2 \]
and a homotopy between the two composites in the diagram
\begin{equation} \label{eq:diag-tate} \begin{diagram}
  \node{A_1} \arrow{s,l}{\delta} \arrow{e,t}{m'} \node{\Sigma A_2^{h\Sigma_2}} \arrow{s} \\
  \node{\mathsf{Tate}_{\Sigma_2}(A_1 \smsh A_1)} \arrow{e,t}{m} \node{\mathsf{Tate}_{\Sigma_2}(\Sigma A_2).}
\end{diagram} \end{equation}

\subsubsection*{Functors of the form $\Omega^\infty F$}

As in Section \ref{sec:sptop}, we can easily understand the $K$-coalgebra structure on the derivatives of a functor of the form $\Omega^\infty F$ for $F: \based \to \spectra$. For such a functor, the left module structure on $\der_*(\Omega^\infty F)$ is trivial and we have
\[ K(\der_*(\Omega^\infty F)) \homeq K^{\based}(\der_*F) \]
where $K^{\based}$ is the comonad on symmetric sequences associated to functors $\based \to \spectra$. The $K$-coalgebra structure map for $\Omega^\infty F$ can then be identified with the $K^{\based}$-coalgebra structure map for $F$.

\subsubsection*{Functors of the form $F \Sigma^\infty$}

Finally, consider a $2$-excisive functor $F: \spectra \to \based$. Then $F \Sigma^\infty$ is a $2$-excisive functor $\based \to \based$ with the same derivatives as $F$. We saw in Section \ref{sec:sptop} that $F$ is classified by a left $\der_*I$-module $A$ with terms $A_1,A_2$ together with a nullhomotopy of the composite
\[ A_1 \to \mathsf{Tate}_{\Sigma_2}(\Sigma A_2). \]
The functor $F \Sigma^\infty$ is classified by this same information with the map
\[ m': A_1 \to \Sigma A_2 \]
taken to be trivial (i.e. the right $\der_*I$-module structure is trivial). The nullhomotopy then gives the required homotopy between the two composites of the diagram (\ref{eq:diag-tate}).

\bibliographystyle{amsplain}
\bibliography{mcching}

\providecommand{\bysame}{\leavevmode\hbox to3em{\hrulefill}\thinspace}
\providecommand{\MR}{\relax\ifhmode\unskip\space\fi MR }
\providecommand{\MRhref}[2]{%
  \href{http://www.ams.org/mathscinet-getitem?mr=#1}{#2}
}
\providecommand{\href}[2]{#2}
\begin{thebibliography}{10}

\bibitem{arone:1999}
Greg Arone, \emph{A generalization of {S}naith-type filtration}, Trans. Amer.
  Math. Soc. \textbf{351} (1999), no.~3, 1123--1150. \MR{MR1638238 (99i:55011)}

\bibitem{arone/ching:2011}
Greg Arone and Michael Ching, \emph{Operads and chain rules for the calculus of
  functors}, Ast\'erisque (2011), no.~338, vi+158. \MR{2840569 (2012i:55012)}

\bibitem{arone/dwyer/lesh:2008}
Gregory~Z. Arone, William~G. Dwyer, and Kathryn Lesh, \emph{Loop structures in
  {T}aylor towers}, Algebr. Geom. Topol. \textbf{8} (2008), no.~1, 173--210.
  \MR{2377281 (2008m:55025)}

\bibitem{batanin:1998}
Mikhail~A. Batanin, \emph{Homotopy coherent category theory and
  {$A_\infty$}-structures in monoidal categories}, J. Pure Appl. Algebra
  \textbf{123} (1998), no.~1-3, 67--103. \MR{1492896 (99c:18007)}

\bibitem{chaoha:2004}
Phichet Chaoha, \emph{Splitting criteria for homotopy functors of spectra},
  Trans. Amer. Math. Soc. \textbf{356} (2004), no.~4, 1271--1280 (electronic).
  \MR{MR2034308 (2005a:55008)}

\bibitem{ching:2005}
Michael Ching, \emph{Bar constructions for topological operads and the
  {G}oodwillie derivatives of the identity}, Geom. Topol. \textbf{9} (2005),
  833--933 (electronic). \MR{MR2140994}

\bibitem{ching:2010}
\bysame, \emph{{A chain rule for Goodwillie derivatives of functors from
  spectra to spectra}}, Trans. Amer. Math. Soc. \textbf{362} (2010), no.~1,
  399--426.

\bibitem{ching/riehl:2014}
Michael Ching and Emily Riehl, \emph{Coalgebraic models for combinatorial model
  categories}, 2014.

\bibitem{elmendorf/kriz/mandell/may:1997}
A.~D. Elmendorf, I.~Kriz, M.~A. Mandell, and J.~P. May, \emph{Rings, modules,
  and algebras in stable homotopy theory}, Mathematical Surveys and Monographs,
  vol.~47, American Mathematical Society, Providence, RI, 1997, With an
  appendix by M. Cole. \MR{97h:55006}

\bibitem{fresse:2000}
Benoit Fresse, \emph{On the homotopy of simplicial algebras over an operad},
  Trans. Amer. Math. Soc. \textbf{352} (2000), no.~9, 4113--4141. \MR{MR1665330
  (2000m:18015)}

\bibitem{goodwillie:1990}
Thomas~G. Goodwillie, \emph{Calculus. {I}. {T}he first derivative of
  pseudoisotopy theory}, $K$-Theory \textbf{4} (1990), no.~1, 1--27.
  \MR{92m:57027}

\bibitem{goodwillie:1991}
\bysame, \emph{Calculus. {II}. {A}nalytic functors}, $K$-Theory \textbf{5}
  (1991/92), no.~4, 295--332. \MR{93i:55015}

\bibitem{goodwillie:2003}
\bysame, \emph{Calculus. {III}. {T}aylor series}, Geom. Topol. \textbf{7}
  (2003), 645--711 (electronic). \MR{2 026 544}

\bibitem{greenlees/may:1995}
J.~P.~C. Greenlees and J.~P. May, \emph{Generalized {T}ate cohomology}, Mem.
  Amer. Math. Soc. \textbf{113} (1995), no.~543, viii+178. \MR{MR1230773
  (96e:55006)}

\bibitem{hess:2010}
Kathryn Hess, \emph{A general framework for homotopic descent and codescent},
  2010.

\bibitem{hess/shipley:2014}
Kathryn Hess and Brooke Shipley, \emph{The homotopy theory of coalgebras over a
  comonad}, Proceedings of the London Mathematical Society \textbf{108} (2014),
  no.~2, 484--516.

\bibitem{kuhn:2004}
Nicholas~J. Kuhn, \emph{Tate cohomology and periodic localization of polynomial
  functors}, Invent. Math. \textbf{157} (2004), no.~2, 345--370. \MR{MR2076926
  (2005f:55008)}

\bibitem{lurie:2014}
Jacob Lurie, \emph{Higher algebra}, 2014.

\bibitem{mccarthy:2001}
Randy McCarthy, \emph{Dual calculus for functors to spectra}, Homotopy methods
  in algebraic topology (Boulder, CO, 1999), Contemp. Math., vol. 271, Amer.
  Math. Soc., Providence, RI, 2001, pp.~183--215. \MR{MR1831354 (2002c:18009)}

\bibitem{mcclure/smith:2002}
James~E. McClure and Jeffrey~H. Smith, \emph{A solution of {D}eligne's
  {H}ochschild cohomology conjecture}, Recent progress in homotopy theory
  ({B}altimore, {MD}, 2000), Contemp. Math., vol. 293, Amer. Math. Soc.,
  Providence, RI, 2002, pp.~153--193. \MR{1890736 (2003f:55013)}

\bibitem{mcclure/smith:2004}
\bysame, \emph{Cosimplicial objects and little {$n$}-cubes. {I}}, Amer. J.
  Math. \textbf{126} (2004), no.~5, 1109--1153. \MR{MR2089084 (2005g:55011)}

\bibitem{nikolaus:2011}
Thomas Nikolaus, \emph{Algebraic models for higher categories}, Indag. Math.
  (N.S.) \textbf{21} (2011), no.~1-2, 52--75. \MR{2832482 (2012g:55029)}

\bibitem{oman:2010}
Peter~J. Oman, \emph{{Models for the Maclaurin tower of a simplicial functor
  via a derived Yoneda embedding}}, Journal of Pure and Applied Algebra
  \textbf{214} (2010), no.~12, 2148 -- 2158.

\bibitem{quillen:1969}
Daniel Quillen, \emph{Rational homotopy theory}, Ann. of Math. (2) \textbf{90}
  (1969), 205--295. \MR{MR0258031 (41 \#2678)}

\bibitem{riehl/verity:2013}
Emily {Riehl} and Dominic {Verity}, \emph{{Homotopy coherent adjunctions and
  the formal theory of monads}}, ArXiv e-prints (2013), 1--86.

\bibitem{rosicky:2009}
J.~Rosicky, \emph{Are all cofibrantly generated model categories
  combinatorial?}, Cah. Topol. G\'eom. Diff\'er. Cat\'eg. \textbf{50} (2009),
  no.~3, 233--238. \MR{2553541 (2010k:18005)}

\end{thebibliography}

\end{document}